\documentclass[11pt]{amsart}

\usepackage[svgnames,x11names]{xcolor}
\usepackage{amssymb,amsfonts}
\usepackage{mdwlist}

\usepackage{amssymb,textcomp}

\usepackage{enumerate}

\usepackage[utf8]{inputenc}
\usepackage[T1]{fontenc}

\usepackage[mathscr]{eucal}

 \usepackage[pagebackref,colorlinks,urlcolor={DarkOliveGreen4},
 linkcolor={DarkSlateBlue},
 citecolor={Chocolate4}]{hyperref}

\usepackage[capitalize]{cleveref}

\usepackage[a4paper, twoside=false, vmargin={2cm,3cm}, includehead]{geometry}

\usepackage{comment}
\newtheorem{lemma}{Lemma}[section]
\newtheorem{theorem}[lemma]{Theorem}
\newtheorem*{theorem*}{Theorem}
\newtheorem{claim}[lemma]{Claim}

\newtheorem{proposition}[lemma]{Proposition}
\newtheorem*{proposition*}{Proposition}

\newtheorem{conjecture}{Conjecture}

\newtheorem*{problem*}{Problem}

\theoremstyle{definition}
\newtheorem*{claim*}{Claim}

\newtheorem{definition}[lemma]{Definition}

\newtheorem*{remark}{Remark}
\newtheorem*{remarks}{Remarks}

\DeclareMathOperator*{\E}{\mathbb{E}}

\newcommand{\C}{{\mathbb C}}
\newcommand{\D}{{\mathbb D}}

\newcommand{\N}{{\mathbb N}}

\newcommand{\Q}{{\mathbb Q}}
\newcommand{\R}{{\mathbb R}}
\renewcommand{\S}{\mathbb{S}}
\newcommand{\T}{{\mathbb T}}
\newcommand{\Z}{{\mathbb Z}}

\newcommand{\CA}{{\mathcal A}}

\newcommand{\CC}{{\mathcal C}}
\newcommand{\CD}{{\mathcal D}}
\newcommand{\CE}{{\mathcal E}}
\newcommand{\CF}{{\mathcal F}}
\newcommand{\CG}{{\mathcal G}}
\newcommand{\CH}{{\mathcal H}}
\newcommand{\CI}{{\mathcal I}}

\newcommand{\CX}{{\mathcal X}}
\newcommand{\CY}{{\mathcal Y}}

\newcommand{\CZ}{{\mathcal Z}}

\newcommand{\veps}{\varepsilon}

\newcommand{\eps}{\epsilon}
\newcommand{\ueps}{{{\epsilon}}}

\renewcommand{\angle}[1]{\mathopen{}\left\langle #1\mathclose{}\right\rangle}

\newcommand{\f}{\textbf{f}}

\newcommand{\norm}[1]{\left\Vert #1\right\Vert}

\newcommand{\Bignorm}[1]{\Big\Vert #1\Big\Vert}
\newcommand{\nnorm}[1]{\lvert\!|\!| #1|\!|\!\rvert}

\newcommand{\inv}{^{-1}}

\DeclareMathOperator{\unif}{unif}
\DeclareMathOperator{\str}{str}
\DeclareMathOperator{\Str}{Str}
\DeclareMathOperator{\sml}{sml}

\DeclareMathOperator{\Span}{Span}

\newcommand{\abs}[1]{\mathopen{}\left| #1\mathclose{}\right|}

\newcommand{\Bigabs}[1]{\Bigl| #1 \Bigr|}
\newcommand{\brac}[1]{\mathopen{}\left( #1 \mathclose{}\right)}
\newcommand{\bigbrac}[1]{\bigl( #1 \bigr)}
\newcommand{\Bigbrac}[1]{\Bigl( #1 \Bigr)}

\newcommand{\rem}[1]{\left \{ #1 \right \}}

\begin{document}

	\title[Ergodic averages for sparse corners]{Ergodic averages for sparse corners}

\author{Nikos Frantzikinakis and Borys Kuca}
\address[Nikos Frantzikinakis]{University of Crete, Department of mathematics and applied mathematics, Voutes University Campus, Heraklion 70013, Greece} \email{frantzikinakis@gmail.com}

\address[Borys Kuca]{Jagiellonian University, Faculty of Mathematics and Computer Science, 30-348 Krak\'ow, Poland}
\email{borys.kuca@uj.edu.pl}

\begin{abstract}
We develop a framework for the study of the limiting behavior of multiple
ergodic averages with commuting transformations when all iterates are given by
the same sparse sequence; this enables us to partially resolve several longstanding problems.
First, we address a special case of the joint intersectivity question of
Bergelson, Leibman, and Lesigne by giving necessary and sufficient conditions under
which the multidimensional polynomial Szemer\'edi theorem holds for length-three patterns. Second, we show that
for two commuting transformations, the Furstenberg averages remain unchanged when the iterates are taken along
sparse sequences such as $[n^c]$ for a positive noninteger $c$, advancing a conjecture
of the first author. Third, we extend a result of Chu  on popular common differences in linear corners to polynomial and Hardy corners. Lastly, we  answer open problems of Le, Moreira, and Richter concerning decomposition results for double correlation sequences.
Our toolbox includes recent degree lowering and seminorm smoothing techniques,
the machinery of magic extensions of Host, and novel structured extensions
motivated by works of Tao and Leng. Combined, these techniques reduce the analysis to settings where the
Host-Kra theory of characteristic factors and equidistribution on nilmanifolds
yield a family of striking identities from which our main results follow.
\end{abstract}

\thanks{The  first author was supported  by the  Research Grant ELIDEK HFRI-NextGenerationEU-15689. The second author was supported by the National Science Center (NCN Poland) Sonata grant No. 2024/55/D/ST1/00468 and would also like to acknowledge the Foundation of Polish Science (FNP) Start stipend.}

\subjclass[2020]{Primary: 37A44; Secondary:    28D05, 05D10, 11B30.}

\keywords{Ergodic averages, intersective polynomials, fractional  powers,   recurrence, popular differences, box seminorms, Host-Kra factors.}

\date{}

\maketitle

\setcounter{tocdepth}{2}
\tableofcontents

\section{Introduction}

Determining the limiting behavior of multiple ergodic averages is a central theme in modern ergodic theory, with deep applications to combinatorics and number theory.
More precisely, we seek to identify those sequences $a_1,\ldots, a_\ell\colon \N\to \Z$, such that
for every \textit{system} $(X,\CX,\mu,T_1,\ldots, T_\ell)$, i.e., commuting invertible measure-preserving
transformations $T_1, \ldots, T_\ell$ acting on a standard probability space $(X, \mathcal{X}, \mu)$,  the averages
\begin{equation}\label{E:general}
	\frac{1}{N} \sum_{n=1}^N f_1(T_1^{a_1(n)}x) \cdots f_\ell(T_\ell^{a_\ell(n)}x)
\end{equation}
converge (in $L^2(\mu)$, weakly, or pointwise $\mu$-a.e.) for all $f_1, \ldots, f_\ell \in
L^\infty(\mu)$. Whenever possible, we also want to identify the limit: this is both a natural goal in itself and a gateway to new extensions of the Szemer\'edi theorem~\cite{Sz75}.
When $T_1,\ldots, T_\ell$ are powers of a single transformation,
the question has been addressed for various sequences using
the Host-Kra theory of characteristic factors~\cite{HK05a} and equidistribution
results on nilmanifolds (e.g. \cite{Lei05a}). For general commuting transformations, no equally powerful
alternative theory is currently available,  and progress remains limited.
Tao~\cite{Ta08} and Walsh~\cite{Wal12} respectively proved mean convergence when all iterates
are linear and polynomial; yet their methods yield no
information about the limit. Under suitable independence assumptions,
the authors~\cite{FrKu22a} recently evaluated the limit for polynomial
sequences, significantly improving on an earlier result of Chu, Host, and the first author \cite{CFH11}.
Soon after, Donoso, Koutsogiannis, Sun, Tsinas, and the second author~\cite{DKKST24} handled more general sequences of polynomial growth while Koutsogiannis and Tsinas obtained new limiting formulas for sparse sequences involving primes \cite{KT25}. However, the situation
worsens without independence, with the most extreme case occurring when all
iterates are equal. The main goal of this article is to build foundations
that allow us to understand the limiting behavior of \eqref{E:general}
in such settings, and in doing so we cover special cases
of several open problems that have resisted progress for many years.

 The first problem is a conjecture of the first author    (see \cite[Problem 5]{Fr10}, also stated as \cite[Problem 29]{Fr16}):
\begin{conjecture}\label{Con1}
	If $c$ is a positive noninteger,  then for every system $(X,\CX,\mu,T_1,\ldots,T_\ell)$ and functions $f_0, \ldots, f_\ell \in L^\infty(\mu)$ we have
		$$
	\lim_{N\to\infty} \frac{1}{N}\sum_{n=1}^N \,\int f_0\cdot  T_1^{[n^c]}f_1\cdots T_\ell^{[n^c]}f_\ell\, d\mu=
	\lim_{N\to\infty} \frac{1}{N}\sum_{n=1}^N \, \int f_0\cdot  T_1^nf_1\cdots  T_\ell^nf_\ell\, d\mu.
	$$
\end{conjecture}
The limit on the right-hand side  is known to exist by \cite{Ta08}, so the conjecture in particular claims the existence of the limit on the left-hand side.
The conjecture has remained open for all
noninteger $c>1$ even for $\ell=2$, with the exception of
$c\in (1,23/22)$ recently handled by Daskalakis \cite{Da25}.
When all the transformations are equal, the identity follows
from \cite{Fr10} using methodology not applicable to the setting
of multiple commuting transformations.
In \cref{T:Main1} we confirm the conjecture for $\ell=2$.

The second problem is related to the multidimensional polynomial Szemer\'edi
theorem of Bergelson and Leibman~\cite{BL96}, which implies that if
$\Lambda \subseteq \mathbb{Z}^d$ has positive upper density and
$p_1,\ldots, p_\ell$ are integer polynomials with {\em zero constant term},
then for any $v_1, \ldots, v_\ell \in \mathbb{Z}^d$ there
exist $m\in\Z^d$ and $n\in\N$ such that
\begin{align}\label{E: multidim pattern}
m,\, m + p_1(n)v_1,\, \ldots,\, m + p_\ell(n)v_\ell
\end{align}
lies in $\Lambda$.

What if the polynomials $p_1,\ldots, p_\ell$ do not vanish at $0$? If a polynomial $p$ is not \textit{intersective}, i.e., some $r\in\N$ does not divide $p(n)$ for any $n\in\Z$,
then the set $\Lambda = r\mathbb{Z}$ cannot simultaneously contain the points $m,\; m + p(n)$ for any $m,n\in\Z$. By adjusting this example to subsets of $\Z^d$ and several polynomials $p_1, \ldots, p_\ell$, one quickly observes that unless all polynomials are simultaneously divisible by every integer, there exists a positive-density subset of $\Z^d$ without the pattern \eqref{E: multidim pattern}.
This motivates the following definition from \cite{BLL08}: a family $p_1, \ldots, p_\ell$ is
\emph{jointly intersective} if for every $r \in \mathbb{N}$, there
exists $n\in \N$ for which $p_1(n), \ldots, p_\ell(n)$ are all divisible
by $r$. In \cite[Proposition~6.1]{BLL08} it is shown that a family of   polynomials is jointly intersective if and only if all its members are all divisible
 by the same intersective polynomial $p$. While the most obvious examples of intersective polynomials $p$ with $p(0)\neq 0$ are those that can be factorized into linear parts (e.g. $p(n) = n^2-1$), there exist intersective polynomials without linear factors, for instance $p(n)=(n^2+1)(n^2-2)(n^2+2)$ (see \cite{Bere} for more examples).

The following conjecture of
Bergelson, Leibman, and Lesigne~\cite{BLL08} seeks to completely characterize polynomial families good for multiple recurrence.
\begin{conjecture}\label{Con2}
	Let $p_1,\ldots, p_\ell\in\Z[t]$ be jointly intersective. Then for every
	system $(X,\mathcal{X},\mu, T_1,\ldots, T_\ell)$ and all  $A \in \mathcal{X}$ with $\mu(A) > 0$,  we have
	$$
	\liminf_{N\to\infty}\frac{1}{N}\sum_{n=1}^N \mu\big(A \cap T_1^{-p_1(n)}A \cap \cdots \cap T_\ell^{-p_\ell(n)}A\big) > 0.
	$$
\end{conjecture}
The conjecture has been open even for $\ell=2$ and $p_1=p_2$.
For general $\ell\in \N$, a positive answer is known only
when $T_1, \ldots, T_\ell$ are powers of a single transformation~(\cite[Theorem~F]{Fr08} for $\ell=2,3$ and \cite{BLL08} for general $\ell$),
and the argument again uses the Host-Kra theory of
characteristic factors~\cite{HK05a}, which is inapplicable to the setting of several commuting transformations. Neither is Conjecture \ref{Con2} amenable to the Bergelson-Leibman proof of the polynomial Szemer\'edi theorem
for polynomials with zero constant term \cite{BL96}, as that argument depends on a partition regularity analogue that remains unresolved for intersective polynomials. In \cref{T:Main2'} we confirm the conjecture for $\ell=2$, the main novelty lies in resolving the diagonal case $p_1=p_2$, whereas prior approaches addressed only the case of rationally independent polynomials.

The last problem concerns the structure of  multiple correlation sequences. 
Given a system $(X,\CX,\mu,T_1,\ldots, T_\ell)$ and  functions $f_0,f_1, \ldots, f_\ell\in L^\infty(\mu)$, a topic that has attracted considerable   attention  in ergodic theory	
in recent years \cite{BHK05, DFKS24,  F21, Fr15, FH18, FrKu22a,  Kou18, KLMR21,  Le18, LME21,  
	Lei15,	Leng25, TT19}
is to prove ``nil plus null'' decomposition results for the multiple correlation  sequences 
\begin{equation}\label{E:multicorr}
	C_\ell(n):=\int f_0\cdot T_1^nf_1\cdots T_\ell^n f_\ell\, d\mu, \quad \textrm{where}\quad n\in \N.  
\end{equation}
When $T_1,\ldots, T_\ell$ are powers of a single transformation, Bergelson, Host, and Kra \cite{BHK05} showed that there exists a \textit{nilsequence} $(\psi(n))_n$ (see definition in \cref{SS:Nilstuff}) such that the difference $(C_\ell(n)-\psi(n))_n$ is a \textit{null-sequence}, i.e., satisfies 
	$$
	\lim_{N\to\infty}\frac{1}{N}\sum_{n=1}^N |C_\ell(n)-\psi(n)|=0.
$$
This result was recently extended to general systems of commuting transformations by Leng~\cite{Leng25}, and also to the setting where $n$ is replaced by the $n$-th prime $p_n$. Moreover, in the setting of~\cite{BHK05}, which concerns powers of a single transformation,  Le~\cite{Le18} showed that the sequence $(C_\ell(n)-\psi(n))_n$ is null along several other sparse subsets of $\N$, such as integer polynomial sequences and sequences of the form $([n^c])_n$, where $c\in\R_+$ is non-integer.
The next conjecture aims to generalize all these results; special cases of the conjecture have appeared, for instance, in \cite{Fr15, Fr16, KLMR21, Le18, LME21}.
\begin{conjecture}\label{Con3}
	If $(C_\ell(n))_{n}$ is as in \eqref{E:multicorr}, then we have a decomposition 
	$$
	C_\ell(n)=\psi(n)+ \eta(n), 
	$$
	where
 $(\psi(n))_n$ is a nilsequence  and 	
	 $
	 \lim\limits_{N\to\infty}\frac{1}{N}\sum\limits_{n=1}^N|\eta(a(n))|=0,
	 $ whenever $a\in \Z[t]$ is a polynomial, or $a(n)=[n^c]$ for some $c\in \R_+$, or any of these sequences taken along the $n$-th prime $p_n$. 
\end{conjecture}
Optimally, we would like the nilsequence  $(\psi(n))_n$ to be chosen independently of $(a(n))_n$. In \cref{T:Main4} we verify that the conjecture holds for $\ell=2$ in a  broader setting.

\subsection{Limiting formulas}
Our first main result establishes Conjecture~\ref{Con1} for $\ell=2$
in a broader setting that, for example, also accommodates sequences
like $[n\log{n}]$, $[n^2\sqrt{2}+n]$, or $[n\alpha+(\log{n})^2]$ where
$\alpha\in \R$. All the definitions regarding Hardy fields can be found in Section \ref{SS: Hardy}.
\begin{theorem}[Limiting formula for Hardy corners]\label{T:Main1}
	Let $a\colon \R_+\to \R$ be a Hardy-field function of polynomial growth that satisfies
    \begin{align}\label{E: log away}
      \lim_{t\to\infty} \frac{|a(t)-cp(t)|}{\log t}= \infty\quad \textrm{for every}\;\; c\in\R \;\;\textrm{and}\;\;   p\in \Z[t].
    \end{align}
	Then for every system $(X,\CX,\mu,T_1, T_2)$ and functions $f_0,f_1,f_2 \in L^\infty(\mu)$, we have the identity
 	\begin{align}\label{E: main identity}
 	    \lim_{N\to\infty} \frac{1}{N}\sum_{n=1}^N \, \int f_0\cdot T_1^{[a(n)]}f_1\cdot T_2^{[a(n)]}f_2\, d\mu =
	\lim_{N\to\infty} \frac{1}{N}\sum_{n=1}^N\,\int f_0\cdot   T_1^n f_1\cdot  T_2^n f_2\, d\mu;
 	\end{align}
    in particular, both limits above exist.
\end{theorem}
\begin{remarks}
The existence of the limit on the right-hand side follows
from work of  Conze and Lesigne~\cite{CL84}.  The difficult part is to show
that the limit on the left-hand side exists and
that the two limits coincide.

{Condition \eqref{E: log away} is motivated by equidistribution results on the circle due to Boshernitzan~\cite{Bos94}. When this condition fails, the identity \eqref{E: main identity} generally does not hold; this can be seen e.g. when $a(t)$ equals  $2t$, $t^2$, or $\log t$ by considering rotations on the circle.}

We expect Theorem \ref{T:Main1} to hold for an arbitrary number of commuting transformations; same with Theorems \ref{T:Main2} and \ref{T:Main2'} below. Likewise, we expect \eqref{E: main identity} (and similarly \eqref{E: main identity 2} below) to hold at the level of $L^2(\mu)$ rather than weak limits. We shall explain in Section \ref{SS: issues} why we are unable to accomplish this at the moment.
\end{remarks}

Using a recent comparison result of Koutsogiannis and Tsinas~\cite[Theorem~1.2]{KT25} (which builds on number-theoretic input from \cite{MSTT22}),
we also obtain analogous statements with $[a(p_n)]$ in place
of $[a(n)]$, where $p_n$ denotes the $n$-th prime; for instance we get the identity
$$
\lim_{N\to\infty} \frac{1}{N}\sum_{n=1}^N \, \int f_0\cdot T_1^{[p_n^c]}f_1\cdot T_2^{[p_n^c]}f_2\, d\mu =
\lim_{N\to\infty} \frac{1}{N}\sum_{n=1}^N\,\int f_0\cdot   T_1^n f_1\cdot  T_2^n f_2\, d\mu
$$
for all $c\in \R_+\setminus \Z$.

Our second main result is a version of Theorem \ref{T:Main1} for integer polynomials in the form needed to address Conjecture~\ref{Con2} for $\ell=2$.
\begin{theorem}[Limiting formula for polynomial corners]\label{T:Main2}
	Let 	$p\in \Z[t]$ be a nonzero intersective polynomial and  for every $k\in \N$, take $n_k\in \N_0$ for which $k!\mid p(n_k)$.
	Then for every system $(X,\CX,\mu,T_1,T_2)$ and functions $f_0,f_1, f_2 \in L^\infty(\mu)$, we have the identity
 	\begin{multline}\label{E: main identity 2}
 	 \lim_{k\to \infty}\lim_{N\to\infty}\frac{1}{N}\sum_{n=1}^N\,  \int f_0\cdot T_1^{p(k!n+n_k)}f_1\cdot T_2^{p(k!n+n_k)}f_2\; d\mu\\
     =\lim_{k\to \infty}\lim_{N\to\infty}\frac{1}{N}\sum_{n=1}^N\,  \int f_0\cdot T_1^{k!n}f_1\cdot T_2^{k!n}f_2\, d\mu;
 	\end{multline}
 	     in particular, both limits above exist. Moreover, we can replace the  Ces\`aro average over $n\in \N$ with an average along an arbitrary F\o lner  sequence in $\N$.\footnote{A {\em F\o lner sequence in $\N$} is a sequence $(I_N)_{N}$ of finite subsets of $\N$ that is  asymptotically invariant under translation, in the sense that $\lim\limits_{N\to\infty}\big|I_N \triangle  (x+ I_N)\big|/|I_N|=0$ for all  $x\in \N$.}
\end{theorem}
\begin{remarks}
The existence of the limit on the right-hand side follows
from Lemma~\ref{L:convergence} below, whose proof uses the machinery of
magic extensions. The difficult part, again, is to show
that the limit on the left-hand side exists and
that the two limits coincide.

If $p(0)=0$, then we can simply take $n_k=0$ for every
$k\in \N$. Also, in place of $k!$ we can use
any other highly divisible sequence, i.e., a sequence $(c_k)_k$ with the property that for every $r\in \N$, we have $r\mid c_k$
for all sufficiently large $k\in \N$.
Our argument also gives that for every nonconstant polynomial $p\in \Z[t]$ and all $j_k\in \Z$,  we have the identity
	\begin{multline}\label{E:pkj}
	\lim_{k\to \infty}\Big|\lim_{N\to\infty}\frac{1}{N}\sum_{n=1}^N\,  \int f_0\cdot T_1^{p(k!n+j_k)}f_1\cdot T_2^{p(k!n+j_k)}f_2\; d\mu-\\
	\lim_{N\to\infty}\frac{1}{N}\sum_{n=1}^N\,  \int f_0\cdot T_1^{k!n+p(j_k)}f_1\cdot T_2^{k!n+p(j_k)}f_2\, d\mu\Big|=0.
\end{multline}
This, together with \cref{P: linear} and the scaling property~\eqref{I:scaling} of the box seminorms from \cref{SS:boxseminorms}, easily yields the following implication
\begin{equation}\label{E:cfpolies}
\nnorm{f_2}_{T_2 T_1\inv, T_2}=0\implies \lim_{N\to\infty}\frac{1}{N}\sum\limits_{n=1}^N\,  \int f_0\cdot T_1^{p(n)}f_1\cdot T_2^{p(n)}f_2\; d\mu=0.
\end{equation}

A quantitative, finitary counterpart of \eqref{E:cfpolies}, under additional assumptions on the polynomial $p$, was recently established by Kravitz, Leng, and the second author \cite[Theorem~1.2]{KKL24b}; related results were obtained earlier in \cite{PSS23} and later in \cite{AS25}. Our overall proof strategy for Theorems \ref{T:Main1} and \ref{T:Main2} is  partly inspired by \cite{KKL24b}; however, the difficulties in the combinatorial and ergodic setting tend to lie in different places, with various steps being easy in one setting but difficult in the other. See \cref{SS:proofsketch} for a summary of differences.
\end{remarks}

\subsection{Szemer\'edi-type theorems}
We deduce several applications from the identities in Theorems \ref{T:Main1} and \ref{T:Main2}, starting with a uniform-limit version of Conjecture~\ref{Con2} for  $\ell=2$.
\begin{theorem}[Existence of length-three jointly intersective polynomial patterns]\label{T:Main2'}
	Let $p_1$, $p_2\in \Z[t]$ be jointly intersective polynomials. Then for every system $(X,\mathcal{X},\mu, T_1,T_2)$ and  sets $A \in \mathcal{X}$ with $\mu(A) > 0$,  we have
$$
\lim_{N-M\to\infty}\frac{1}{N-M}\sum_{M\leq n<N} \mu\big(A \cap T_1^{-p_1(n)}A \cap T_2^{-p_2(n)}A\big) > 0
$$
(the limit exists by \cite{Wal12,Zo15a}).
Moreover, for any choice of direction vectors $v_1,v_2\in \Z^d$, every subset $\Lambda\subseteq\Z^d$ of positive upper Banach density contains patterns of the form
\begin{align*}
    m,\; m+v_1 p_1(n),\; m+v_2 p_2(n)
\end{align*}
for some $m\in\Z^d$ and $n\in\N$.
\end{theorem}
\begin{proof}[Proof of \cref{T:Main2'} assuming \cref{T:Main2}]
The combinatorial statement follows from the ergodic statement via the Furstenberg correspondence principle \cite{Fu77,Fu81}, so we just prove the first part.
We consider two cases. If the polynomials $1,p_1,p_2$ are
linearly independent, then the result was shown in \cite[Corollary~2.11]{FrKu22a}.
If they are linearly dependent, then the joint intersectivity of
the polynomials $p_1,p_2$ implies they have the form
$p_1=k_1 p$, $p_2=k_2 p$, for some $k_1,k_2\in \Z$ and an
intersective polynomial $p$. Thus, letting $T_1':=T_1^{k_1}$ and
$T_2':=T_2^{k_2}$, we are reduced to the case where $p_1=p_2=p$.
	
We will cover this case using \cref{T:Main2}. Let $A\in \CX$ with $\mu(A)>0$. It is known~\cite{BeMc96, FuK79}
that there exists $c=c(\mu(A))$ (it is important that $c$
does not depend on $k$) such that for every $k\in \N$, we have
$$
\liminf_{N-M\to\infty}\frac{1}{N-M}\sum_{M\leq n<N}  \mu(A\cap T_1^{-k!n}A\cap  T_2^{-k!n}A)\geq c.
$$
By Theorem~\ref{T:Main2} there exist $k,n_k\in \N$ such that
\begin{multline*}
	\limsup_{N-M\to\infty}\Big|\frac{1}{N-M}\sum_{M\leq n<N}
	 \mu(A\cap T_1^{-p(k!n+n_k)}A \cap T_2^{-p(k!n+n_k)}A) -\\ \frac{1}{N-M}\sum_{M\leq n<N}   \mu(A\cap T_1^{-k!n}A \cap T_2^{-k!n}A)\Big|\leq c/2.
\end{multline*}
Combining the above, we deduce that
$$
\liminf_{N-M\to\infty}\frac{1}{N-M}\sum_{M\leq n<N}  \mu(A\cap T_1^{-p(k!n+n_k)}A\cap  T_2^{-p(k!n+n_k)}A)\geq c/2,
$$
completing the proof.
\end{proof}

\subsection{Popular common differences}

Our next application concerns lower bounds for multiple recurrence, which in the combinatorial setting translate into the existence of popular common differences for Hardy and polynomial corners. By combining a result of Chu~\cite{Chu11} with Theorems \ref{T:Main1} and \ref{T:Main2} respectively, we obtain the following two results. In the ergodic setting, we study quantitative recurrence  for  actions $(X,\CX,\mu,T_1,T_2)$ that are ergodic, i.e., $T_1f=T_2f=f$ implies that $f$ is constant $\mu$-a.e..  Below, $d^*(\Lambda)$ denotes the \textit{upper Banach density}\footnote{The \textit{upper Banach density} of $\Lambda\subseteq\Z^d$ is defined to be $d^*(\Lambda) = \sup\limits_{(I_N)_N}\limsup\limits_{N\to\infty}\frac{|\Lambda\cap I_N|}{|I_N|}$, where the supremum is taken over all F\o lner sequences $(I_N)_N$ on $\Z^d$.} of $\Lambda$.
\begin{theorem}[Popular common differences for Hardy corners]\label{T:Main3}
    Let $a\colon \R_+\to \R$ be a Hardy-field function of polynomial growth that satisfies \eqref{E: log away}. Then for every ergodic system $(X,\mathcal{X},\mu, T_1,T_2)$, set  $A\in \CX$, and $\varepsilon>0$,
	the set
	\begin{equation*}
	\big\{n\in \N\colon  \mu(A\cap T_1^{-[a(n)]}A\cap T_2^{-[a(n)]}A)\geq (\mu(A))^4-\varepsilon\big\}
	\end{equation*}
	has positive lower density. So does the set
	\begin{equation*}
	\big\{n\in \N\colon   d^*(\Lambda\cap (\Lambda-v_1[a(n)])\cap (\Lambda-v_2[a(n)])\geq ({d}^*(\Lambda))^4-\varepsilon\big\}
	\end{equation*}
    for any directions $v_1, v_2\in\Z^d$, a set $\Lambda\subseteq\Z^d$, and $\veps>0$.
\end{theorem}
By mimicking the proof of Theorem \ref{T:Main2'}, we could obtain the weaker statement that
\begin{align}\label{E: multiple recurrence Hardy}
    \lim_{N\to\infty}\frac{1}{N}\sum_{n=1}^N \mu\big(A \cap T_1^{-[a(n)]}A \cap T_2^{-[a(n)]}A\big) > 0
\end{align}
for every system $(X, \CX, \mu, T_1, T_2)$ and every positive-measure set $A\in \CX$, together with the corollary on the existence of patterns
\begin{align*}
    m,\; m+v_1[a(n)],\; m+v_2[a(n)]
\end{align*}
in subsets $\Lambda\subseteq\Z^d$ of positive upper density for any choice of direction vectors $v_1, v_2\in\Z^d$. Even these weaker statements are new for $a(n) = n^c$ with a nonrational $c\geq 23/22.$

For Hardy sequences staying away from polynomials, we cannot in general replace the Ces\`aro averages in  \eqref{E: multiple recurrence Hardy} by averages over arbitrary F{\o}lner sequences in $\Z$; nor can the sets of $n$'s obtained in Theorem \ref{T:Main3} be taken to be syndetic. See \cite{BMR20} for more discussion of this issue.

We skip the proof of Theorem \ref{T:Main3} due to its resemblance to the proof of our next result, which strengthens Theorem \ref{T:Main2'} when $p_1 = p_2$.
\begin{theorem}[Popular common differences for polynomial corners]\label{T:Main3'}
	Let $p\in \Z[t]$  be intersective. Then for every ergodic system  $(X,\CX,\mu, T_1,T_2)$, every  $A\in \CX$,  and   every $\varepsilon>0$,
	the set
	\begin{equation}\label{E:syndetic}
	\big\{n\in \N\colon  \mu(A\cap T_1^{-p(n)}A\cap T_2^{-p(n)}A)\geq (\mu(A))^4-\varepsilon\big\}
	\end{equation}
	is syndetic. So is the set
	\begin{equation*}
	\big\{n\in \N\colon  d^*(\Lambda\cap (\Lambda-v_1 p(n))\cap (\Lambda-v_2p(n)))\geq (d^*(\Lambda))^4-\varepsilon\big\}
	\end{equation*}
    for any directions $v_1, v_2\in\Z^d$, a set $\Lambda\subseteq\Z^d$, and $\veps>0$.
\end{theorem}
\begin{remarks}
When $p(n)=n$, this was established by Chu~\cite{Chu11} using
the machinery of magic extensions of Host~\cite{H09}. Mandache \cite{Ma21} later gave a purely finitary proof of the combinatorial statement.
Donoso and Sun \cite{DS18} then showed that the lower bounds are optimal,
in the sense that no power of $\mu(A)$ smaller than $4$ works. Later, Fox et. all \cite{FSSSZ20} provided an alternative proof of the combinatorial counterpart of this statement, and Berger \cite{Ber21} extended the aforementioned results to corners in general compact Abelian groups.

None of these methodologies apply to general polynomial corners; Theorem \ref{T:Main3'} has previously been unknown even for $p(n)=n^2$. In the general case of two jointly intersective polynomials $p_1,p_2$, the lower bounds still hold. Indeed, if they are independent, this was established by the authors in \cite[Corollary~2.11]{FrKu22a} with the lower bound $(\mu(A))^3$ in place of $(\mu(A))^4$ and without any ergodicity assumptions; if they are dependent, it follows from \eqref{E:syndetic}.

The phenomenon presented in Theorems \ref{T:Main3} and \ref{T:Main3'} is very specific to length-three corners: indeed, the lower bounds fail if we add a third transformation \cite{DS18}. In general, the existence and nonexistence of popular common difference is a vastly more intricate phenomenon than multiple recurrence, governed by more obscure principles and therefore harder to characterize. In fact, it fails for most linear configurations - see \cite{SSZ21} for a comprehensive discussion of the topic.
\end{remarks}
\begin{proof}[Proof of \cref{T:Main3'} assuming \cref{T:Main2}]
Let $\varepsilon>0$. We assume, as we may, that $\mu(A)>0$ and $\varepsilon <(\mu(A))^4$.  Let also $f:={\bf 1}_A$.

First, we extract from \cite[Section 4]{Chu11} (by combining Proposition 4.6-Corollary 4.9) the existence of a compact Abelian group $Z$, an element $\alpha \in Z$, a nonempty open subset $U$ of $Z$ containing the identity, and nonnegative $\phi\in C(Z)$ with $\phi=1$ on $U$, such that for every $k\in \N$, we have
\begin{equation}\label{E:k!}
	\lim_{N-M\to\infty} \frac{1}{N-M}\sum_{M\leq n<N}   \phi(k! n \alpha)\cdot \int f\cdot T_1^{k!n}f\cdot T_2^{k!n}f\, d\mu \geq \delta_k\cdot \Big(\int f\, d\mu\Big)^4,
\end{equation}
where $$\delta_k:=\lim\limits_{N-M\to\infty}\frac{1}{N-M}\sum_{M\leq n<N} \phi(k! n \alpha).$$
(While the results in \cite{Chu11} are stated for $k=1$, the same argument works for any $k\in\N$.)
From the Khintchine recurrence theorem for the rotation by $k!\alpha$, we infer that $\inf\limits_{k\in \N} \delta_k>0$.

Next, we claim that
\begin{multline}\label{E:diffpk!}
	\lim_{k\to\infty}\lim_{N-M\to\infty} \frac{1}{N-M}\sum_{M\leq n<N}
	 \left(\phi(p(k! n + n_k) \alpha)\cdot \int f\cdot T_1^{p(k! n + n_k)}f\cdot T_2^{p(k!n + n_k)}f\, d\mu\right.\\
	\left.-\phi(k! n \alpha)\cdot \int f\cdot T_1^{k!n}f\cdot T_2^{k!n}f\, d\mu\right)=0
\end{multline}
for a sequence $(n_k)_k$ chosen so that $k!$ divides $p(n_k)$ for every $k\in\N$.
Indeed, after approximating $\phi$ uniformly by linear combinations of characters,
we see that it suffices to verify \eqref{E:diffpk!} when $e(p(k!n + n_k)\beta)$ and $e(k!n\beta)$
take the place of $\phi(p(k! n +n_k) \alpha)$ and $\phi(k! n \alpha)$ for some $\beta\in [0,1)$.\footnote{In performing this reduction, we combine the following facts: 1) every compact Abelian group is an inverse limit of compact Abelian Lie groups; 2) every compact Abelian Lie group $G$ is isomorphic to $\T^d\times H$ for some $d\in\N_0$ and a finite group $H$, and hence can be embedded $\iota_G\colon G\to\T^{d'}$ for some $d'\in\N$; 3) every character on a compact Abelian Lie group $G$ takes the form $e(\eta\cdot\iota_G(x))$ for some $\eta\in\Z^{d'}$; 4) and hence by the Stone-Weierstrass theorem, $\phi$ can be approximated uniformly by the countable collection of functions $\chi(x) = e(\eta\cdot\iota_G(\pi_G (x)))$, where $G$ belongs to a countable collection of compact Abelian Lie groups, $\eta\in\Z^{d'}$, and $\pi_G\colon Z\to G$ is the factor map. For any $m\in\Z$, we then have $\chi(m\alpha) = e(m\beta)$ for $\beta:=\eta\cdot\iota_G(\pi_G (\alpha))$.}
To verify this, we apply \cref{T:Main2} to the (not necessarily ergodic)
system $(\tilde{X},\CX, \tilde{\mu},\tilde{T}_1,\tilde{T}_2)$ defined by
\begin{align*}
 \tilde{X}:=X\times \T,\quad \tilde{\mu}:=\mu\times m_\T,\quad \tilde{T}_1:=T_1\times R, \quad \textrm{where}\quad
R y:=y+\beta,\quad \tilde{T}_2:=T_2\times \text{id},
\end{align*}
and the functions
\begin{align*}
    \tilde{f}_0:=f\otimes \overline{g}\quad \textrm{with} \quad g(y):=e(y),\quad
\tilde{f}_1:=f\otimes g,\quad \tilde{f}_2:=f\otimes 1.
\end{align*}
Combining \eqref{E:k!}
and \eqref{E:diffpk!},
we deduce that
\begin{multline*}
\lim_{N-M\to\infty} \frac{1}{N-M}\sum_{M\leq n<N}\phi(p(k! n + n_k) \alpha)\cdot
\int f\cdot T_1^{p(k! n + n_k)}f\cdot T_2^{p(k! n + n_k)}f\, d\mu\\
\geq \delta_k\cdot
\Big( \Big(\int f\, d\mu\Big)^4-\varepsilon\Big)
\end{multline*}
for all sufficiently large $k\in \N$. The claim then follows from the identity
\begin{align*}
    \lim_{k\to\infty} \lim_{N-M\to\infty} \frac{1}{N-M}\sum_{M\leq n<N}\phi(p(k! n + n_k) \alpha) = \lim_{k\to\infty} \delta_k
\end{align*}
and the previously established fact that the limit is positive. This identity can in turn be verified by approximating $\phi(p(k! n + n_k) \alpha)$ and $\phi(k! n \alpha)$ with linear combinations of $e(p(k! n + n_k)\beta)$ and $e(k!n\beta)$ and noting that
\begin{align*}
        \lim_{k\to\infty} \lim_{N-M\to\infty} \frac{1}{N-M}\sum_{M\leq n<N}e(p(k! n + n_k) \beta) = \textbf{1}_\Q(\beta) =  \lim_{k\to\infty} \delta_k;
\end{align*}
the intersectivity of $p$ and the choice of $n_k$ crucially ensure that every $r\in\N$ divides $p(k! n + n_k)$ for all $n\in\Z$ and sufficiently large $k\in\N$ depending on $r$.

Recalling that $f:={\bf 1}_A$, this implies that the set \eqref{E:syndetic} is syndetic, and the combinatorial statement then follows from a multi-dimensional variant of the uniform Furstenberg correspondence principle for ergodic systems \cite[Proposition 3.1]{BHK05}.
\end{proof}

We remark that, unlike the results proved in \cite{Fr21, FrKu22b,FrKu22a},
where the theory of characteristic factors and extensions was sidestepped,
the current results crucially rely on this theory. In fact, we
use a broad array of techniques developed over the past two decades
in ergodic theory: recent degree lowering  and seminorm smoothing techniques,
the machinery of magic extensions of Host, and new structured extension results
motivated by work of Tao and Leng, as well as foundational tools
such as the Host-Kra theory of characteristic factors and
equidistribution on nilmanifolds. The reader will find a detailed
proof sketch explaining the main technical novelties of this article
in \cref{SS:proofsketch}, after the appropriate terminology is introduced. In \cref{SS: issues} we then discuss the obstructions preventing us from extending Theorems \ref{T:Main1} and \ref{T:Main2} to longer averages and mean convergence.

\subsection{Decomposition results}
Our final application is to verify \cref{Con3} for $\ell=2$ by combining the limit formulas of Theorems~\ref{T:Main1} and~\ref{T:Main2} with the decomposition result of Leng~\cite{Leng25}; the part concerning primes also relies on the work of Koutsogiannis and Tsinas~\cite{KT25}.
	  This also partially resolves  problems posed  by Le~\cite[Section~1.4]{Le18} and Le, Moreira, and Richter~\cite[Questions~1 and~2]{LME21}.
\begin{theorem}\label{T:Main4}
	If $(C_2(n))_n$ is as in \eqref{E:multicorr}, 
	then there exists a nilsequence $(\psi(n))_n$  such that if  	$a\in \Z[t]$  is a nonconstant polynomial, or $a\colon \R_+\to \R$  is a Hardy-field function of polynomial growth 
  that satisfies \eqref{E: log away},    we have
	\begin{equation}\label{E:Cpsi}
	\lim_{N\to\infty}\frac{1}{N}\sum_{n=1}^N |C_2(a(n))-\psi(a(n))|=0.
	\end{equation}
Furthermore, \eqref{E:Cpsi} still holds when $n$ is replaced by the $n$-th prime $p_n$.
\end{theorem}
\begin{remark}
	 This result, combined with the nil plus  null decomposition theorem from \cite[Theorem~2.17]{FrKu22a} for sequences of the form $\int f_0 \cdot T_1^{p_1(n)} f_1 \cdot T_2^{p_2(n)} f_2\, d\mu$, which handles the case of rationally independent polynomials $p_1, p_2$, allows us to obtain nil plus  null decomposition results for any pair of integer polynomials. Extending such results to general families of three or more integer polynomials remains a challenging open problem.
\end{remark}
\begin{proof}
By \cite{Leng25}, there exists a nilsequence $(\psi(n))_n$ satisfying \eqref{E:Cpsi} for $a(n):=n$. We work with this nilsequence from now on and show that it also satisfies \eqref{E:Cpsi}  for all asserted choices of $a\colon \R_+\to \R$.
	
	Suppose first that  $a\colon \R_+\to \R$ is a Hardy-field function of polynomial growth 
	that satisfies \eqref{E: log away}.
	It suffices to show that 
	\begin{equation}\label{E:corid}
	\lim_{N\to\infty}\frac{1}{N}\sum_{n=1}^N |C_2(a(n))-\psi(a(n))|^2=
	\lim_{N\to\infty}\frac{1}{N}\sum_{n=1}^N |C_2(n)-\psi(n)|^2,
	\end{equation}
since the right-hand side vanishes by our choice of $\psi$. 
We do this by expanding the squares on both sides and verifying that the three resulting terms coincide one by one.
We start by applying \cref{T:Main1} to the product system $(X\times X,\mathcal{X}\times\mathcal{X},\mu\times\mu,T_1\times T_1,T_2\times T_2)$ for the functions $f_j\otimes \overline{f_j}$ $j=0,1,2$;  
	we get in a straightforward way that 
	$$
	 	\lim_{N\to\infty}\frac{1}{N}\sum_{n=1}^N |C_2(a(n))|^2=\lim_{N\to\infty}\frac{1}{N}\sum_{n=1}^N |C_2(n)|^2.
	$$ 
	 Similarly, 
	 using 	 \cite[Proposition~2.4]{Fr15} to approximate $\psi(n)$ uniformly by  multiple correlation sequences of the form \eqref{E:multicorr} with $\ell=2$,  and 	applying \cref{T:Main1} to another  product system, we deduce  that 
	 $$
 	\lim_{N\to\infty}\frac{1}{N}\sum_{n=1}^N C_2(a(n))\cdot \overline{\psi(a(n))}=
 	\lim_{N\to\infty}\frac{1}{N}\sum_{n=1}^N C_2(n)\cdot \overline{\psi(n)}.
	 $$ 
	 Lastly,  it follows by \cite{Fr09} (see \cref{T: Hardy equidistribute} below) that 
	 $$
	 	\lim_{N\to\infty}\frac{1}{N}\sum_{n=1}^N |\psi(a(n))|^2=	\lim_{N\to\infty}\frac{1}{N}\sum_{n=1}^N |\psi(n)|^2.
	 $$
	 (Alternatively, one could approximate $\psi$ by a multiple correlation sequence and apply \cref{T:Main1} again.)
	 Combining the above, we obtain \eqref{E:corid}.

Moreover, using the corresponding refinement of \cref{T:Main1} along the primes, mentioned earlier, we similarly obtain that if $a\colon \R_+ \to \R$ is a Hardy-field function of polynomial growth satisfying \eqref{E: log away}, then equation \eqref{E:Cpsi} remains valid when $n$ is replaced by the $n$-th prime $p_n$.

	 Suppose now that $a\in \Z[t]$ is a nonconstant polynomial. 
	 We first claim that for all $j_k\in \N$ we have  
	 $$
	\lim_{k\to\infty}  \lim_{N\to\infty}\frac{1}{N}\sum_{n=1}^N |C_2(a(k!n+j_k))-\psi(a(k!n+j_k))|^2=
	0.
	 $$
To see this, we use \eqref{E:pkj} and argue as before to deduce that the left-hand side equals
$$\lim_{k\to\infty}\lim_{N\to\infty}\frac{1}{N}\sum_{n=1}^N |C_2(k!n+a(j_k))-\psi(k!n+a(j_k))|^2,$$
which vanishes by our assumption that \eqref{E:Cpsi} holds for $a(n)=n$ (in fact, the inner limit vanishes for every $k\in\N$). We deduce that 
	  $$
	 \lim_{k\to\infty}  \frac{1}{k!}\sum_{j=1}^{k!}\lim_{N\to\infty}\frac{1}{N}\sum_{n=1}^N |C_2(a(k!n+j))-\psi(a(k!n+j))|^2=
	 0,
	 $$
	 hence,
	 \begin{equation}\label{E:Cpsiqn}
	  \lim_{N\to\infty}\frac{1}{N}\sum_{n=1}^N |C_2(a(n))-\psi(a(n))|^2=
	 0,
	 \end{equation}
 as required.

 Lastly, we claim that if $a\in \Z[t]$ is a nonconstant polynomial, then 
\begin{equation}\label{E:Cpsipn}
	\lim_{N\to\infty}\frac{1}{N}\sum_{n=1}^N |C_2(a(p_n))-\psi(a(p_n))|^2=0.
\end{equation}
 Expanding the square,  using  \cite[Lemma~2.1 and Proposition 3.6]{FrHK11}
	(which builds on the Gowers uniformity of the bon Mangoldt function \cite[Theorem~7.2]{GT09b}),  and arguing as before, we get that the left-hand side equals 
$$
\lim_{k\to\infty}\frac{1}{\phi(k!)}\sum_{b\leq k!, \,  (b,k!)=1}	\lim_{N\to\infty}\frac{1}{N}\sum_{n=1}^N |C_2(a(k!n+b))-\psi(a(k!n+b))|^2	
$$
($k!$ plays the same role as $W$ in \cite[Proposition 3.6]{FrHK11}), where $\phi$ is Euler's totient function.
The last limit vanishes, since by \eqref{E:Cpsiqn}  the inner limit vanishes for every $k\in\N$.  This proves \eqref{E:Cpsipn}. 
 Alternatively, for this last part one could use \eqref{E:Cpsiqn} and argue as in \cite[Proposition~4.5]{TT19}, which relies on softer input about the primes, such as \cite[Proposition~8.1]{GT08}, and the elementary estimate \cite[Lemma~3.5]{FrHK11}. 
\end{proof}

\subsection{Notation}
The letters $\C, \R, \R_+, \Q, \Z, \N, \N_0$ stand for the set of complex numbers, reals, positive reals, rationals, integers, positive integers, and nonnegative integers. With  $\D := \{z\in\C: |z|\leq 1\}$ we denote the closed complex unit disc, $\S^1 := \{z\in \C: |z| = 1\}$ is the unit circle, while $\T$ stands for the one dimensional torus, which we often identify with $\R/\Z$ or  with $[0,1)$. We let $[N]:=\{1, \ldots, N\}$ and $[\pm N]:=\{-N, \ldots, N\}$ for any $N\in\N$.  For a ring $R$, we use $R[t]$ to denote the collection of single-variable polynomials with coefficients in $R$.

We denote the indicator function of a set $E$ by $\mathbf{1}_E$.

For an element $t\in \R$, we let $e(t):=e^{2\pi i t}$.

  If $a\colon \N^s\to \C$ and $A$ is a nonempty finite subset of $\N^s$ for some $s\in\N$,  we denote the average of $a$ over $A$ via
  $$
  \E_{n\in A} a(n):=\frac{1}{|A|}\sum_{n\in A}\, a(n).
  $$

 We use the notation $(a(n))_{n\in I}$ for an $I$-indexed sequence; when $I = \N^s$ for some $s\in\N$ clear from the context, we employ the shorthand notation $(a(n))_n$. If $I = \{0,1\}^s$ for some $s\in\N$, we also write $a(\ueps)_\ueps$ for $(a(\ueps))_{\ueps\in\{0,1\}^s}$.

  Given a system $(X, \CX, \mu, T_1, \ldots, T_\ell)$, we denote the group of transformations generated by $T_1, \ldots, T_\ell$ with $\angle{T_1, \ldots, T_\ell}$. For a vector $h=(h_1,\ldots, h_\ell)\in \Z^\ell$, we denote by $T^h$ the element of $\angle{T_1, \ldots, T_\ell}$ given by
 $$
 T^{h}:=T_1^{h_1}\cdots T_\ell^{h_\ell}.
 $$
 We also let $e_1, \ldots, e_\ell$ be the unit coordinate vectors in $\Z^\ell$ so that $T^{e_j} = T_j$ for every $j\in[\ell]$.

 For a system $(X, \CX, \mu, T)$, we denote by $\CI(T)$ the $\sigma$-algebra of $T$-invariant sets, and we set $I(T):=L^\infty(X, \CI(T),\mu)$ to be the (closed, $T$-invariant) algebra of bounded $T$-invariant functions.

We let $\CC z := \overline{z}$ be the complex conjugate of $z\in \C$.

We use the standard asymptotic notation: $f\ll g$, $g\gg f$, or $f=O(g)$ whenever there exists $C>0$ so that $|f(t)|\leq C g(t)$ for all $t$ in the domain of $f,g$. If $C>0$ depends on some parameter, we record it in the subscript, e.g., $f\ll_k g$ if $C=C(k)$. We also write $f = o_{k_1; \ldots; k_r; m\to \infty}(g)$ to denote that $\lim\limits_{m\to\infty}\frac{f(k_1, \ldots, k_r, m)}{g(k_1, \ldots, k_r, m)} = 0$ for some parameters $k_1, \ldots, k_r, m$.

\subsection{Acknowledgments}
We thank J. Leng for explanations regarding his work \cite{Leng25}, and for sharing with us the construction given in Section \ref{SS: extension} before it appeared online. {We also thank J. Griesmer for useful remarks.} 
For the purpose of Open Access, the authors have applied a CC-BY public copyright licence to any Author Accepted Manuscript (AAM) version arising from this submission.

\section{Background in ergodic theory}
This section is dedicated to basic notions of the Host-Kra theory, which will then be used to state generalizations of our main results in the following section.  We also collect essential facts on
magic extensions for later use. All systems in this article are assumed to be regular, i.e., they are defined on a complete separable metric space endowed with Borel $\sigma$-algebra and a Borel probability measure.
\subsection{Box seminorms}\label{SS:boxseminorms}
Let $(X, \CX, \mu, T_1, \ldots, T_\ell)$ be a system and $f\in L^{\infty}(\mu)$. For each $R\in\angle{T_1, \ldots, T_\ell}$, we define the \textit{multiplicative derivative}
 $$
 \Delta_{R} f :=  f \cdot R \overline{f}
 $$
   and for $R_1, \ldots, R_s\in\angle{T_1, \ldots, T_\ell}$, we denote the \textit{iterated multiplicative derivative} via
$$
\Delta_{R_1, \ldots, R_s} f  := \Delta_{R_1}\cdots\Delta_{R_s} f = \prod_{\ueps\in\{0,1\}^s} \CC^{|\ueps|} R_1^{\eps_1}\cdots R_s^{\eps_s}f.
$$
If each transformation $R_j$ appears $k_j$ times, we also denote
\begin{align*}
    \Delta_{R_1^{\times k_1}, \dots, R_s^{\times k_s}}f := \Delta_{\underbrace{R_1, \ldots, R_1}_{k_1},\; \ldots,\; \underbrace{R_s, \ldots, R_s}_{k_s}}f.
\end{align*}
Lastly, we employ the conventions
\begin{align*}
    \Delta_{R_1, \ldots, R_s; h}f &:=\Delta_{R_1^{h_1},\ldots, R_s^{h_s}}f\quad &&\textrm{for}\quad h\in\Z^s\\
    \textrm{and}\quad \Delta_{R_1^{\times k_1}, \ldots, R_s^{\times k_s}; h}f &:=\Delta_{R_1^{h_{11}},\ldots, R_1^{h_{1k_1}}}\cdots\Delta_{R_s^{h_{s1}},\ldots,  R_s^{h_{sk_s}}}f\quad &&\textrm{for}\quad h\in\Z^{k_1+\cdots+k_s}
\end{align*}
whenever convenient.

Following Host~\cite{H09}, we inductively define \textit{box seminorms} by letting
$$
\nnorm{f}_{\emptyset}:=\int f\, d\mu
$$
and
\begin{align}\label{ergodic identity}
	\nnorm{f}_{R_1, \ldots, {R_s}}^{2^{s}}:=\lim_{H\to\infty}\E_{h\in[H]}\nnorm{\Delta_{R_s^h}f}_{R_1, \ldots, R_{s-1}}^{2^{s-1}}
\end{align}
for $s\in\N$ and $R_1, \ldots, R_s\in\angle{T_1, \ldots, T_\ell}$. Expanding the definition, we thus have
\begin{align*}
    \nnorm{f}_{R_1, \ldots, {R_s}}^{2^{s}} = \lim_{H_s\to\infty}\E_{h_s\in[H_s]}\cdots \lim_{H_1\to\infty}\E_{h_1\in[H_1]}\int \Delta_{R_1, \ldots, R_s; h}f\; d\nu.
\end{align*}
As explained in \cite[Lemma 1]{H09}, we can replace $[H_i]$ in the average by any F{\o}lner sequence in $\Z$. Combined with \cite[Lemma 1.1]{BL15}, this allows us to replace any $s'$ iterated limits by the single limit $\lim\limits_{H\to\infty}\E_{h\in[H]^{s'}}$.

In the special case when $R:=R_1, \ldots, R_s$, we obtain the \textit{Host-Kra seminorm} of $R$ of degree $s$, denoted by
\begin{align*}
    \nnorm{f}_{s, R}:=\nnorm{f}_{R_1, \ldots, R_s}
\end{align*}
and originally defined by Host and Kra in their groundbreaking work on norm convergence of multiple ergodic averages \cite{HK05a} {(an alternative proof of the latter result was given by Ziegler~\cite{Zi07}).}
While the Host-Kra seminorms predate box seminorms by several years, the latter appear more naturally in the context of multiple ergodic averages of commuting transformations and therefore occupy more prominent place in this work.

Sometimes, we may wish to emphasize the measure with respect to which the seminorm is defined; if so, we write $\nnorm{f}_{R_1, \ldots, R_s;\mu}$ and $\nnorm{f}_{s, R;\mu}$. For instance, if $\mu = \int \mu_x\; d\mu(x)$ is the ergodic decomposition of $\mu$ with respect to the joint action of $R_1, \ldots, R_s$, then we have the easy-to-derive formula
\begin{align}\label{E: ergodic decomposition}
    \nnorm{f}_{R_1, \ldots, R_s;\mu}^{2^s} = \int  \nnorm{f}_{R_1, \ldots, R_s;\mu_x}^{2^s}\; d\mu(x).
\end{align}

We now record several other standard properties of box seminorms freely used throughout the paper. Their proofs can be found e.g. in \cite{DKKST24, FrKu22b, H09}. In what follows, we take $R_1, \ldots, R_{s}\in\angle{T_1, \ldots, T_\ell}$ and $f\in L^\infty(\mu)$.
\begin{enumerate}
    \item (Permutation invariance) For
    any permutation $\sigma:[s]\to[s]$,  let
    \begin{align*}
        \nnorm{f}_{R_1, \ldots, R_s} = \nnorm{f}_{R_{\sigma(1)},\ldots, R_{\sigma(s)}}.
    \end{align*}
    \item (Monotonicity) We have
    \begin{align*}
    \nnorm{f}_{R_1}\leq \nnorm{f}_{R_1, R_2}\leq \cdots\leq    \nnorm{f}_{R_1, \ldots, R_s}.
    \end{align*}
    \item (Inductive formula) For any $1\leq s'\leq s$, let 
    \begin{align*}
        \nnorm{f}_{R_1, \ldots, R_s}^{2^s} = \lim_{H_s\to\infty}\E_{h_s\in[H_s]}\cdots \lim_{H_{s'+1}\to\infty}\E_{h_{s'+1}\in[H_{s'+1}]}\nnorm{\Delta_{R_{s'+1}, \ldots, R_s; h'}f}_{R_1, \ldots, R_{s'}}^{2^{s'}},
    \end{align*}
    where $h' = (h_{s'+1}, \ldots, h_s)$. Moreover, the iterated limit can be replaced by the single limit $\lim\limits_{H\to\infty}\E_{h'\in [H]^{s-s'}}$.
    \item (Gowers-Cauchy-Schwarz inequality) For any $(f_\ueps)_{\ueps\in\{0,1\}^s}\subseteq L^\infty(\mu)$, we have 
    \begin{align*}
        \Big|\lim_{H_s\to\infty}\E_{h_s\in[H_s]}\cdots \lim_{H_1\to\infty}\E_{h_1\in[H_1]}\int \prod_{\ueps\in\{0,1\}^s}\CC^{|\ueps|} R_1^{\eps_1}\cdots R_s^{\eps_s}f_\ueps\; d\mu\Big|\leq \prod_{\ueps\in\{0,1\}^s}\nnorm{f_\ueps}_{R_1, \ldots, R_s}.
    \end{align*}
    Once again, the iterated limit on the left-hand side can be replaced by the single limit $\lim\limits_{H\to\infty}\E_{h\in [H]^{s}}$.
    \item \label{I:scaling} (Scaling) For any nonzero $r_1, \ldots, r_s\in\Z$, let 
    \begin{align*}
        \nnorm{f}_{R_1, \ldots, R_s}\leq \nnorm{f}_{R_1^{r_1}, \ldots, R_s^{r_s}}.
    \end{align*}
    (We caution that $R_i^{r_i}$ denotes the $r_i$-th power of $R_i$, not $r_i$ copies of $R_i$, denoted by $R_i^{\times r_i}$.) Conversely, if $s\geq 2$, then
    \begin{align*}
        \nnorm{f}_{R_1^{r_1}, \ldots, R_s^{r_s}}\leq |r_1\cdots r_s|^{1/2^s}\nnorm{f}_{R_1, \ldots, R_s}.
    \end{align*}
\end{enumerate}

\subsection{Box factors and dual functions}
Let $(X, \CX, \mu, T_1, \ldots, T_\ell)$ be a system, $s\in\N$, and $R_1, \ldots, R_s\in\angle{T_1, \ldots, T_\ell}$. Set also $\{0,1\}^s_* = \{0,1\}^s\setminus\{{0}\}$. For $(f_\ueps)_{\ueps\in \{0,1\}^s_*}\subseteq L^\infty(\mu)$, we define the \textit{dual function of  $(f_\ueps)_{\ueps\in \{0,1\}^s_*}$ along $R_1, \ldots, R_s$} via
\begin{align*}
    \CD_{R_1, \ldots, R_s}((f_\ueps)_\ueps) := \lim_{H_s\to\infty}\E_{h_s\in[H_s]}\cdots \lim_{H_1\to\infty}\E_{h_1\in[H_1]} \prod_{\ueps\in\{0,1\}^s_*}\CC^{|\ueps|} R_1^{\eps_1}\cdots R_s^{\eps_s}f_\ueps.
\end{align*}
The limit exists in $L^2(\mu)$ (see e.g. \cite[Proposition 2.2]{TZ16}) and can be replaced by the single limit $\lim\limits_{H\to\infty}\E_{h\in [H]^{s}}$ thanks to \cite[Lemma 1.1]{BL15}. We call $s$ the \textit{degree} of the dual function. If $f_\ueps = f$ for all $\ueps\in\{0,1\}^s_*$, we simply denote
\begin{align*}
    \CD_{R_1, \ldots, R_s}(f) := \CD_{R_1, \ldots, R_s}((f_\ueps)_\ueps)
\end{align*}
and observe that
\begin{align}\label{E: dual identity}
        \nnorm{f}_{R_1, \ldots, R_s}^{2^s} = \int f \cdot \CD_{R_1, \ldots, R_s}(f)\, d\mu.
\end{align}

We then define
\begin{align*}
    Z(R_1, \ldots, R_s) := \overline{\Span_\C\{\CD_{R_1, \ldots, R_s}((f_\ueps)_\ueps)\colon  (f_\ueps)_{\ueps\in \{0,1\}^s_*}\subseteq L^\infty(\mu)\}}
\end{align*}
to be the closed, $T_1, \ldots, T_\ell$-invariant subspace of $L^2(\mu)$, generated by dual functions along $R_1, \ldots, R_s$. By \cite[Proposition 2.3]{TZ16}, $Z({R_1, \ldots, R_s})$ defines an  algebra (i.e., it is closed under multiplication), and hence (since it is also closed under  conjugation) there exists a factor $\CZ({R_1, \ldots, R_s})\subseteq \CX$ satisfying
\begin{align*}
    Z(R_1, \ldots, R_s) = L^\infty(X, \CZ({R_1, \ldots, R_s}), \mu).
\end{align*}
It is then easy to see that the factor $\CZ({R_1, \ldots, R_s})$ satisfies the well-known property
\begin{align}\label{E: factor property}
    \nnorm{f}_{R_1, \ldots, R_s} = 0 \quad \Longleftrightarrow\quad \E(f|\CZ(R_1, \ldots, R_s)) = 0;
\end{align}
one direction follows from \eqref{E: dual identity} while the other one is a consequence of the Gowers-Cauchy-Schwarz inequality.

When $R:=R_1 = \cdots = R_s$, we also denote
\begin{align*}
    Z_{s-1}(R):=Z(R_1, \ldots, R_s)\quad\textrm{and}\quad \CZ_{s-1}(R):=\CZ(R_1, \ldots, R_s)
\end{align*}
(or $Z_{s-1,\mu}(R), \CZ_{s-1,\mu}(R)$ if we want to emphasize the measure),
calling the latter the \textit{Host-Kra factor of degree $s-1$}.

Finally, we observe that
\begin{align*}
  Z(R) = I(R) = \{f\in L^\infty(\mu)\colon Rf = f\} \quad\textrm{and}\quad  \CZ(R) = \CI(R) = \{E\in \CX\colon  R\inv E = E\}
\end{align*}
are the algebra/$\sigma$-algebra of $R$-invariant functions/sets. Here and elsewhere, all the equalities are understood to hold up to sets of measure 0.

\subsection{The factor $\CZ_1(T)$ and nonergodic eigenfunctions}
Let $(X, \CX, \mu, T)$ be a (not necessarily ergodic) system. The factor $\CZ_1(T) = \CZ(T,T)$ admits a particularly nice structure that can be expressed in terms of the following notion.

\begin{definition}[Nonergodic eigenfunctions]
A function $\chi\in L^\infty(\mu)$ is called a \textit{nonergodic eigenfunction} of $T$ if it satisfies the following properties:
    \begin{enumerate}
        \item $T\chi = \lambda \chi$ for some $\lambda\in I(T)$ (which we call the \textit{nonergodic eigenvalue} of $f$);
        \item $|\chi(x)|\in\{0,1\}$ for $\mu$-a.e. $x\in X$, and $\lambda(x)=0$ whenever $\chi(x)=0$.
\end{enumerate}
We denote the set of nonergodic eigenfunctions of $T$ by $\CE(T)$.
\end{definition}

		 For ergodic systems, a nonergodic eigenfunction is either the zero function or a classical  unit modulus eigenfunction. For general systems, each function $\chi\in \CE(T)$ satisfies
		 \begin{equation}\label{E:nonergodiceigen}
		 	\chi(Tx)={\bf 1}_E(x) \, e(\phi(x))\, \chi(x)
		 \end{equation}
		  for some $T$-invariant set $E\in \CX$ and measurable $T$-invariant function $\phi \colon X\to \T$.

By \cite[Theorem 5.2]{FH18}, the factor $Z_1(T)$ is a closed linear span of nonergodic eigenfunctions.

Later in the paper, we will make use of the following lemma, which shows that if we slightly modify the conditions defining nonergodic eigenfunctions, we still get an element of $Z_1(T)$.
\begin{lemma}\label{L:NonErgodicEigen}
	Let $(X, \CX, \mu, T)$ be a system and $\chi\in L^\infty(\mu)$ be a function satisfying $|\chi|=1$ and $T\chi = \lambda \chi$ for some $\lambda\in I(T)$.
     Then $\chi\in Z_1(T)$.
\end{lemma}
\begin{proof}
    	Our assumptions give that $|\lambda|=1$ and
	$$
	\chi = T^{h_1+h_2}\overline{\chi}\cdot T^{h_1}\chi\cdot T^{h_2}\chi \quad \text{for all}\quad h_1,h_2\in \Z.
	$$
	 Averaging over $h_1,h_2\in \N$, we get that
	\begin{equation}\label{E:chi}
		\chi=\lim_{H\to\infty} \E_{h_1,h_2\in [H]}T^{h_1+h_2}\overline{\chi}\cdot T^{h_1}\chi\cdot T^{h_2}\chi	\in Z_1(T),
	\end{equation}
completing the proof. 
\end{proof}

\subsection{Nilsystems and the Host-Kra structure theorem}\label{SS:Nilstuff}
An {\em $s$-step  nilmanifold} is a compact homogeneous space $X=G/\Gamma$ where $G$ is an $s$-step nilpotent Lie group and $\Gamma$ is a discrete cocompact subgroup.  We denote by  $e_X$ the image  of the identity element  of $G$ in $X$. For $b\in G$, the transformation $T\colon X\to X$ defined by $Tx=bx$ is called an {\em $s$-step  nilrotation} on $X$ and the system $(X,m_X,T)$, where $m_X$ is the projection of the Haar measure of $G$ on $X$,  is called an {\em $s$-step nilsystem.}
If $X=G/\Gamma$ is a nilmanifold and $b\in G$, $x\in X$, then the closure $Y$ of the sequence $(b^n\cdot x)_n$ admits the structure of a nilmanifold, and the action  of $b$ on $Y$ is uniquely ergodic  (\cite[Chapter~11, Theorem~17]{HK18} or \cite[Section~2]{Lei05a}).  Following \cite{BHK05}, we say that $(\psi(n))_n$ is a {\em nilsequence} if it is a uniform limit of sequences of the form $(F(b^nx))_n$ for some  $b\in G$, $x\in X$, and $F\in C(X)$. 

\begin{definition}[Equidistribution on nilmanifolds]
	We say that a sequence $(x_n)_n$ is \textit{equidistributed} on a nilmanifold $X$ if
	$$
	\lim_{N\to\infty}\E_{n\in[N]} F(x_n)=\int F\, dm_X
	$$ for every $F\in C(X)$.
\end{definition}
Nilsystems and their equidistribution properties are relevant to our study
because of the following groundbreaking structural result from \cite{HK05a} (see also
\cite[Chapter~16, Theorem~2]{HK18}).
  \begin{theorem}[Host-Kra structure theorem]\label{T:HK}
 	Let $(X, \CX, \mu,T)$ be an ergodic system. Then for every $s\in \N$, the factor $\CZ_s(T)$ is an inverse limit of $s$-step ergodic nilsystems.
 \end{theorem}
We will not use \cref{T:HK} explicitly in this article;
rather, we rely on a consequence, namely the decomposition
result in \cite[Proposition 3.1]{CFH11}, which is used
crucially in the proof of \cref{P: base case}.

\subsection{Hardy fields}\label{SS: Hardy}
Let $B$ be the collection of equivalence classes (``germs'') of real valued functions defined on some halfline $(t_0,\infty)$ for $t_0\geq 0$, where we identify two functions that eventually agree.  A \emph{Hardy field} is a subfield $\CH$ of the ring $(B, +, \cdot)$ that is closed under differentiation~\cite{Hardy12}. We call $a\colon (t_0,\infty)\to\R$ a \textit{Hardy-field function} if it belongs to some Hardy field $\CH$. 
 By abuse of notation, we speak of ``functions'' rather than ``germs'', understanding that all the operations defined and statements made for elements of $\CH$ are considered only for sufficiently large values of $t\in \mathbb{R}$. 
  We say that $a\in\CH$ has \emph{polynomial growth} if there exists $d>0$ such that $a(t)\ll t^d$.   
 Basic properties of Hardy-field functions can be found in \cite{Bos94, Fr10}.

A typical example of a Hardy field is the Hardy field $\mathcal{LE}$ of \textit{logarithmico-exponential functions}, i.e., functions constructed using a finite combination of symbols $+, -, \times, \div$, $\exp,$ $\log$ acting on the real variable $t$ and on real constants. These include, for example,  the functions $t^b (\log{t})^c$, where $b,c\in \R$.
{\em In this article, we focus on Hardy-field functions in $\mathcal{LE}$.} Our results also apply beyond
this setting, to more general Hardy-fields, provided they satisfy the mild assumptions in \cite[Section 2.2]{DKKST24}.
{In our statements, we also assume that the Hardy-field functions are defined on all of $\R_+$, but all arguments remain valid without modification if they are defined only on a halfline.}

To prove \cref{T:Main1}, we will need the following results on the equidistribution of Hardy-field functions on nilmanifolds and seminorm estimates for multiple ergodic averages along Hardy-field functions.
\begin{theorem}[{\cite[Theorem 1.2]{Fr09}}]\label{T: Hardy equidistribute}
Let $a\colon \R_+\to\R$ be a Hardy-field function of polynomial growth satisfying \eqref{E: log away}. {Then for every nilmanifold $X=G/\Gamma$ and $b\in G$, $x\in X$, the sequence $(b^{[a(n)]}x)_n$  is equidistributed in $\overline{(b^nx)_n}$.}
\end{theorem}

\begin{theorem}[Box seminorm control for Hardy-field functions]\label{T: estimates for Hardy}
    Let $(X, \CX, \mu, T_1, \ldots, T_\ell)$ be a system, $m\in[\ell]$, and $a\colon \R_+\to\R$ be a Hardy-field function of polynomial growth satisfying
        \begin{align}\label{E: beats log}
        \lim_{t\to\infty}\abs{\frac{a(t)}{\log t}} = \infty.
    \end{align}
    Then there exists $s\in\N$ such that for all functions $f_1, \ldots, f_\ell\in L^\infty(\mu)$ with $f_{m+1}\in\CE(T_{m+1})$,  $\ldots, f_\ell\in\CE(T_\ell)$, we have
    \begin{align*}
        \lim_{N\to\infty}\norm{\E_{n\in[N]}T_1^{[a(n)]}f_1\cdots T_\ell^{[a(n)]}f_\ell}_{L^2(\mu)} = 0
    \end{align*}
    whenever $\nnorm{f_m}_{(T_mT_1\inv)^{\times s}, \ldots (T_mT_{m-1}\inv)^{\times s}, T_m^{\times s}} = 0$.
\end{theorem}
We will use Theorem \ref{T: estimates for Hardy} for $(m,\ell) = (2,2)$ and $(m,\ell) = (1,2)$. Theorem \ref{T: estimates for Hardy} follows from \cite[Theorem 10.2]{DKKST24} when $a(t)\gg t^\delta$ for some $\delta>0$ and from \cite[Proposition 6.3]{DKKST24} whenever $a(t) \ll t^\delta$ for all $\delta>0$. When $a\in\R[t]$, the derivation of Theorem \ref{T: estimates for Hardy} from \cite[Theorem 10.2]{DKKST24} implicitly uses the scaling property of box seminorms from \cref{SS:boxseminorms}.
Both aforementioned results cover multiple ergodic averages along Hardy sequences in which $f_{m+1}, \ldots, f_\ell$ are bounded-degree dual functions of the respective transformations rather than nonergodic eigenfunctions. Since degree-2 dual functions of a transformation $T_j$ span $Z_1(T_j)$, hence $\CE(T_j)$, we can derive Theorem \ref{T: estimates for Hardy} from the cited results by approximating nonergodic eigenfunctions in $L^2(\mu)$ by degree-2 dual functions.

 In the simpler case of $m=\ell$, Theorem \ref{T: estimates for Hardy} could also be derived from (much more general) \cite[Theorem 10.1]{DKKST24}, and the two special cases when $a$ is (i) an integer polynomial, and (ii) stays  logarithmically away from real polynomials (this includes fractional powers) had previously been established in \cite[Theorem 2.5]{DFKS24} and \cite[Theorem 5.1]{DKS23} respectively.

\subsection{Magic extensions} We will later need to consider structured extensions of a system
on which certain inverse theorems take a particularly simple form.
One such extension is a slight variant of the magic
systems and extensions defined by Host \cite{H09}.
\begin{definition}[Magic systems and extensions]
	We call a system $(X, \CX, \mu, T_1, \ldots, T_\ell)$ \textit{magic} with respect to $R_1, \ldots, R_s\in\angle{T_1, \ldots, T_\ell}$ if
	\begin{align}\label{E: magic property}
		\CZ(R_1, \ldots, R_s) = \CI(R_1)\vee\cdots\vee\CI(R_s).
	\end{align}
	
	Suppose that the transformations $R_1, \ldots, R_s$ are given by $$R_i := T^{b_i} = T_1^{b_{i1}}\cdots T_\ell^{b_{i\ell}}$$ for some $b_i=(b_{i1}, \ldots, b_{i\ell})\in\Z^\ell$. We call an extension $(X^*, \CX^*, \mu^*, T_1^*, \ldots, T_\ell^*)$  of $(X, \CX, \mu, T_1, \ldots, T_\ell)$ \textit{magic} with respect to $R_1, \ldots, R_s$ if  $(X^*, \CX^*, \mu^*, T_1^*, \ldots, T_\ell^*)$ is a magic system with respect to the transformations $R_1^*, \ldots, R_s^*\in \angle{T_1^*, \ldots, T_\ell^*}$ given by $R_i^* = (T^*)^{b_i}$.
\end{definition}
Our definition differs from Host's in that we want the magic property \eqref{E: magic property} to hold for a fixed subset of $\angle{T_1, \ldots, T_\ell}$ rather than for $T_1, \ldots, T_\ell$ themselves. The next result concerns the existence of magic extensions under certain conditions on $R_1, \ldots, R_s$ that always arise in our applications.
\begin{proposition}[Existence of magic extensions]\label{P: magic extensions exist}
	Let $(X, \CX, \mu, T_1, \ldots, T_\ell)$ be a system, and suppose that $R_1, \ldots, R_s\in\angle{T_1, \ldots, T_\ell}$ generate $\angle{T_1, \ldots, T_s}$ for some $s\in[\ell]$.
	Then $(X, \CX, \mu, T_1, \ldots, T_\ell)$ admits a magic extension $(X^*, \CX^*, \mu^*, T_1^*, \ldots, T_\ell^*)$ with respect to $R_1, \ldots, R_s$.
\end{proposition}
\begin{proof}
	Consider the system $(X, \CX, \mu, R_1, \ldots, R_s)$. By \cite[Theorem~2]{H09}, it admits an extension $(X^*, \CX^*, \mu^*, R_1^*, \ldots, R_s^*)$ satisfying
	\begin{align*}
		\CZ(R^*_1, \ldots, R^*_s) = \CI(R_1^*)\vee\cdots\vee\CI(R_s^*),
	\end{align*}
	constructed as follows: $X^* := X^{2^s}$, $\CX^* := \CX^{\otimes 2^s}$, $\mu$ is the $s$-dimensional cubic measure induced by $R_1, \ldots, R_s$ (see \cite[Section 2]{H09} for the details), and the transformations $R_1^*, \ldots, R_s^*$ are given by
	\begin{align*}
		R_i^*((x_\ueps)_\ueps) := \begin{cases}
			R_i x_\ueps,\; &\eps_i = 0\\
			x_\ueps,\; &\eps_i = 1.
		\end{cases}
	\end{align*}
	Then the factor map is given by the projection onto the coordinate indexed by $0$.
	
	It remains to define $T_1^*, \ldots, T_\ell^*$ so that $(X^*, \CX^*, \mu^*, T_1^*, \ldots, T_\ell^*)$ is also an extension of $(X, \CX, \mu, T_1, \ldots, T_\ell)$ with respect to the same factor map. By assumption,
	for every $j\in[s]$, we can express $T_j$ uniquely as
	\begin{align*}
		T_j = R_1^{c_{j1}}\cdots R_s^{c_{js}}
	\end{align*}
	for some $c_{ji}\in\Z$. Then we simply define
	\begin{align*}
		T_j^* := (R_1^*)^{c_{j1}}\cdots(R_s^*)^{c_{js}};
	\end{align*}
	in particular, the invariance of $\mu^*$ with regards to $R_1^*, \ldots, R_s^*$ immediately gives the invariance with respect to $T_1^*, \ldots, T_s^*$.
	For $j\in[s+1,\ell]$, we lift $T_j$ to $T^*_j := \underbrace{T_j \times \cdots \times T_j}_{2^s}$; its invariance can then be deduced inductively from the cubic structure of the measure $\mu^*$ and the $T_j$-invariance of $\mu$.
\end{proof}

\begin{proposition}[Soft inverse theorem for magic systems]\label{P: inverse theorem magic}
	Let $(X, \CX, \mu, T_1, \ldots, T_\ell)$ be a magic system with respect to $R_1, \ldots, R_s\in\angle{T_1, \ldots, T_\ell}$. Then for any $\veps>0$ there exists $\delta>0$ (depending on $\veps$, the system, and $R_1, \ldots, R_s$) such that if a 1-bounded function $f\in L^\infty(\mu)$ satisfies
	$$
	\nnorm{f}_{R_1, \ldots, R_s}\geq \veps,
	$$ then there exist 1-bounded functions $g_1\in I(R_1), \ldots, g_s\in I(R_s)$ for which the integral below is real and
	\begin{align*}
		\int f\cdot g_1\cdots g_s\; d\mu \geq \delta.
	\end{align*}
\end{proposition}
\begin{proof}
	Define the seminorm on $L^\infty(\mu)$ by
	$$
	\nnorm{f}^*:=\sup\rem{\abs{\int f\cdot g_1\cdots g_s\; d\mu}\colon g_1\in I(R_1), \ldots, g_s\in I(R_s)\textrm{ all } 1\textrm{-bounded}}
$$
and the multilinear functional $A\colon (L^{2^s}(\mu))^{2^s}\to \C$ by
\begin{align*}
	A((f_\ueps)_\ueps):= \lim_{H\to\infty}\E_{h\in[H]^s}\int \prod_{\ueps\in\{0,1\}^s}\CC^{|\ueps|} R_1^{\eps_1}\cdots R_s^{\eps_s}f_\ueps\; d\mu.
\end{align*}
Note that the limit is known to exist by combining results from \cite{BL15, H09}.

Crucially, \eqref{E: factor property} and our assumption gives us \textit{qualitative control} of $A$ by the norm $\nnorm{\cdot}^*$ at the index 0: if
$\nnorm{f_0}^*=0$, then $\nnorm{f_0}_{R_1, \ldots, R_s}=0$, and hence also  $A((f_\ueps)_\ueps)=0$ by the  Gowers-Cauchy-Schwarz inequality. To get the claimed \textit{soft quantitative control}, we then apply \cite[Proposition~A.2]{FrKu22a} with $X_j:=L^{2^s}(\mu)$, $X_j':=L^\infty(\mu)$, $\norm{\cdot}_{X_j'}:=\norm{\cdot}_{L^\infty(\mu)}$, $\norm{\cdot}_{X_j}:=\norm{\cdot}_{L^{2^s}(\mu)}$ for $j\in[2^s]$ and the seminorm $\nnorm{\cdot}:= \nnorm{\cdot}^*$. For every $\veps>0$ it gives us $\delta>0$ (depending only on $\veps$, the system and $R_1, \ldots, R_s$) such that if $f_\ueps$ is 1-bounded for every $\ueps\in\{0,1\}^s$, then $\nnorm{f_0}^*\leq \delta$ implies $|A((f_\ueps)_\ueps)|\leq \veps^{2^s}$. Applying this result with $f_\ueps := f$ for every $\ueps\in\{0,1\}^s$, we deduce that $\nnorm{f}_{R_1, \ldots, R_s}\leq \veps$. The claim then follows by contrapositive and multiplying $g_1$ with an appropriate complex number on the unit circle.
\end{proof}

\section{General framework, proof strategy, and obstructions}
To prove \cref{T:Main1}, we will work in a more general setting that covers sequences satisfying the following two properties.
\begin{definition}\label{D: good properties}
	Let $a\colon \N\to \Z$ be a sequence.
    We say that
	\begin{enumerate}
		\item    \label{I:gp1i}  $(a(n))_n$ \textit{admits box seminorm control} if there exists $s\in\N$ such that for every system $(X, \CX, \mu, T_1, T_2)$ and all functions $f_0, f_1, f_2\in L^\infty(\mu)$, we have
		\begin{equation}\label{E:vanish}
            \lim_{N\to\infty}\E_{n\in[N]}\int \, f_0\cdot T_1^{a(n)}f_1\cdot T_2^{a(n)}f_2\; d\mu = 0
		\end{equation}
		whenever $\nnorm{f_2}_{(T_2T_1\inv)^{\times s}, T_2^{\times s}} = 0$, and if additionally $f_2\in\CE(T_2)$, then \eqref{E:vanish} also holds whenever $\nnorm{f_1}_{T_1^{\times s}} = 0$;
		\item $(a(n))_n$    \textit{equidistributes on nilsystems} if for every nilsystem $(X, \CX, \mu, T)$ and every $x\in X$, the sequence $(T^{a(n)}x)_n$ is equidistributed in $\overline{(T^n x)_n}$.
	\end{enumerate}
\end{definition}
Examples of sequences that satisfy these properties are the fractional powers, i.e., $a(n) = [n^c]$ with $c\in \R_+\setminus \Z$; see \cref{SS: Hardy} for a detailed explanation.

Unfortunately, polynomial sequences fail to equidistribute on nilsystems; to prove Theorem~\ref{T:Main2}, we therefore proceed to
work with sequences that satisfy the following variant of the preceding properties.
\begin{definition}\label{D: good properties 2}
	Let $a\colon \N\to \Z$ be a sequence and $n_k\in \N_0$ for $k\in \N$.     We say that
	\begin{enumerate}
		\item   \label{I:gp2i} $(a(k!n+n_k))_{n,k}$ \textit{admits box seminorm control} if there exists $s\in\N$ such that for every system $(X, \CX, \mu, T_1, T_2)$ and all functions $f_0, f_1, f_2\in L^\infty(\mu)$, we have
		\begin{equation}\label{E:vanish'}
			\lim_{k\to\infty}\limsup_{N\to\infty}\abs{\E_{n\in[N]}\int f_0\cdot T_1^{a(k!n+n_k)}f_1\cdot T_2^{a(k!n+n_k)}f_2\; d\mu} = 0
		\end{equation}
		whenever $\nnorm{f_2}_{(T_2T_1\inv)^{\times s}, T_2^{\times s}} = 0$, and if additionally $f_2\in\CE(T_2)$, then \eqref{E:vanish'} also holds whenever $\nnorm{f_1}_{T_1^{\times s}} = 0$;
		\item \label{I:3.2ii}  $(a(k!n+n_k))_{n,k}$    \textit{equidistributes on nilsystems} if  for every nilsystem $(X, \CX, \mu, T)$, function $F\in C(X)$,  and every $x\in X$, we have
		\begin{equation}\label{E:nk}
			\lim_{k\to\infty}\lim_{N\to\infty} \E_{n\in [N]}F(T^{a(k!n+n_k)}x)=  \lim_{k\to\infty}\lim_{N\to\infty} \E_{n\in [N]}F(T^{k!n}x).
		\end{equation}
	\end{enumerate}
{Moreover, analogous definitions apply if we  replace the  Ces\`aro average over $n\in \N$ with an average along an arbitrary F\o lner  sequence in $\N$.}
\end{definition}
\begin{remarks}
    The limit on the right-hand side of \eqref{E:nk} can be shown to exist; for example, see the argument in the second part of Section~\ref{SS:pf1.2}.

    It suffices to verify \eqref{E:nk} when $x=e_X$. Indeed, if $X=G/\Gamma$, $Tx=bx$ for some $b\in G$, and   $x=g\cdot e_X$, note that $F(b^nx)=\widetilde F(c^n\cdot e_X)$ for every $n\in \N$, where $\widetilde F(x):=F(gx)$ is in $C(X)$ and  $c:=g^{-1}\, b\, g\in G$.
\end{remarks}
 We explain in  \cref{SS:pf1.2} that nonzero intersective polynomials satisfy these properties.

The next result will be used to prove
\cref{T:Main1}.
\begin{theorem}[Equality of limits I]\label{T: limiting formula}
	Let $a\colon \N\to\Z$ be a sequence that admits box seminorm control and equidistributes on nilsystems. Then for every system $(X, \CX, \mu, T_1, T_2)$ and all functions $f_0,f_1, f_2\in L^\infty(\mu)$, we have the identity
	\begin{align}\label{E: equality of limits}
		\lim_{N\to\infty}\E_{n\in[N]}\int f_0\cdot T_1^{a(n)}f_1\cdot T_2^{a(n)}f_2 \, d\mu = \lim_{N\to\infty}\E_{n\in[N]} \int f_0\cdot T_1^n f_1\cdot T_2^n f_2\, d\mu;
	\end{align}
    in particular, both limits above exist. {Moreover, an analogous result holds when, in both the hypothesis
    and the conclusion, we replace the Ces\`aro average over $n\in\N$
    by an average along an arbitrary F\o lner sequence in $\N$.}
\end{theorem}
\begin{remark}
	Our assumptions  are close to  optimal. Indeed, if \eqref{E: equality of limits} holds, then $(a(n))_n$ admits (optimal) box seminorm control  and \cite[Proposition~8.2]{FrKu22c} gives that it equidistributes on $2$-step nilsystems.
\end{remark}

The next result
  will be used to prove \cref{T:Main2}.
\begin{theorem}[Equality of limits II]\label{T: limiting formula 2}
	Let $a\colon \N\to\Z$ be a sequence and $n_k\in \N_0$ be such that $(a(k!n+n_k))_{n,k}$ admits box seminorm control and  equidistributes on nilsystems.
	Then for every system $(X, \CX, \mu, T_1, T_2)$ and all functions $f_0,f_1, f_2\in L^\infty(\mu)$, we have the identity
		\begin{multline}\label{E: equality of limits'}
		\lim_{k\to\infty} 	\lim_{N\to\infty}\E_{n\in[N]}\int f_0\cdot T_1^{a(k!n+n_k)}f_1\cdot T_2^{a(k!n+n_k)}f_2 \, d\mu\\
        = \lim_{k\to\infty} \lim_{N\to\infty}\E_{n\in[N]}\int f_0\cdot
		T_1^{k!n}f_1\cdot T_2^{k!n}f_2 \, d\mu;
	\end{multline}
        in particular, both iterated limits above exist. {Moreover, an analogous result holds when, in both the hypothesis
        	and the conclusion, we replace the Ces\`aro average over $n\in\N$
        	by an average along an arbitrary F\o lner sequence in $\N$.}
\end{theorem}

We remark that all the previous results in this section remain valid, with no essential changes in the proofs, for multivariable sequences $a\colon \N^d \to \Z$, provided the corresponding definitions and averaging procedures are adjusted in the natural way.

Combining Theorem \ref{T: limiting formula} with \cite[Proposition 1]{H09} (restated as Proposition \ref{P: linear} below), we obtain the following corollary.
\begin{theorem}[Optimal seminorm control]\label{T: optimal control}
	Let $a\colon \N\to\Z$ be a sequence that admits box seminorm control and equidistributes on nilsystems. Then for every system $(X, \CX, \mu, T_1, T_2)$ and all 1-bounded functions $f_0,f_1, f_2\in L^\infty(\mu)$, we have
	\begin{align*}
		\lim_{N\to\infty}\abs{\E_{n\in[N]} \int f_0\cdot T_1^{a(n)}f_1\cdot T_2^{a(n)}f_2\; d\mu}\leq \min\bigbrac{\nnorm{f_0}_{T_1, T_2},\nnorm{f_1}_{T_1, T_2 T_1\inv},\nnorm{f_2}_{T_2 T_1\inv, T_2}}.
	\end{align*}
\end{theorem}

\subsection{Proof of \cref{T:Main1} assuming \cref{T: limiting formula}}\label{SS:pf1.1}
By \cref{T: limiting formula}, it suffices to verify that $(a(n))_n$ admits optimal box seminorm control and equidistributes on nilsystems. The first property follows from Theorem \ref{T: estimates for Hardy} while the second one is given by \cref{T: Hardy equidistribute}.

\subsection{Proof of \cref{T:Main2}  assuming \cref{T: limiting formula 2}}\label{SS:pf1.2} {We give the argument for Ces\`aro averages; the proof is identical
for averages along arbitrary F\o lner sequences in $\N$.}

By \cref{T: limiting formula 2}, it once again suffices to ascertain that $(a(k!n+n_k))_{n,k}$ admits box seminorm control and  equidistributes on nilsystems. The first of these properties follows from \cref{T: estimates for Hardy} by applying the seminorm control from the latter to the average
\begin{align*}
    \limsup_{N\to\infty}\abs{\E_{n\in[N]}\int f_0\cdot T_1^{a(k!n+n_k)}f_1\cdot T_2^{a(k!n+n_k)}f_2\; d\mu}
\end{align*}
separately for every $k\in\N$.

It remains to show that $(p(k!n+n_k))_{n,k}$ equidistributes
on nilsystems. Fix a nilsystem $(X,\CX,\mu,T)$. By the second remark following
\cref{D: good properties 2} we can assume that $x=e_X$. There exists $r_0\in \N$
such that the nilmanifold $\overline{(T^{r_0n}x)_n}$
is the connected component $X_0$ of the identity element $e_X$ in  $X$ (see for example \cite[Corollary~8, page 182]{HK18}). By \cite[Proposition~2.7]{Fr08},
for every nonconstant polynomial $q\in \Z[t]$, the sequence
$(T^{r_0q(n)}x)_n$ is equidistributed in $X_0$.
Let now $k\in \N$ with $k\geq r_0$, in which case we have
$r_0\mid p(n_k)$ and $r_0\mid k!$, and as a consequence, the polynomial
$q(n):=r_0^{-1}p(k!n+n_k)$ has integer coefficients.
We deduce that for $k\geq r_0$ the sequences
$(T^{p(k!n+n_k)}x)_n$ and $(T^{k!n}x)_n$
are both equidistributed in $X_0$. This implies that for every $k\geq r_0$ we have
$$
\lim_{N\to\infty} \E_{n\in [N]}F(T^{p(k!n+n_k)}x)=
\lim_{N\to\infty} \E_{n\in [N]}F(T^{k!n}x).
$$
Letting $k\to\infty$ gives \eqref{E:nk}.

\subsection{Proof strategy of Theorems~\ref{T: limiting formula} and \ref{T: limiting formula 2}} \label{SS:proofsketch}
We give a pretty detailed description of the proof strategy for \cref{T: limiting formula}. The proof of \cref{T: limiting formula 2} is similar up to a few additional technical details that are addressed in \cref{SS:lineark!}.

\smallskip
{\bf The single transformation case.}
Our starting point is the proof of \cref{T: limiting formula} for a single transformation, i.e., when $f_2=1$, in which case the result follows
from an application of the spectral theorem and from
the fact that the sequence $(a(n))_n$ equidistributes for
all circle rotations.

\smallskip
{\bf Optimal box seminorm control and magic extensions.}
For two transformations, our strategy amounts to establishing the \textit{optimal box seminorm control} for the averages
\begin{equation}\label{E:averages}
	\E_{n\in[N]}\int f_0\cdot T_1^{a(n)}f_1\cdot T_2^{a(n)}f_2\; d\mu;
\end{equation}
with respect to the function $f_2$, i.e., we show that if $\nnorm{f_2}_{T_2T_1^{-1},T_2}=0$, then the averages
\eqref{E:averages} vanish as $N\to\infty$. This is highly nontrivial and constitutes the main advance of this article. Once this is done, we use the machinery of magic extensions due to Host~\cite{H09} to reduce matters to the case where $f_2$ is a product of  a $T_2$-invariant function and  a $T_2T_1^{-1}$-invariant function. This immediately puts us in the single transformation case.
We explain next how to get optimal box seminorm control.

\smallskip
{\bf Known box seminorm control.} 
Our starting point is the assumption that the averages \eqref{E:averages}
admit box seminorm control, i.e., there exist $s_1,s_2\in \N$, possibly large, such that
if $\nnorm{f_2}_{(T_2T_1\inv)^{\times s_1},\, T_2^{\times s_2}}=0$, then
the averages \eqref{E:averages} vanish as $N\to\infty$. This property
holds for sequences covered in Theorems~\ref{T:Main1} and \ref{T:Main2}, a highly
nontrivial fact established in great pain within the last few years (see \cref{T: estimates for Hardy}).

Our next step is to perform a degree lowering argument that will eventually yield to optimal box seminorm control. Recently, an analogous degree reduction for box seminorms has been carried out quantitatively for polynomial corners in the finitary setting by Kravitz, Leng, and the second author \cite{KKL24b}. Our argument follows, in part, the overall philosophy of \cite{KKL24b}. However, some of the quantitative arguments from \cite{KKL24b} become significantly simpler in the ergodic setting (mostly those related to equidistribution on nilmanifolds) whereas certain easy steps in the finitary setting (such as obtaining inverse theorems for degree-3 box norms) require substantial effort in the ergodic universe.

\smallskip

{\bf Degree lowering.}
We explain how we reduce the degree $s_1+s_2$ of the seminorm in three representative test cases.

\smallskip

{\bf Case 1: $(s_1, s_2) = (1,2)$.}
Suppose that $\nnorm{f_2}_{T_2T_1\inv, T_2, T_2}=0$ implies
vanishing of the averages \eqref{E:averages} as $N\to\infty$. Our goal is to
show a similar property under the weaker hypothesis
$\nnorm{f_2}_{T_2T_1\inv , T_2}=0$. The crucial maneuver is
to pass to a suitable structured extension on which the following
inverse theorem holds:
\begin{equation}\label{E:invthm}
	f_2 \,\bot\, \bigbrac{\CI(T_2 T_1^{-1}) \vee \CZ_1(T_2)}
	\implies \nnorm{f_2}_{T_2T_1\inv , T_2, T_2}=0.
\end{equation}
For simplicity, assume that the original system
$(X,\CX,\mu,T_1,T_2)$ satisfies \eqref{E:invthm}. 
Given \eqref{E:invthm}, to analyze
\eqref{E:averages} we may assume that $f_2$ is a product of
a function in $I(T_2 T_1^{-1})$ and $\CE(T_2)$. In this
case, invoking our box seminorm control assumption again,
the averages \eqref{E:averages} vanish whenever
$\nnorm{f_1}_{T_1^{\times s}}=0$ and $s$ is  sufficiently large. Using this and the decomposition result
\cite[Proposition 3.1]{CFH11} (a nonergodic variant of the
Host-Kra structural result in \cref{T:HK}), we establish the identity
\eqref{E: equality of limits} (and hence the
optimal box seminorm control) as long as
$$
\lim_{N\to\infty}\E_{n\in [N]} \psi(a(n))=
\lim_{N\to\infty}\E_{n\in [N]} \psi(n)
$$
holds for every nilsequence $\psi$. The latter follows from our
assumption that $(a(n))_n$ equidistributes on nilsystems; in
the context of \cref{T:Main1} this follows from Theorem \ref{T: Hardy equidistribute}
(\cite[Theorem~1.2]{Fr09}). Details appear in \cref{SS:base} and
in Steps $1$-$3$ of \cref{SS: degree lowering}.

\smallskip
{\bf Case 2: $(s_1, s_2) = (2,1)$.}   Suppose that the averages \eqref{E:averages} vanish as $N\to\infty$ whenever $\nnorm{f_2}_{T_2T_1\inv , T_2T_1\inv, T_2}=0$; we want to obtain the same conclusion under the weaker assumption that $\nnorm{f_2}_{T_2T_1\inv , T_2}=0$. The difference from the previous case is that now, we want to lower the degree of $T_2T_1\inv$ in the seminorm controlling our average even though it is the transformation $T_2$ that acts on $f_2$.
This incompatibility can be fixed by composing the integral with $T_1^{-a(n)}$, which reparametrizes our average \eqref{E:averages} as
\begin{align*}
    	\E_{n\in[N]}\int f_1\cdot T_1^{-a(n)}f_0\cdot (T_2T_1\inv)^{a(n)}f_2\; d\mu.
\end{align*}
After this change of variables, we can simply invoke the previous case for the system $(X, \CX, \mu, T_1\inv, T_2T_1\inv)$. The need to perform this reparametrization is the main reason why we study the weak rather than strong limit of our averages.

\smallskip

{\bf Case 3: $(s_1, s_2) = (2,2)$.}
Suppose that $\nnorm{f_2}_{T_2T_1\inv , T_2T_1\inv, T_2, T_2}=0$
implies vanishing of the averages \eqref{E:averages} in the limit. Our goal is
to show a similar property under the weaker hypothesis
$\nnorm{f_2}_{T_2T_1\inv , T_2T_1\inv,T_2}=0$ (which can be further
weakened to $\nnorm{f_2}_{T_2T_1\inv , T_2}=0$ by Case $2$). Again, upon passing to a structured extension of the system
we can assume that \eqref{E:invthm} holds. This time, however, the
deduction is more intricate and we rely on the degree lowering
argument pioneered by Peluse and Prendiville~\cite{P19a,P19b,PP19} in the finitary universe, and specifically on its version due to Kravitz, Leng, and the second author \cite{KKL24b}, which we adapt
to our ergodic setting (much like the arguments in \cite{Fr21,FrKu22a}). Suppose that
\begin{equation}\label{E:positiveT12}
	\limsup_{N\to\infty}\Big|\E_{n\in[N]}\int \, f_0\cdot T_1^{a(n)}f_1\cdot
	T_2^{a(n)}f_2\; d\mu\Big|>0,
\end{equation}
then our goal is to show that
$\nnorm{f_2}_{T_2T_1\inv , T_2T_1\inv,T_2}>0$.

\smallskip

{\bf Step 1: Degree lowering (incomplete).}
 Using simple
maneuvers, we deduce from \eqref{E:positiveT12} that
$$
\limsup_{N\to\infty}\abs{\E_{n\in[N]} \int f_0\cdot T_1^{a(n)}{f}_1
	\cdot T_2^{a(n)}F_2\; d\mu} > 0,
$$
where
$$
F_2 := \lim_{N\to\infty}\E_{n\in[N]}T_2^{-a(n)}\overline{f}_0\cdot
(T_1T_2\inv)^{a(n)}\overline{f}_1.
$$
For the sake of exposition, we assume that the limit above exists in
$L^2(\mu)$; since we do not know this to hold in general, we will work with a weak subsequential limit later on.
Our standing box seminorm assumption gives
$$
\nnorm{F_2}_{T_2T_1\inv , T_2T_1\inv, T_2, T_2} > 0.
$$
Using the inductive definition of these seminorms given in
\eqref{ergodic identity}, and a soft quantitative variant of
\eqref{E:invthm} as in \cref{P: inverse theorem} (which is
nontrivial to deduce but follows from a general principle
developed in \cite{FrKu22a}), we deduce that there exist
$g_h\in I(T_2T_1^{-1})$ and $\chi_h\in \CE(T_2)$ such that the
integrals below are real and
$$
\liminf_{H\to\infty} \E_{h\in [H]} \int (T_2T_1^{-1})^hF_2\cdot
\overline{F}_2\cdot g_h\cdot \chi_h\, d\mu>0.
$$
Using some Cauchy-Schwarz maneuvering (known as ``dual-difference
interchange'') we deduce that
$$
\liminf_{H\to\infty} \E_{h.h'\in [H]} \limsup_{N\to\infty}\Big|
\E_{n\in[N]} \int f_{0,h,h'}\cdot T_1^{a(n)}f_{1,h,h'}\cdot
T_2^{a(n)}(g_{h,h'}\cdot \chi_{h,h'})\, d\mu \Big|>0,
$$
where for $h,h'\in\N$ and $j=0,1$, we consider 1-bounded functions
$$
f_{j,h,h'}:=\Delta_{T_2T_1^{-1}; {h-h'}} f_j,
$$
$g_{h,h'}\in I(T_2T_1^{-1})$, and $\chi_{h,h'}\in \CE(T_2)$. The advantage now is that the
functions associated to the transformation $T_2$ have very
special form, so working with the innermost average and arguing
as in the second part of Case 1, we get that the
limit of this average remains unchanged if we replace $a(n)$ by
$n$. Upon doing that, we can use the elementary estimate of Host \cite{H09} (see
\cref{P: linear}) to deduce
$$
\liminf_{H\to\infty} \E_{h.h'\in [H]}
\nnorm{f_{1,h,h'}}_{T_1,T_1T_2^{-1}}>0.
$$
(We get this with $f_{1,h,h'}\cdot g_{h,h'}$ in place of
$ f_{1,h,h'}$ but the functions $g_{h,h'}$ can easily be
disposed.) Using the inductive definition of the seminorms in
\eqref{ergodic identity}, it is easy to infer from this that
$$
\nnorm{f_1}_{T_1,T_1T_2^{-1},T_1T_2^{-1}}>0.
$$
 Details for this step appear
in Steps $4$-$7$ of \cref{SS: degree lowering}.

At this point we note that although we have produced a positivity
property for a box seminorm of smaller complexity than the one we started with, it is not exactly of the form we wanted. It involves the function $f_1$ (instead of $f_2$) and the role of the transformations $T_1$ and $T_2T_1^{-1}$ is reversed.
Although in this case we can use a reparametrization trick that
allows us to overcome this problem, such a trick does not
generalize to the case $s_1,s_2\geq 3$ (see the discussion at
the beginning of Step 8 in \cref{SS: degree lowering}).

\smallskip

{\bf Step 2: Seminorm smoothing.}
 To alleviate the problem just mentioned, we use a smoothing technique,
 which amounts to maneuvering similar to Step 1, but now we pass  to a magic
  extension and use an inverse theorem for the simpler
 box seminorms $\nnorm{\cdot}_{T_1,T_1T_2^{-1}}$, in the
 soft quantitative form of \cref{P: inverse theorem}. This move
 will not yield positivity for seminorms of even lower degree
 (that was achieved in Step $1$), but it will ``smooth''
 them, producing the needed positivity
 $$
 \nnorm{f_2}_{T_2T_1\inv, T_2T_1\inv,T_2}>0.
 $$
 Details for this step appear in Steps $8$-$9$ of
 \cref{SS: degree lowering} (see the Interlude for a demonstration
 of this maneuver in the simplest possible setting).

\smallskip

{\bf Concluding the argument.}
For general $s_1,s_2\in \N$, an argument very similar to
the one described in the case $(s_1,s_2)=(2,2)$ enables us to reduce seminorm control from 
$\nnorm{f_2}_{(T_2T_1\inv)^{\times s_1}, T_2^{\times s_2}}$  to one in terms of
$\nnorm{f_2}_{(T_2T_1\inv)^{\times s_1-1}, T_2^{\times s_2}}$ and
$\nnorm{f_2}_{(T_2T_1\inv)^{\times s_1}, T_2^{\times s_2-1}}$,
as long as $s_1\geq 2$ and $s_2\geq 2$ respectively. Applying this
reduction $s_1+s_2-2$ times we conclude with the strived-for control in terms of $\nnorm{f_2}_{T_2T_1\inv , T_2}$. Thus, we have effectively reduced matters to the case where the
magic extension machinery of Host is applicable, and we conclude the
proof of  \cref{T: limiting formula 2}  using the standard argument described in \cref{SS:end}.

\smallskip

{\bf Existence of structured extensions.}
Lastly, we briefly describe the proof strategy for establishing the existence
of an extension on which \eqref{E:invthm} holds; the details of this argument appear in \cref{S: structured extension}. The maneuver of passing to a structured extension in order to facilitate the proof of mean convergence results for multiple ergodic averages has long history. Originating in the work of Furstenberg and Weiss \cite{FW96}, it has later been refined e.g. by Austin \cite{ Au09, Au15a, Au15b}, Host \cite{H09}, Leng 
\cite{Leng25}, and others. Of particular relevance to us are the works of the last three authors. Austin \cite{Au09} showed that on a suitable extension, the characteristic factors for various averages under consideration admit a simpler description; moreover, such a description cannot in general be attained on the original system. Host \cite{H09} then provided a simpler alternative to Austin's extension, introducing the notion of magic extensions. Lastly, building on ideas of Tao \cite{Tao15}, Leng \cite{Leng25} used a finitary inverse theorem for some multidimensional box norms to construct yet a different type of extensions that served as a direct inspiration for our construction.

Our goal thus is to prove that
 every system $(X,\CX,\mu,T_1,T_2)$ can be extended to a system
$(Y, \CY, \nu,S_1,S_2)$ on which
$\CZ({S_1,S_2,S_2})=\CI(S_1)\vee \CZ_1(S_2)$; then \eqref{E:invthm} will follow from \eqref{E: factor property}. Unlike
Host's proof of the existence of magic extensions \cite{H09}, which yields a similar
result  for the seminorms $\nnorm{f_2}_{T_1,T_2}$, we did not manage
to give a proof entirely within ergodic theory. Instead, we rely on a
method pioneered by Tao~\cite{Tao15} and recently utilized by Leng~\cite{Leng25} in a
setting closer to ours. The first step is to find a corresponding inverse
theorem in the finitary world (see \cref{P: inverse theorem}); somewhat
surprisingly, this is straightforward and is the main reason for the
method's effectiveness. From this we deduce a decomposition result for
sequences $f\colon [\pm N]^2\to \C$ (see \cref{P: regularity lemma}) as a sum
of structured, box seminorm uniform, and negligible components, which then implies that
finitary variants of dual sequences can be well approximated by these
structured components (see \cref{P: structured & anti-uniform}). We use
this to show that a countable dense set $\CD\subseteq L^2(\CZ)$ of dual
functions, such as
$$
\CD f:=\lim_{N\to\infty}\E_{h_1,h_2\in [N]} T_1^{h_1}\overline{f}\cdot
T_2^{h_2}\overline{f}\cdot T_1^{h_1}T_2^{h_2}f,
$$
where $f\in L^\infty(\mu)$ and the limit is taken in $L^2(\mu)$, can be
approximated pointwise on all finite scales by these structured sequences.
We then use these structured sequences to build a system  $(\widetilde X,\widetilde\CX,\widetilde\mu,\widetilde T_1,\widetilde T_2)$ in which all
correlations of functions in $\CD$ are reproduced; this makes it an
extension of $(X,\CZ,\mu,T_1,T_2)$. The advantage of the extended system
is that, because it is built from simple structured components, it is
straightforward to show that it satisfies $\CZ({\widetilde T_1,\widetilde T_2,\widetilde T_2})=\CI(\widetilde T_1)\vee \CZ_1(\widetilde T_2)$, as required.
To give a better sense of the argument, note that the structured
components are restrictions to $[\pm N]^2$ of linear combinations of sequences
of the form
$$
a(m,n):=e(\phi(m)n)\cdot b(m)\cdot c(n) , \quad m,n\in \N,
$$
where $\phi\colon\N\to \C$ is arbitrary and $b,c\colon \N\to \C$ are
$1$-bounded. It is not hard to verify that the Furstenberg system (i.e.,
the measure preserving system  $(X:=\mathbb{D}^{\Z^2},\mu,T_1,T_2)$ associated to
$(a(m,n))_{m,n}$ via the Furstenberg correspondence principle) of these infinite sequences is indeed spanned by
$\CI(T_1)\vee \CZ_1(T_2)$. The system $(\widetilde X,\widetilde\CX,\widetilde\mu,\widetilde T_1,\widetilde T_2)$ we build is the
joint Furstenberg system of infinitely many such
sequences. {A minor nuisance is that this system extends the factor $(X,\CZ,\mu,T_1,T_2)$, not the full original system. However, by taking an infinite tower of such extensions following Leng \cite[Appendix C]{Leng25}, we obtain a system $(Y, \CY, \nu, S_1, S_2)$ which both extends $(X,\CX,\mu,T_1,T_2)$ and satisfies \eqref{E:invthm}.} 

\subsection{Obstructions for longer averages and mean convergence}\label{SS: issues}
Having explained the proof strategy for Theorem \ref{T:Main1}, we move on to explain where the argument above fails if we aim to establish the (highly plausible) $L^2(\mu)$ identity
\begin{align*}
    \lim_{N\to\infty}\norm{\E_{n\in[N]} T_1^{a(n)}f_1\cdot T_2^{a(n)}f_2-\E_{n\in[N]} T_1^nf_1\cdot T_2^nf_2}_{L^2(\mu)} = 0.
\end{align*}
The argument presented in Case 1, which reduces the seminorm control by $\nnorm{f_2}_{T_2T_1\inv, T_2, T_2}$ to the seminorm control by $\nnorm{f_2}_{T_2T_1\inv, T_2}$, can be adapted in a straightforward way to the mean convergence setting, i.e., to get good seminorm control for the limit
\begin{equation}\label{E:averages 2}
	\limsup_{N\to\infty}\norm{\E_{n\in[N]} T_1^{a(n)}f_1\cdot T_2^{a(n)}f_2}_{L^2(\mu)}.
\end{equation}
But this no longer holds for Case 2, in which the goal is to derive the seminorm control in terms of $\nnorm{f_2}_{T_2T_1\inv, T_2T_1\inv, T_2}$ from one by $\nnorm{f_2}_{T_2T_1\inv, T_2}$. We recall that the reduction of Case 2 to Case 1 proceeds by composing the integral with $T_1^{-a(n)}$, a trick that can only be applied in the setting of weak convergence.\footnote{One could try to circumvent this by recasting the square of \eqref{E:averages 2} as $\limsup\limits_{N\to\infty}\E\limits_{n\in[N]}\int f_{0,N}\cdot T_1^{a(n)}f_1\cdot T_2^{a(n)}f_2\; d\mu;$ however, it is essential to our approach that $f_{0,N}$ be independent of $N$ (since we need intermediate seminorm estimates in terms of $f_0$), which we cannot ensure without knowing the $L^2(\mu)$ convergence of the average in \eqref{E:averages 2}.
}

Hence, Case 2 for \eqref{E:averages 2} is genuinely different than Case 1, unlike in the weak-limit realm. One therefore needs to attack Case 2 head-on. By passing to a structured extension of the factor $\CZ(T_2T_1\inv, T_2T_1\inv, T_2)$, we can reduce the entire complexity of Case 2 to the situation where $f_2$ is a {nonergodic eigenfunction} of $T_2T_1\inv$, i.e., $T_2T_1\inv f_2 = \lambda\,  f_2$ for some $\lambda\in I(T_2T_1\inv)$. Plugging this formula back to \eqref{E:averages 2}, we end up with the $L^2(\mu)$ norm of
\begin{align}\label{E:averages 3}
    \E_{n\in[N]} T_1^{a(n)}(f_1f_2)\cdot (T_1^{a(n)}\lambda)^{a(n)}.
\end{align}
It is highly unclear how to deal with this expression, not least because it is not linear in $\lambda$ and $\lambda$ may not be constant. If $a(n) = n$, then an argument of Lesigne (private communication) shows that \eqref{E:averages 3} vanishes unless $\lambda$ is constant on a positive-measure set. For nonlinear sequences such as $n^2$ or $[n^{3/2}]$, the likely conclusion one could derive is that \eqref{E:averages 3} is nonzero only if $\lambda$ agrees on a positive-measure set with a function in $Z_s(T_1)$ for some $s\in\N$. However, it is neither clear how to reach this conclusion for general polynomials or Hardy sequences, nor how to apply it. The nonlinearity of \eqref{E:averages 3} in $\lambda$ makes it difficult to follow the instinct and perform the standard decomposition of $\lambda$ into structured, uniform, and small components.

Even more issues arise if we aim to adapt our arguments to more transformations. In the process of running the degree lowering argument for the triple average
\begin{equation}\label{E:averages 4}
	\E_{n\in[N]}\int f_0\cdot T_1^{a(n)}f_1\cdot T_2^{a(n)}f_2\cdot T_3^{a(n)}f_3\; d\mu,
\end{equation}
one would reduce the problem to the scenario when $f_2\in\CE(T_2)\cup\CE(T_2T_1\inv)$ whereas $f_3\in\CE(T_3)\cup\CE(T_3T_1\inv)\cup\CE(T_3T_2\inv)$. If, say, $f_2\in\CE(T_2)$ and $f_3\in\CE(T_3T_1\inv)$, then we run again into the problem of incompatibility: $f_3$ is acted on by $T_3$ despite being a nonergodic eigenfunction of $T_3T_1\inv$. If in turn we compose \eqref{E:averages 4} with $T_1^{-a(n)}$, getting
\begin{align*}
    \E_{n\in[N]}\int f_1\cdot T_1^{-a(n)}f_0\cdot (T_2T_1\inv)^{a(n)}f_2\cdot (T_3T_1\inv)^{a(n)}f_3\; d\mu,
\end{align*}
then $f_3$ is compatible with the transformation that acts on it, but this is no longer the case for $f_2$. It is clear that there is no reparametrization of the integral that makes all the eigenfunctions compatible, and we currently do not know how to handle this problem.

In summary, while our approach opens a path toward resolving
Conjecture~\ref{Con1} and Conjecture~\ref{Con2} (when all polynomials coincide),
formidable obstacles remain for the case of $\ell\geq 3$ transformations or mean convergence analogs of Theorems \ref{T:Main1} and \ref{T:Main2}.

\section{Structured extensions}\label{S: structured extension}
 Let $(X, \CX, \mu, T_1, \ldots, T_\ell)$ be an ergodic system. In this section, we construct its structured extension on which a certain box factor of interest takes a particularly pleasing form.
Specifically, we show the following result, which we will only use for $\ell=2$ in this article.
\begin{theorem}[Existence of a structured extension]\label{T: structured extension}
    Let $(X, \CX, \mu, T_1, \ldots, T_\ell)$ be an ergodic system. Then it admits an ergodic extension $(Y, \CY, \nu, S_1, \ldots, S_\ell)$ satisfying
    \begin{align}\label{E: structural property}
        \CZ(S_1, \ldots, S_\ell, S_\ell) = \CI(S_1)\vee \cdots \vee \CI(S_{\ell-1})\vee\CZ_1(S_{\ell}).
    \end{align}
\end{theorem}
The proof of Theorem \ref{T: structured extension} consists of several steps of vastly differing complexity:
\begin{enumerate}
    \item First, we derive a (pretty straightforward) inverse theorem for a certain finitary box norm analogous to $\nnorm{\cdot}_{S_1, \ldots, S_\ell, S_\ell}$ for compactly supported functions on $\Z^\ell$ (\cref{SS: finitary box norms}).
    \item Second, we use a variant of the arithmetic regularity lemma to obtain a decomposition of an arbitrary function on $\Z^\ell$ into terms that are ``structured'' and ``uniform'' with respect to the aforementioned finitary box norm as well as a small error term (\cref{SS: ARL}).
    \item Third (this is where the bulk of the work goes), we construct an ergodic extension $(Y, \CY, \nu, S_1, \ldots, S_\ell)$ of the factor 
    \begin{align*}
        \CZ := \CZ(T_1, \ldots, T_\ell, T_\ell)
    \end{align*}
    in such a way that any $\CZ$-measurable function on $X$ lifts to a function in \eqref{E: structural property} (\cref{SS: main construction}).
    \item Fourth, we upgrade the construction from an extension of the factor to the extension of the full system (\cref{SS: extension}).
\end{enumerate}
Our argument is heavily inspired by a blog post of Tao \cite{Tao15} and recent work of Leng \cite[Section 9]{Leng25}. 
It would be interesting to see if a structured extension satisfying \cref{T: structured extension} can be constructed by purely ergodic means, in a finite number of steps and without resorting to finitary tools, in a manner similar to Host's construction of magic systems \cite{H09}.

The utility of working inside the structured extension  of \cref{T: structured extension} comes from the following inverse theorem
that will play a key role in our degree lowering argument later on. Its proof is very close to the proof of Proposition \ref{P: inverse theorem magic} and so we skip it.
\begin{proposition}[Soft quantitative inverse theorem]\label{P: inverse theorem}
    Let $(X, \CX, \mu, T_1, \ldots, T_\ell)$ be a system with the property that
    \begin{align*}
        \CZ(T_1, \ldots, T_\ell, T_\ell) = \CI(T_1)\vee\cdots\vee\CI(T_{\ell-1})\vee\CZ_1(T_\ell).
    \end{align*}
    Then for any $\veps>0$ there exists $\delta>0$ (depending on $\veps$ and the system) such that if a 1-bounded function $f\in L^\infty(\mu)$ satisfies
    $$
    \nnorm{f}_{T_1, \ldots, T_\ell, T_\ell}\geq \veps,
    $$
    then there exist 1-bounded functions $g_1\in I(T_1), \ldots, g_{\ell-1}\in I(T_{\ell-1}), \chi\in\CE(T_\ell)$ for which the integral below is real and
    \begin{align*}
        \int f\cdot g_1\cdots g_{\ell-1}\cdot\chi\; d\mu \geq \delta.
    \end{align*}
\end{proposition}

\subsection{Defining factor maps from correlation identities}
To prove \cref{T: structured extension}, we plan to construct
a system $(Y, \CY, \nu, S_1, \ldots, S_\ell)$ and
a collection of $L^\infty(\nu)$ functions that imitate
the correlations of a  family of $L^\infty(\mu)$ functions that generate a dense subalgebra of $L^2(\CZ,\mu)$. We then need a
result that allows us to extract a factor map
from this property, and we carry out this standard
preliminary step in this subsection. This will give an extension of the factor $(X, \CZ, \mu, T_1, \ldots, T_\ell)$, and   with additional maneuvering, it will produce an extension of the original system. 
Recall that all systems in this article are assumed to be regular, and that $T^{h}=T_1^{h_1}\cdots T_\ell^{h_\ell}$ for every $h=(h_1,\ldots, h_\ell)\in \Z^\ell$. 

 We will use the following variant of a well-known result in operator theory
 (see for example \cite[Theorem 12.15]{EFHN15} that covers the case  $F=L^\infty(\mu)$).
 \begin{lemma}\label{L:Phi}
 		Let $(X, \CX, \mu, T_1,\ldots,T_\ell)$ and $(Y, \CY, \nu, S_1,\ldots, S_\ell)$ be systems and  $F \subseteq L^\infty(\mu)$ be  a conjugation-closed subalgebra that contains $1$,
 	 is invariant under $T_1, \ldots, T_\ell$,  and is
 		 dense in $L^2(\mu)$. Furthermore, let $\Phi \colon F \to L^\infty(\nu)$
 		 be a map that satisfies
\begin{enumerate}
	\item \label{I:E1} $\Phi$ is an algebra  homomorphism on $F$ and $\Phi(\overline{f})=\overline{\Phi (f)}$ for all $f\in F$;
	
	\item  \label{I:E2}	$\int \Phi(f)\, d\nu=\int f\, d\mu$ for all  $f\in F$;

	\item   \label{I:E3} $\Phi(T^hf)=S^h(\Phi(f))$  for all  $f\in F$ and  $h\in \Z^\ell$.
\end{enumerate}
Then  there exists a factor map
$\phi\colon Y\to X$, i.e., $\phi_*\nu=\mu$, $\phi\circ S^h=T^h\circ \phi$ for all $h\in \Z^\ell$.
 \end{lemma}
\begin{proof}
The argument is rather standard, so we only sketch it.

Our assumptions give that $\Phi$ is linear and $\norm{\Phi(f)}_{L^2(\nu)}=\norm{f}_{L^2(\mu)}$ for $f\in F$, i.e., $\Phi$ is a linear isometry on a dense subspace of $L^2(\mu)$ and hence extends uniquely to a linear isometry $\Phi\colon L^2(\mu)\to L^2(\nu)$. The goal is to show that there exists a measure-preserving
measure-algebra $\sigma$-homomorphism $\alpha\colon \CX\to \CY$~\footnote{I.e., $\nu(\alpha(A))=\mu(A)$,  $\alpha(A^c) =\alpha(A)^c$, for all $A\in \CX$, and  $\alpha\Bigbrac{\bigcup\limits_{i\in \N}A_i}=\bigcup\limits_{i\in \N}\, \alpha(A_i)$ for all $A_i\in \CX$.} with $\alpha(T^{-h}A)=S^{-h}\alpha(A)$ for all $A\in \CX$ and $h\in \Z^\ell$ and such that $\Phi({\bf 1}_A)={\bf 1}_{\alpha(A)}$ for all $A\in \CX$. Since we
are working with standard probability spaces, this then implies (see for example  \cite[Theorem~12.14]{EFHN15} or \cite[Theorem~2.2]{Wa82}) the existence of a point map $\phi\colon Y\to X$ that satisfies the required properties ($\phi$ is related to $\alpha$ via $\alpha(A)=\phi^{-1}(A)$ for $A\in \CX$).

Let $A\in \CX$. Since $F \subseteq L^\infty(\mu)$ is dense in $L^2(\mu)$ and closed under conjugation, there 
exist $u_n\in F$ taking values in $[-n,n]$ and such that $u_n\to {\bf 1}_A$ in $L^2(\mu)$. 
If  $\psi\colon \R\to [0,1]$ is given by $\psi(t):= \min\{t\cdot {\bf 1}_{[0,+\infty)}(t),1\}$,
using  the Weierstrass approximation theorem, we get that  there exist 
  Bernstein polynomials $p_n\colon [-n,n]\to [0,1]$ 
 such that 
$ \norm{\psi-p_n}_{L^\infty[-n,n]}\to 0$. 
So if for  $n\in \N$,  we let
$f_n:=p_n(u_n)$, then 
 $0\leq f_n\leq 1$, $f_n\in F$ (since $F$ is  an algebra that contains $1$), and 
 \begin{equation}\label{E:uniform}
	\lim_{n\to\infty} \norm{\psi(u_n)-f_n}_{L^\infty(\mu)}=0.
\end{equation} 
  Furthermore,  since  $\psi({\bf 1}_A)={\bf 1}_A$ and $\psi$ has Lipschitz constant $1$, we have 
$$
\norm{\psi(u_n)-{\bf 1}_A}_{L^2(\mu)} =
\norm{\psi(u_n)-\psi({\bf 1}_A)}_{L^2(\mu)} 
\leq \norm{u_n-{\bf 1}_A}_{L^2(\mu)} \to 0.
	$$ 
Combining this with \eqref{E:uniform}, we deduce that $f_n\to {\bf 1}_A$ in $L^2(\mu)$, and since 
$0\leq f_n\leq 1$, this implies that $f_n^2\to {\bf 1}_A$ in $L^2(\mu)$. From this and the
multiplicativity of the linear isometry $\Phi$ on $F$, we deduce that
$\norm{(\Phi({\bf 1}_A))^2-\Phi({\bf 1}_A)}_{L^1(\nu)}=0$, hence,
$(\Phi({\bf 1}_A))^2=\Phi({\bf 1}_A)$. As a consequence, there exists $\alpha\colon \CX\to \CY$ such that  $\Phi({\bf 1}_A)={\bf 1}_{\alpha(A)}$ for all $A\in \CX$.

 Using a similar argument, we get that   for all $A,B\in \CX$ we have 
$$\Phi({\bf 1}_{A\cap B})={\bf 1}_{\alpha(A)\cap \alpha(B)} \quad \text{and} \quad  \Phi({\bf 1}_{A^c})={\bf 1}_{\alpha(A)^c},
$$ showing that $\alpha$ is a measure-algebra homomorphism. Using the monotone convergence theorem, we deduce that $\alpha$ is a $\sigma$-homomorphism. In a similar
fashion, we show that $\Phi({\bf 1}_{T^{-h}A})={\bf 1}_{S^{-h}\alpha(A)}$, hence
$\alpha(T^{-h}A)=S^{-h}\alpha(A)$ for all $A\in \CX$ and $h\in \Z^\ell$. Lastly, 
$$
\nu(\alpha(A))=\norm{\Phi({\bf 1}_A)}_{L^2(\nu)}=\norm{{\bf 1}_A}_{L^2(\mu)}=\mu(A),
$$
for all $A\in \CX$.  This  completes  the proof of the asserted properties.
\end{proof}

\begin{lemma}\label{L:factor}
	Let $(X, \CX, \mu, T_1,\ldots,T_\ell)$ and $(Y, \CY, \nu, S_1,\ldots, S_\ell)$ be systems and $(f_n)_n\subseteq L^\infty(\mu)$ a collection of functions, closed under conjugation,  that generates a dense subalgebra of $L^2(\mu)$.
	Suppose that there exist  functions
	$(g_n)_n\subseteq L^\infty(\nu)$  such that
	\begin{equation}\label{E:fg}
		\int T^{h_1}f_{n_1}\cdots T^{h_r}f_{n_r}\, d\mu=
		\int S^{h_1}g_{n_1}\cdots S^{h_r}g_{n_r}\, d\nu
	\end{equation}
	for all $r\in \N$, $h_1,\ldots, h_r\in \Z^\ell$, and $n_1,\ldots, n_r\in \N$.  Then
	$(X, \CX, \mu, T_1,\ldots, T_\ell)$ is a factor of  $(Y, \CY, \nu, S_1,\ldots, S_\ell)$.
\end{lemma}
\begin{proof}
	 We let
	 $F$ be the $T_1,\ldots, T_\ell$-invariant subalgebra generated by $(f_n)_n\cup \{1\}$ and define
	$\Phi\colon F\to L^\infty(\nu)$ via $\Phi(1):=1$ and 
	\begin{equation}\label{E:algebraid}
		\Phi\Big(	\sum_{i=1}^l c_i\, T^{h_{i,1}}f_{n_{i,1}}\cdots T^{h_{i,r}}f_{n_{i,r}}\Big):=
		\sum_{i=1}^l c_i\,  S^{h_{i,1}}g_{n_{i,1}}\cdots S^{h_{i,r}}g_{n_{i,r}},
	\end{equation}
	for all
	$l\in \N$, $h_{i,1}\in \Z^\ell$, $n_{i,1}\in \N$, $c_i\in \C$.  We claim that $\Phi$ is well defined.
	To this end, since $\Phi$ is linear, 	it suffices to  establish the implication
	$$
	\sum_{i=1}^l c_i\, T^{h_{i,1}}f_{n_{i,1}}\cdots T^{h_{i,r}}f_{n_{i,r}}=0 \,   \implies
	\sum_{i=1}^l c_i\,  S^{h_{i,1}}g_{n_{i,1}}\cdots S^{h_{i,r}}g_{n_{i,r}}=0,
	$$
	where
	$l\in \N$, $h_{i,1}\in \Z^\ell$, $n_{i,1}\in \N$, $c_i\in \C$,  the first identity holds $\mu$-a.e. and the second $\nu$-a.e.  Equivalently, it suffices to show    that
	$$
	\int	\Bigabs{\sum_{i=1}^l c_i\, T^{h_{i,1}}f_{n_{i,1}}\cdots T^{h_{i,r}}f_{n_{i,r}}}^2\, d\mu =0 \,   \implies
	\int \Bigabs{\sum_{i=1}^l c_i\,  S^{h_{i,1}}g_{n_{i,1}}\cdots S^{h_{i,r}}g_{n_{i,r}}}^2\, d\nu=0.
	$$
	This is a consequence of the identity
	\begin{equation}\label{E:isometry}
		\int	\Bigabs{\sum_{i=1}^l c_i\, T^{h_{i,1}}f_{n_{i,1}}\cdots T^{h_{i,r}}f_{n_{i,r}}}^2\, d\mu =
		\int \Bigabs{\sum_{i=1}^l c_i\,  S^{h_{i,1}}g_{n_{i,1}}\cdots S^{h_{i,r}}g_{n_{i,r}}}^2\, d\nu,
	\end{equation}
	which can be easily established by 	 expanding the squares on the left and right  hand side and using \eqref{E:fg} (we also used here that the set $F$ is closed under conjugation).

	Using \eqref{E:algebraid} we easily get that Properties \eqref{I:E1}-\eqref{I:E3} of \cref{L:Phi} are satisfied, and we deduce that there exists a factor map $\phi\colon Y\to X$, as desired. 	
\end{proof}

\subsection{Finitary box norms}\label{SS: finitary box norms}
We now switch gears to the finitary setting in order to prove a finitary inverse theorem that will later play a key part in the proof of Theorem \ref{T: structured extension}.
Let $f\colon \Z^\ell\to\C$ be finitely supported. For $h,h'\in\Z$, we define its \textit{(symmetric and asymmetric) multiplicative derivatives} to be
\begin{align*}
    \Delta'_{(h,h')}f(x):=f(x+h)\, \overline{f}(x+h')\quad \textrm{and}\quad \Delta_h f(x) := f(x)\, \overline{f}(x+h),
\end{align*}
and then construct higher-degree derivatives via
\begin{align*}
    \Delta'_{(h_1,h_1'), \ldots, (h_s,h_s')}f:= \Delta_{(h_1,h_1')}'\cdots\Delta_{(h_s,h_s')}'f\quad \textrm{and}\quad \Delta_{h_1, \ldots, h_s}f := \Delta_{h_1}\cdots\Delta_{h_s}f.
\end{align*}
Given finite sets $E_1, \ldots, E_s\subseteq\Z^\ell$, we define the \textit{box norm} of $f$ along $E_1, \ldots, E_s$ to be
\begin{align*}
    \norm{f}_{E_1, \ldots, E_s} &:= \Big(\E_{h_1,h_1'\in E_1}\cdots \E_{h_s,h_s'\in E_s} \sum_{x\in\Z^\ell}\Delta'_{(h_1,h'_1), \ldots, (h_s,h'_s)}f(x)\Big)^{\frac{1}{2^s}}\\
    &=\Big(\sum_{m_1, \ldots, m_s\in\Z^\ell}\mu_{E_1, \ldots, E_s}(m) \sum_{x\in\Z^\ell}\Delta_{m_1, \ldots, m_s}f(x)\Big)^{\frac{1}{2^s}},
\end{align*}
where
\begin{align}\label{E: weight}
    \mu_{E_1, \ldots, E_s}(m) :=  \prod_{j=1}^s \E_{h_j,h'_j\in E_j}\textbf{1}_{m_j = h_j'-h_j}
\end{align}
for any $m=(m_1, \ldots, m_s)\in\Z^{\ell s}$.
We will need the following two finitary inverse theorems. The first one plays a merely auxiliary role and can be found e.g. in \cite[Lemma A.5]{KKL24b}.
\begin{lemma}[Finitary $U^2$ inverse theorem]\label{L: U^2 inverse}
    Let $\delta>0$, $N\in\N$, and let $f\colon\Z^2\to\C$ be 1-bounded and supported on $[\pm N]^2$. Then there exist $\phi, \psi\colon \Z\to\T$ such that
    \begin{align*}
        \norm{f}_{e_1[\pm N], e_1[\pm N]}^4\geq \delta N^2 \quad \Longrightarrow \quad \sum_{x,y\in\Z}f(x,y)\, e(\phi(y)x+\psi(y)) \gg \delta^{3/2}N^2.
    \end{align*}
\end{lemma}
The second one is what we really need; it is a straightforward generalization of \cite[Lemma A.8]{KKL24b}.
\begin{proposition}[Finitary box norm inverse theorem]\label{P: finitary inverse theorem}
    Let $\delta>0$, $N\in\N$, and let $f\colon\Z^\ell\to\C$ be 1-bounded and supported on $[\pm N]^\ell$. Then there exist $\phi\colon\Z^{\ell-1}\to\T$ and 1-bounded $b_1, \ldots, b_\ell\colon\Z^{\ell-1}\to\C$ supported on $[\pm N]^\ell$ such that
    \begin{align*}
        \norm{f}_{e_1[\pm N], \ldots, e_\ell[\pm N], e_\ell[\pm N]}^{2^{\ell+1}}\geq \delta N^\ell \quad \Longrightarrow \quad \sum_{x\in\Z^\ell}f(x)\, e(\phi(\hat x_\ell)x_\ell)\, b_1(\hat x_1)\cdots b_\ell(\hat x_\ell) \gg_\ell \delta^{3/2}N^\ell,
    \end{align*}
    where $\hat x_i := (x_1, \ldots, x_{i-1}, x_{i+1}, \ldots, x_\ell)$.
\end{proposition}
\begin{proof}
By the inductive formula for the box norms (e.g. \cite[Lemma 3.5]{KKL24a}), we have
\begin{align*}
    \sum_{m\in\Z^{\ell-1}}\mu_N(m)\norm{\Delta_{m_1, \ldots, m_{\ell-1}}f}_{e_\ell[\pm N], e_\ell[\pm N]}^4 = \delta N^\ell,
\end{align*}
where $\mu_N := \mu_{\underbrace{[\pm N]^\ell, \ldots, [\pm N]^\ell}_{\ell-1}}$ as in \eqref{E: weight}.
    Applying Lemma \ref{L: U^2 inverse} to each multiplicative derivative, we get that
    \begin{align*}
        \sum_{m\in\Z^{\ell-1}}\mu_N(m) \sum_{x\in\Z^\ell}\Delta_{m_1, \ldots, m_{\ell-1}}f(x)\, e(\phi_m(\hat x_\ell)x_\ell + \psi_m(\hat x_\ell)) \gg \delta^{3/2}N^\ell
    \end{align*}
     for some $\phi_m, \psi_m\colon \Z^{\ell-1}\to\T$.

    By modifying the phase function $\psi_{m}$ appropriately, we can assume that
    \begin{align}\label{E: positivity fin}
        \sum_{x\in\Z^\ell}\Delta_{m_1, \ldots, m_{\ell-1}}f(x)\, e(\phi_m(\hat x_\ell)x_\ell + \psi_m(\hat x_\ell))\geq 0
    \end{align}
    for all $m$. Using the positivity property \eqref{E: positivity fin} as well as the pointwise bound $$0\leq \mu_N(m)\leq\frac{1}{(2N+1)^{\ell-1}},$$ we remove the weight $\mu_N(m)$ and extend the range of $m$ to all of $\Z^{\ell-1}$ so that
    \begin{align*}
        \sum_{m\in\Z^{\ell-1}}\sum_{x\in\Z^\ell}\Delta_{m_1, \ldots, m_{\ell-1}}f(x)\, e(\phi_m(\hat x_\ell)x_\ell + \psi_m(\hat x_\ell)) \gg_\ell \delta^{3/2}N^{2\ell-1}.
    \end{align*}
    We then
    shift $m_i\mapsto m_i-x_i$ for every $i\in[\ell-1]$. Under this change of variables, $\Delta_{m_1, \ldots, m_{\ell-1}}f(x)$ becomes a product of $f(x)$ and $\ell-1$ functions, each of which does not depend on one of the coordinates $x_1, \ldots, x_{\ell-1}$. As a result, we obtain 1-bounded weights $b_{1,m}, \ldots, b_{\ell,m}\colon \Z^{\ell-1}\to\C$, which are supported on $[\pm N]^{\ell-1}$ and are identically 0 whenever $\max_i |m_i|>N$, as well as a phase $\tilde \phi_m\colon\Z^\ell\to\C$ such that
    \begin{align*}
        \sum_{m\in\Z^{\ell-1}}\sum_{x\in\Z^\ell}f(x)\, e(\tilde\phi_m(\hat x_\ell)x_\ell)b_{1,m}(\hat{x}_1)\cdots b_{\ell,m}(\hat{x}_\ell) \gg_\ell \delta^{3/2}N^{2\ell-1}.
    \end{align*}
    The result then follows on pigeonholing in the $O_\ell(N^{\ell-1})$ values $m\in\Z^{\ell-1}$ for which the sum over $x$ does not vanish and setting $b_i:= b_{i,m}$ for the choice of $m$ given by the pigeonhole principle.
\end{proof}

\subsection{Decomposition results and anti-uniformity}\label{SS: ARL}

Having established an inverse theorem for the finitary box norm under consideration, we proceed to establish a version of the arithmetic regularity lemma for this norm. In what follows, we will use the following definitions.
\begin{definition}
    Let $M, N\in\N$, and $\veps>0$. We say that $f\colon \Z^\ell\to\C$ is
    \begin{enumerate}
        \item $M$-\textit{structured}
if
\begin{align*}
    f(x) = \sum_{i=1}^M c_i \, e(\tilde\phi_i(\hat x_\ell)x_\ell)\, b_{i,1}(\hat{x}_1)\cdots b_{i,\ell}(\hat{x}_\ell)
\end{align*}
for some weights $c_i\in\C$ satisfying $\max\limits_{i\in[M]}|c_i|\leq M$, phase functions $\phi_i\colon \Z^\ell\to\T$, and 1-bounded functions $b_{i,j}\colon \Z^{\ell-1}\to\C$;
\item \textit{$(M,N)$-anti-uniform} if
\begin{align*}
    \abs{\E_{x\in[\pm N]^\ell} f(x)\, g(x)}\leq \frac{M}{ N^{\ell/2^{\ell+1}}}\cdot \norm{g}_{e_1[\pm N], \ldots, e_\ell[\pm N], e_\ell[\pm N]}
\end{align*}
for any $g\colon \Z^\ell\to\C$ {supported on $[\pm N]^\ell$};
\item \textit{$(\veps, N)$-uniform} if
\begin{align*}
    \norm{f}_{e_1[\pm N], \ldots, e_\ell[\pm N], e_\ell[\pm N]}^{2^{\ell+1}}\leq \veps N^\ell;
\end{align*}
\item \textit{$(\veps, N)$-small} if $\norm{f}_{L^2([\pm N]^\ell)}\leq \veps$.
    \end{enumerate}
\end{definition}

Proposition \ref{P: finitary inverse theorem} can then be upgraded to the following decomposition result.
\begin{proposition}[Arithmetic regularity lemma]\label{P: regularity lemma}
    Let $N\in\N$, $\veps>0$, and $\CF\colon\R_+\to\R_+$ be increasing.
    Then there exists $1\leq M\ll_{\veps, \CF}1$ such that for every 1-bounded $f\colon\Z^\ell\to\C$ supported on $[\pm N]^\ell$,
    we have a decomposition
    \begin{align}\label{E: ARL}
        f = f_{\str}+f_{\sml}+f_{\unif},
    \end{align}
     where
    \begin{enumerate}
        \item $f_{\str}$ is   $M$-structured and {1-bounded};
        \item $f_{\sml}$ is $(\veps, N)$-small;
        \item $f_{\unif}$ is $(1/\CF(M), N)$-uniform and supported on $[\pm N]^\ell$.\footnote{The assumption on the support of $f_{\unif}$ is needed in the proof of Proposition \ref{P: structured & anti-uniform} below to properly use the anti-uniformity assumption. It can be ensured because $f_{\unif}$ is the difference of $f$ and a conditional expectation of $f$ onto some structured factor, both of which are supported on $[\pm N]^\ell$. Note that we cannot in general assume that $f_{\str}$ is supported on $[\pm N]^\ell$: if $\ell=1$, then the structured term would be a bounded-degree trigonometric polynomial, and these are typically not supported on finite intervals.}
    \end{enumerate}
    Furthermore, all terms in \eqref{E: ARL} are 4-bounded.
\end{proposition}
\begin{proof}[Sketch of proof]
	Proposition \ref{P: regularity lemma} can be proved just like a similar decomposition result for Gowers norms \cite[Proposition 2.7]{GT10a} (except that we do not need our structured term to be ``irrational'' in any sense, as given by \cite[Theorem 1.2]{GT10a}). The place of polynomial nilsequences of complexity $M$ is taken by $M$-structured functions. An important property used in the proof is that these functions form a graded conjugation-closed algebra in the sense that if $f,g$ are $M$-structured,  then  $\overline{f}$, $cf$,  $f+g$,  and $f\cdot g$ are $M^2$-structured functions for all $c\in \C$ with $|c|\leq M$.
	\end{proof}

A simple argument based on the Cauchy-Schwarz and van der Corput inequalities shows that every $M$-structured function is $(O_\ell(M^2), N)$-anti-uniform, and so a priori anti-uniformity can be conceived as a weak notion of structure. The next result shows the converse and formalizes the heuristic that the classes of structured and anti-uniform functions are in fact the same.
\begin{proposition}[Structured functions approximate anti-uniform functions]\label{P: structured & anti-uniform}
Let $\veps,A>0$ and $N\in\N$. Then there exists $M=O_{\veps, A,\ell}(1)$ such that if a 1-bounded $f\colon \Z^\ell\to\C$ is $(A,N)$-anti-uniform, then we can find a 1-bounded, $M$-structured $g\colon\Z^\ell\to\C$ satisfying
\begin{align*}
    \norm{f - g}_{L^2([\pm N]^\ell)}\leq \veps.
\end{align*}
\end{proposition}
\begin{proof}
Let $\veps_0>0$ and $\CF\colon\R_+\to\R_+$ be an increasing function, both to be specified later.
By Proposition \ref{P: regularity lemma}, there exists $M =O_{\veps, \CF}(1)$ and
 a decomposition
\begin{align*}
    f  = f_{\str} + f_{\sml} + f_{\unif},
\end{align*}
where
$f_{\str}, f_{\sml}, f_{\unif}\colon \Z^\ell\to \C$ are
such that $f_{\str}$ is 1-bounded and $M$-structured, $f_{\sml}$ is $(\veps_0, N)$-small,
and $f_{\unif}$ is $(1/\CF(M), N)$-uniform.
Hence,
\begin{align*}
    \norm{f-f_{\str}}_{L^2([\pm N]^\ell)}^2\leq \norm{f_{\unif}}_{L^2([\pm N]^\ell)}^2+O(\veps_0).
\end{align*}
To evaluate the term $\norm{f_{\unif}}_{L^2([\pm N]^\ell)}^2$, we use the properties of $f, f_{\str}, f_{\sml}, f_{\unif}$. Expanding and using the triangle inequality, we get
\begin{align*}
    \norm{f_{\unif}}_{L^2([\pm N]^\ell)}^2 &\leq \abs{\angle{f_{\unif}, f}_{L^2([\pm N]^\ell)}}+ \abs{\angle{f_{\unif}, f_{\sml}}_{L^2([\pm N]^\ell)}} + \abs{\angle{f_{\unif},  f_{\str}}_{L^2([\pm N]^\ell)}}.
\end{align*}
The first term is at most $A/\CF(M)$ by the anti-uniformity of $f$ (here, we use the fact that $f_{\unif}$ is supported on $[\pm N]^\ell$).
{The second one can be bounded from above by $4\veps_0$ using the Cauchy-Schwarz inequality and the 4-boundedness of $f_{\unif}$.}
Lastly, the third term is at most $O_\ell(M^2/\CF(M))$ by the previously mentioned anti-uniformity of structured functions.
Hence,
\begin{align*}
    \norm{f_{\unif}}_{L^2([\pm N]^\ell)}^2 &\ll_\ell \veps_0 + \max(A,M^2)/\CF(M).
\end{align*}
By choosing $\CF$ growing sufficiently fast with $\veps_0, A, \ell$, this is at most $O_\ell(\veps_0)$, and the result follows on taking $g:=f_{\str}$ and
$\veps_0 := c_\ell \veps$ for some small $c_\ell>0$.
\end{proof}

\subsection{Topology of structured functions}
Let $M\in\N$, and denote the space of $M$-structured functions on $\Z^\ell$ via $\Str_M$. Embedding $\Str_M$ into $\C^{\Z^\ell}$ in the obvious way, we can endow $\Str_M$ with the product topology.
We also consider the \textit{parameter space}
$$V_{M} := ((M\cdot \D)\times \T^{\Z^{\ell-1}}\times (\D^{\Z^{(\ell-1)}})^{\ell})^M.$$
Recalling that a $M$-structured function $\psi$ takes the form
\begin{align}\label{E: M-structured}
    \psi(x) = \sum_{i=1}^{M}c_{i} \, e(\phi_{i}(\hat x_\ell)x_\ell)\, b_{i,1}(\hat{x}_1)\cdots b_{i,\ell}(\hat{x}_\ell)
\end{align}
for $|c_i|\leq M$, $b_{i,1}, \ldots, b_{i,\ell}\colon \Z^{\ell-1}\to\D$, and $\phi_i\colon \Z^{\ell-1}\to\T$, we can define $\pi\colon V_M \to \Str_M$ via
\begin{align*}
    \pi((c_i, \phi_i, (b_{i,j})_{j\in[\ell]})_{i\in [M]})(x)
    := \sum_{i=1}^{M}c_{i} \, e(\phi_{i}(\hat x_\ell)x_\ell)\, b_{i,1}(\hat{x}_1)\cdots b_{i,\ell}(\hat{x}_\ell).
\end{align*}
This map is surjective by the very definition of $M$-structured functions, but it is not injective; for instance, permuting the indices $i\in[M]$ in the domain does not affect the output.
We endow $V_M$ with the product topology; then it is compact and metrizable, hence sequentially compact.

We will use the following fact.
\begin{lemma}[Sequential compactness of structured functions]\label{L: sequential compactness}
    Let $M\in\N$ and $(\psi_m)_m\subseteq\Str_M$. Then there exists a subsequence $(\psi_{m_k})_k$ and $\psi\in\Str_M$ such that $$\sup_{N\in\N}\lim\limits_{k\to\infty}\norm{\psi_{m_k}-\psi}_{L^2([\pm N]^\ell)} = 0.$$
\end{lemma}
\begin{proof}
    Since $\pi$ is surjective, for any $m\in\N$ there exists $\eta_m\in V_M$ for which $\pi(\eta_m) = \psi_m$. By the sequential compactness of $V_M$, we can find a subsequence $(\eta_{m_k})_k$ that converges pointwise to some $\eta\in V_M$. Let $\psi:=\pi(\eta)$.
    It is easy to verify that for every $N\in\N$ we have $\psi_{m_k}\to \psi$ in $L^\infty([\pm N]^\ell)$, and hence also in $L^2([\pm N]^\ell)$.
\end{proof}

We shall also require the following lemma.

\begin{lemma}[Total boundedness of structured functions]\label{L: total boundedness}
    Let $M,N\in\N$. Then the set of $M$-structured, 1-bounded functions is totally bounded with respect to the $L^2([\pm N]^\ell)$ metric. More specifically, for every $\veps>0$ there exists a finite subset $\CA_{M,N}\subseteq \Str_M$ consisting of 1-bounded functions such that for any 1-bounded $\psi\in\Str_M$, we have
    \begin{align*}
        \min_{\gamma\in \CA_{M,N}}\norm{\psi-\gamma}_{L^2([\pm N]^\ell)}\leq \veps.
    \end{align*}
\end{lemma}
\begin{proof}
Note that $$\Str_{M,N} := \{f|_{[\pm N]^\ell}\colon  f\in\Str_M,\; 1\textrm{-bounded}\}$$ is a subset of the compact set $\D^{[\pm N]^\ell}$. The claim then follows from the fact that every subset of a compact set is totally bounded.
\end{proof}

 \subsection{Ergodicity of structured extensions}\label{SS: ergodicity}
We now move away from the finitary world in order to prove one more lemma. It will be used towards the end of the proof of \cref{P: structured extension} in the next section to ensure that the structured extension can be taken to be ergodic.
\begin{lemma}\label{L:magicergodeco}
 Let $(X, \CZ, \mu, T_1, \ldots, T_\ell)$ be a system
where  $\CZ := \CZ_\mu(T_1, \ldots, T_\ell, T_\ell)$. Suppose that
$$
	\CZ = \CI_\mu(T_1)\vee \cdots \vee \CI_\mu(T_{\ell-1})\vee\CZ_{1,\mu}(T_{\ell}).
$$
If $\mu=\int \mu_x\, d\mu(x)$ is the ergodic disintegration of $\mu$ with respect to the joint action of $T_1,\ldots, T_\ell$, then for $\mu$-a.e. $x\in X$, we have
$$
	\CZ_{\mu_x} = \CI_{\mu_x}(T_1)\vee \cdots \vee \CI_{\mu_x}(T_{\ell-1})\vee\CZ_{1,\mu_x}(T_{\ell}).
$$
	\end{lemma}
\begin{proof}
	By \eqref{E: factor property}, we have
	$$
	f\bot  \CZ_\mu(T_1, \ldots, T_\ell, T_\ell)\quad \Longleftrightarrow \quad
 \nnorm{f}_{T_1,\ldots, T_\ell,T_\ell,\mu}=0.	
	$$
    Hence, our assumption is equivalent to the following implication
	 $$
	 f\, \bot \, \CI_\mu(T_1)\vee \cdots \vee \CI_\mu(T_{\ell-1})\vee\CZ_{1,\mu}(T_{\ell})\quad \implies\quad
	 \nnorm{f}_{T_1,\ldots, T_\ell,T_\ell,\mu}=0
	 $$
	  holding for every $f\in L^\infty(\mu)$. We would like to conclude that
	  for $\mu$-a.e. $x\in X$, the   following implication
	 $$
	 f\, \bot\,  \CI_{\mu_x}(T_1)\vee \cdots \vee \CI_{\mu_x}(T_{\ell-1})\vee\CZ_{1,\mu_x}(T_{\ell})\quad \implies\quad
	 \nnorm{f}_{T_1,\ldots, T_\ell,T_\ell,\mu_x}=0
	 $$
	 holds for every $f\in L^\infty(\mu_x)$.

     This can be proved by arguing as in \cite[Section~3.3]{Chu11} where a similar property was established for the seminorms $\nnorm{f}_{T_1,\ldots, T_\ell}$. We can repeat the same argument verbatim, the only difference is in the proof of  \cite[Lemma~3.14]{Chu11}
	where if $(f_n)_n$ is a countable dense family in $C(X)$, we define the functions
	$g_{1,n},\ldots, g_{\ell-1,n}$ (for $n\in \N$)  exactly as in \cite[Lemma~3.14]{Chu11} by the formula
	$$
	g_{j,n}(x):=\lim_{N\to\infty} \E_{h\in [N]} T_j^hf_n(x) \quad \textrm{for}\quad j=1,\ldots, \ell-1,
	$$
	and  we change the definition of the functions
	 $g_{\ell,n}$ (for $n\in \N$)  to
	$$
	g_{\ell,n}(x):=\lim_{N\to\infty}\E_{h_1,h_2\in[N]} T_\ell^{h_1}\overline{f}_n(x)\cdot T_\ell^{h_2}\overline{f}_n(x)\cdot T_\ell^{h_1+h_2}f_n(x)
	$$
	for those $x\in X$ for which the limit exists and $0$ otherwise. The reason for this change is that we want the collection $(g_{\ell,n})_n$ to be dense in $L^2(X, \CZ_{1,\mu}(T_\ell), \mu)$. The rest of the argument is exactly the same.
\end{proof}

\subsection{Structured extension of the factor}\label{SS: main construction}
Having prepared all the ingredients in the preceding sections, we are in the position
to prove Proposition \ref{P: structured extension}, which is the most laborious step in the proof of Theorem \ref{T: structured extension}.
\begin{proposition}[Structured extension of the factor]\label{P: structured extension}
    Let $(X, \CX, \mu, T_1, \ldots, T_\ell)$ be an ergodic system, and let $\CZ := \CZ(T_1, \ldots, T_\ell, T_\ell)$. Then $(X, \CZ, \mu, T_1, \ldots, T_\ell)$ admits an ergodic extension $(Y, \CY, \nu, S_1, \ldots, S_\ell)$ satisfying 
    \begin{align}\label{E: structural property 2}
    \CY = \CI(S_1)\vee \cdots \vee \CI(S_{\ell-1})\vee\CZ_1(S_{\ell}).
\end{align}
\end{proposition}
\begin{proof}
Our argument follows closely Tao \cite{Tao15} and Leng \cite[Section 9]{Leng25}. Throughout this proof, let $$\CZ := \CZ(T_1, \ldots, T_\ell, T_\ell).$$ Recall also that
${T}^h := T_1^{h_1}\cdots T_\ell^{h_\ell}$ for $h\in\Z^\ell$. For $\f:=(f_\ueps)_{\ueps\in \{0,1\}^{\ell+1}_*}\subseteq L^\infty(\mu)$ and $N\in\N$, we define the finite truncation
\begin{align*}
    \CD_N \f := \E_{h_1, \ldots, h_{\ell+1}\in[N]}\prod_{\ueps\in\{0,1\}^{\ell+1}_*}\CC^{|\ueps|} T_1^{\eps_1 h_1}\cdots T_{\ell-1}^{\eps_{\ell-1} h_{\ell-1}}T_\ell^{\eps_\ell h_\ell+\eps_{\ell+1} h_{\ell+1}}f_\ueps.
\end{align*}
By construction, the space $L^2(X, \CZ,\mu)$ is equal to the closure of the linear span of all   dual functions of the form
\begin{align*}
	\CD \f := \lim_{N\to\infty}\CD_N \f
\end{align*}
for $\f\in (L^\infty(\mu))^{2^{\ell+1}-1}$ (the limit is taken in $L^2(\mu)$).
By separability, we can find a countable
collection $(\f_{n})_{n}\subseteq (L^\infty(\mu))^{2^{\ell+1}-1}$ that contains $1$,  is closed under conjugation,  and such that the linear span of  $(\CD \f_n)_n$ is dense in $L^2(X,\CZ, \mu)$. Our goal is to find a system $(Y, \CY, \nu, S_1, \ldots, S_\ell)$ with the needed structural property
and  {1-bounded} functions $(\widetilde f_{n})_{n}\subseteq L^\infty(\nu)$ such that for any $r\in\N$, $n_1,\ldots, n_r\in \N$,
and shifts $h_1, \ldots, h_r\in\Z^\ell$, we have
\begin{align}\label{E: equal integrals}
    \int T^{h_1}\CD \f_{n_1}\cdots T^{h_r}\CD \f_{n_r}\; d\mu = \int S^{h_1}\widetilde f_{n_1}\cdots  S^{h_r}\widetilde f_{n_r}\; d\nu.
\end{align}
If we do this, then Lemma~\ref{L:factor} guarantees that the system
$(X, \CZ, \mu, T_1, \ldots, T_\ell)$ is a factor of the system $(Y, \CY, \nu, S_1, \ldots, S_\ell)$.

The construction of the system $(Y, \CY, \nu, S_1, \ldots, S_\ell)$ and functions $(\widetilde f_n)_n$ involves several steps:
\begin{enumerate}
    \item We approximate $\CD \f_n$ in $L^2(\mu)$ by finite truncations $\CD_N \f_n$;
specifically, for all $n,m\in\N$ we approximate $\CD \f_n$ by $\CD_{N_{n,m}}\f_n$ with an error $o_{n; m\to\infty}(1)$.
\item We use the maximal ergodic theorem to show that for $\mu$-a.e. $x\in X$, the discrete function $h\mapsto \CD \f_n(T^h x)-\CD_{N_{n,m}}\f_n(T^h x)$ is $o_{n; m\to\infty; x}(1)$ in $L^2([\pm N]^\ell)$ uniformly in $N\in\N$.
\item We apply the finitary inverse theorem to show that for $\mu$-a.e. $x\in X$, the function $h\mapsto \CD_{N_{n,m}}\f_n(T^h x)$ can be approximated in $L^2([\pm N_{n,m}]^\ell)$ by a $O_{n,q}(1)$-structured function $\psi_{n,m,q,x}$ with an error $o_{n;q\to\infty}(1)$. To address measurability issues, we ensure that for every $n,m,q\in\N$, there are only finitely many structured functions to consider.
\item We then use the Hardy-Littlewood maximal inequality to ensure that the approximation in fact holds uniformly in $L^2([\pm N]^\ell)$ for all $N\in\N$. In other words, we pass from a single-scale approximation to an estimate at all scales. Combining all of the above with a compactness argument, we deduce that for $\mu$-a.e. $x\in X$, the discrete function $h\mapsto \CD \f_n(T^h x)$  can be approximated by a $O_{n,q}(1)$-structured function $\psi_{n,q,x}$ with an error $o_{n;q\to\infty}(1)$ in $L^2([\pm N]^\ell)$ uniformly for all $N\in\N$.
\item Fixing a generic point $x_0\in X$ for which all of the above holds\footnote{By the ergodicity of the joint action of $T_1, \ldots, T_\ell$, $\mu$-a.e. point $x\in X$ is generic.}, we use weak-* compactness in the spirit of the Furstenberg correspondence principle to view $(\psi_{n,q,x_0})_{n,q}$ as a point in a (massive) system $(Y, \CY, \nu, S_1, \ldots, S_\ell)$.
\item We show that the new system satisfies the claimed structural property \eqref{E: structural property 2} and also the extension can be chosen to be ergodic. Lastly, we construct the functions $(\widetilde f_{n})_{n}$ that satisfy \eqref{E: equal integrals}.
\end{enumerate}

Throughout, all implicit constants are allowed to depend on $\ell$.

\smallskip
\textbf{Step 1: Approximating by finitary dual functions.}
\smallskip
For any $n,m\in\N$, we may find a scale $N_{n,m}\geq m$ for which
\begin{align*}
    \norm{\CD \f_n - \CD_{N_{n,m}}\f_n}_{L^2(\mu)}\leq 2^{-100(n+m)}.
\end{align*}
Using the maximal ergodic theorem, we will show that for $\mu$-a.e. $x\in X$, the proportion of $h\in\Z^\ell$ for which $\CD \f_n(T^h x)$ is poorly approximable by $\CD_{N_{n,m}}\f_n(T^h x)$ is very small.

Observe first that each set
\begin{align*}
    E_{n,m}:=\{x\in X\colon \abs{\CD \f_n(x) - \CD_{N_{n,m}}\f_n(x)}\geq 2^{-(n+m)}\}
\end{align*}
has measure at most $2^{-99(n+m)}$, so that
\begin{align*}
    F := \sum_{n,m}2^{9(n+m)}\mathbf{1}_{E_{n,m}}
\end{align*}
is in $L^1(\mu)$. By the maximal ergodic theorem, the maximal function
\begin{align*}
    F^*(x) := \sup_{N\in\N}\E_{h\in[\pm N]^\ell}F(T^hx)
\end{align*}
satisfies
\begin{align*}
    t\, \mu(F^*>t) \leq \norm{F}_{L^1(\mu)}<\infty \quad \textrm{for\; every}\quad t>0.
\end{align*}
Hence, for $\mu$-a.e. $x\in X$ we can find $C_x>0$ so that $F^*(x)\leq C_x$. By definition,
\begin{align*}
    F^*(x) \geq 2^{9(n+m)}\sup_{N\in\N}\frac{\abs{\{h\in[\pm N]^\ell\colon  T^h x\in E_{n,m}\}}}{(2N+1)^\ell}
\end{align*}
for every $n,m\in\N$, and so
\begin{align*}
    \sup_{N\in\N}\frac{\abs{\{h\in[\pm N]^\ell\colon T^h x\in E_{n,m}\}}}{(2N+1)^\ell} \leq C_x 2^{-9(n+m)}.
\end{align*}
In other words, for every $N\in\N$ and all but a $O_x(2^{-9(n+m)})$ proportion of $h\in[\pm N]^\ell$, we have $$\abs{\CD \f_n(T^h x)-\CD_{N_{n,m}}\f_n(T^h x)}\leq 2^{-(n+m)}.$$
Consequently, by splitting into $h$'s for which $T^h x\in E_{n,m}$ and the rest, we obtain for $\mu$-a.e. $x\in X$ a multiscale bound
\begin{align}\label{E: second moment bound}
    \sup_{N\in\N}{\E_{h\in[\pm N]^\ell}\abs{\CD \f_n(T^h x)-\CD_{N_{n,m}}\f_n(T^h x)}^2}\leq C_x 2^{-9(n+m)} + 2^{-(n+m)}
\end{align}
that we shall need later.

\smallskip
\textbf{Step 2: Applying the finitary inverse theorem.}
\smallskip
Consider now the function $\CD_{n,m,x}(h):=\CD_{N_{n,m}}\f_n(T^h x)$.
A simple application of the Gowers-Cauchy-Schwarz inequality gives that it is $(O_\ell(1), N_{n,m})$-anti-uniform and so by Proposition \ref{P: structured & anti-uniform} (applied with $\veps := 2^{-200(n+q)}$ for some $q\in\N$), we can find $D_{n,q}>0$ and a 1-bounded, $D_{n,q}$-structured function  $\psi^0_{n,m,q,x}$ for which
\begin{align*}
    \norm{\CD_{n,m,x}-\psi^0_{n,m,q,x}}_{L^2([\pm N_{n,m}]^\ell)}^2\leq 2^{-200(n+q)}.
\end{align*}

A priori, there could be uncountably many choices of $(\psi^0_{n,m,q,x})_{x\in X}$ for each fixed $n,m,q\in\N$. For reasons related to measurability, it is convenient if we only need to deal with finitely many distinct structured functions. Using Lemma \ref{L: total boundedness}, for every $n,m,q\in\N$ we can find a finite subset $\CA^0_{n,m,q}$ of  1-bounded, $D_{n,q}$-structured functions which are $2^{-201(n+q)}$-dense in the set of all 1-bounded, $D_{n,q}$-structured functions with respect to the $L^2([\pm N_{n,m}]^\ell)$ metric. {For a later convenience, we extend each $\CA^0_{n,m,q}$ to the countable, translation-invariant set
\begin{align*}
    \CA_{n,m,q} := \{\psi(g+\cdot)\colon  \psi\in\CA^0_{n,m,q},\; g\in\Z^\ell\}.
\end{align*}
}
By the triangle equality, we therefore have
\begin{align}\label{E: regularity bound 0}
        \inf_{\psi\in\CA_{n,m,q}}\norm{\CD_{n,m,x}-\psi}_{L^2([\pm N_{n,m}]^\ell)}^2< 2^{-100(n+q)}.
\end{align}

For later convenience, we define $\CA_{n,q} := \bigcup_{m\in\N}\CA_{n,m,q}$.

\smallskip
\textbf{Step 3: Approximation by structured functions at all scales.}
\smallskip
Our next step is to ensure that for $\mu$-a.e. $x\in X$, the function $\CD_{n,m,x}$ can be well-approximated by
a structured function not only at scale $N_{n,m}$, but in fact at all scales $N\ll N_{n,m}$.
To this end, we construct ``error sets'' $E'_{n,m,q}$ of points $x\in X$ such that each element of $\CA_{n,q}$ fails to approximate $\CD_{n,m,x}$ at some scale. Specifically, we let
\begin{equation}\label{E:Enmx}
E'_{n,m,q}:=\big\{x\in X\colon     \inf_{\psi\in\CA_{n,q}}\sup_{N\in[N_{n,m}/2]}\norm{\CD_{n,m,x}-\psi}_{L^2([\pm N]^\ell)}^2> 2^{-10(n+q)}\big\}.
\end{equation}
Since for fixed $\psi$ and $N$ the set of points $x$ satisfying $\norm{\CD_{n,m,x}-\psi}_{L^2([\pm N]^\ell)}^2> 2^{-10(n+q)}$ is measurable (due to the measurability of $\CD_{N_{n,m}}\f_n$), the set $E'_{n,m,q}$ is also measurable by virtue of being a countable intersection of countable unions of measurable sets.

We now show that for every $m\in\N$, the union of the error sets $\bigcup_{n,q\in\N} E'_{n,m,q}$ is small, in the sense made precise by \eqref{E: smallness of E'} later on. By \eqref{E: regularity bound 0}, for every $n,m,q\in\N$ and $\mu$-a.e. $x\in X$ we can find $\psi_{n,m,q,x}\in\CA_{n,m,q}$ satisfying
\begin{align}\label{E: regularity bound}
        \norm{\CD_{n,m,x}-\psi_{n,m,q,x}}_{L^2([\pm N_{n,m}]^\ell)}^2 \leq 2^{-100(n+q)}.
\end{align}
Note that we do not claim any measurability of $x\mapsto \psi_{n,m,q,x}$. Since for any $g, h\in\Z^\ell$ we have
\begin{align}\label{E: shifted dual}
 \CD_{n,m,x}(h+g) = \CD_{n,m,T^g x}(h),
\end{align}
we use the translation invariance of $\CA_{n,m,q}$ to ensure that the functions  $\psi$ satisfy
\begin{align}\label{E: shifted structure}
  \psi_{n,m,q,x}(h+g) = \psi_{n,m,q,T^g x}(h).
\end{align}

Now,
define $$M_{n,m,q,x}(h):=\abs{\CD_{n,m,x}(h) - \psi_{n,m,q,x}(h)}^2\cdot \mathbf{1}_{[\pm N_{n,m}]^\ell}(h).$$
{Its discrete Hardy-Littlewood maximal function is defined via
\begin{align*}
    M_{n,m,q,x}^*(g) &:= \sup_{N\in\N_0}\norm{(\CD_{n,m,x}-\psi_{n,m,q,x})\cdot \mathbf{1}_{[\pm N_{n,m}]^\ell}}_{L^2(g+[\pm N]^\ell)}\\
    &= \sup_{N\in\N_0}\norm{(\CD_{n,m,T^g x}-\psi_{n,m,q,T^g x})\cdot \mathbf{1}_{[\pm N_{n,m}]^\ell-g}}_{L^2([\pm N]^\ell)}
\end{align*}
 for any $g\in\Z^\ell$, where the second line follows from \eqref{E: shifted dual} and \eqref{E: shifted structure}.
The discrete Hardy-Littlewood maximal inequality asserts that
\begin{align*}
    \sup_{t>0}\brac{t \cdot |\{g\in\Z^\ell\colon  M_{n,m,q,x}^*(g)>t\}|}\ll \sum_{g\in\Z^\ell}M_{n,m,q,x}(g);
\end{align*}
hence by the definition of $M_{n,m,q,x}$ and \eqref{E: regularity bound}, we have
\begin{align*}
    \sup_{t>0}\brac{t \cdot |\{g\in\Z^\ell\colon M_{n,m,q,x}^*(g)>t\}|} \ll (2N_{n,m}+1)^\ell\cdot 2^{-100(n+q)}.
\end{align*}}
Letting  $t:= 2^{-20(n+q)}$, we get that at most $\ll(2N_{n,m}+1)^\ell\cdot 2^{-80(n+q)}$ values of $g\in\Z^\ell$ satisfy
\begin{align}\label{E: G_x}
    \sup_{N\in\N}\norm{(\CD_{n,m,T^g x}-\psi_{n,m,q,T^g x})\cdot \mathbf{1}_{[\pm N_{n,m}]^\ell-g}}_{L^2([\pm N]^\ell)}
    \gg 2^{-20(n+q)}.
\end{align}

Let $G_{n,m,q,x}$ be the set of all $g\in[\pm N_{n,m}/2]^\ell$ satisfying \eqref{E: G_x}; by the considerations above, we have
\begin{align}\label{E: G_x 2}
    |G_{n,m,q,x}|\ll (2N_{n,m}+1)^\ell\cdot 2^{-80(n+q)}.
\end{align}
If $T^g x\in E'_{n,m,q}$, then
\begin{align}\label{E: E'}
    \sup_{N\in[N_{n,m}/2]}\norm{\CD_{n,m,T^g x}-\psi_{n,m,q,T^g x}}_{L^2([\pm N]^\ell)}^2> 2^{-10(n+q)}
\end{align}
since $\psi_{n,m,q,T^g x}\in\CA_{n,q}$ by assumption. If additionally $g\in [\pm N_{n,m}/2]^\ell$, then for any $N\in[N_{n,m}/2]$ and any $h\in[\pm N]^\ell$ we have $g+h\in [\pm N_{n,m}]^\ell$. Hence, we can modify \eqref{E: E'} by inserting an indicator function as follows:
\begin{align*}
        \sup_{N\in[N_{n,m}/2]}\norm{(\CD_{n,m,T^g x}-\psi_{n,m,q,T^g x})\cdot \mathbf{1}_{[\pm N_{n,m}]^\ell-g}}_{L^2([\pm N]^\ell)}^2> 2^{-10(n+q)}.
\end{align*}
It is immediate from this, \eqref{E: G_x}, and the definition of $G_{n,m,q,x}$, that $g\in G_{n,m,q,x}$. From this and the upper bound \eqref{E: G_x 2} on the size of $G_{n,m,q,x}$, it follows that
\begin{align*}
    \E_{g\in[\pm N_{n,m}/2]^\ell}\mathbf{1}_{E'_{n,m,q}(T^g x)}\leq \E_{g\in[\pm N_{n,m}/2]^\ell} \mathbf{1}_{G_{n,m,q,x}}(g)\ll 2^{-80(n+q)}.
\end{align*}
We can use this to finally upper bound the size of $E'_{n,m,q}$ via
\begin{align*}
    \mu(E'_{n,m,q}) = \E_{g\in[\pm N_{n,m}/2]^\ell}\int T^g\mathbf{1}_{E'_{n,m,q}}\; d\mu
    \ll 2^{-80(n+q)}.
\end{align*}

The bound on the size of $E'_{n,m,q}$ thus obtained implies that
\begin{align}\label{E: smallness of E'}
    \sup_{m\in \N} \Bignorm{\sum_{q\in\N}\sum_{n\in\N}2^{5(n+q)}\mathbf{1}_{E'_{n,m,q}}}_{L^1(\mu)}<\infty.
\end{align}
By Fatou's lemma,
\begin{align*}
        \liminf_{m\to\infty}\sum_{q\in\N}\sum_{n\in\N}2^{5(n+q)}\mathbf{1}_{E'_{n,m,q}}
\end{align*}
is in $L^1(\mu)$. Hence, for $\mu$-a.e. $x\in X$ there exists $C'_x>0$ such that
\begin{align*}
        \liminf_{m\to\infty}\sum_{q\in\N}\sum_{n\in\N}2^{5(n+q)}\mathbf{1}_{E'_{n,m,q}}(x)\leq C'_x.
\end{align*}
In particular, it follows that for $\mu$-a.e. $x\in X$ there exists an increasing sequence $(m_{x,k})_k$ of integers such that
\begin{align*}
    \sum_{q\in\N}\sum_{n\in\N}2^{5(n+q)}\mathbf{1}_{E'_{n,m_{x,k},q}}(x)\leq C'_x\quad \textrm{for\; all}\quad k\in\N.
\end{align*}
If $x\in E'_{n,m_{x,k},q}$ for some $n,k\in\N$, then $\sum_{n\in\N}2^{5(n+q)}\mathbf{1}_{E'_{n,m_{x,k},q}}(x)\geq 2^{5q}$ (simply because one of the indicators is nonzero), which is greater than $C_x'$ if $q$ is sufficiently large. Hence, there must exist a threshold $L_x\in\N$ such that for all $n,k\in\N$ and all $q\geq N_x$, we have $x\notin E'_{n,m_{x,k},q}$. {It follows from the definition of the set $E'_{n,m_{x,k},q}$  in \eqref{E:Enmx} that there exists $\psi'_{n,m,q,x}\in\CA_{n,q}$ such that}
\begin{align*}
    \sup_{N\in[N_{n,m}/2]}\norm{\CD_{n,m,x}-\psi'_{n,m,q,x}}_{L^2([\pm N]^\ell)}^2\leq 2^{-10(n+q)}
\end{align*}
for $\mu$-a.e. $x\in X$ and every $q\geq L_x$, $n\in\N$, and $m\in(m_{x,k})_k$. In combination with \eqref{E: second moment bound}, this gives
\begin{align*}
        \sup_{N\in[N_{n,m}/2]}{\E_{h\in[\pm N]^\ell}\abs{\CD \f_n(T^h x)-\psi'_{n,m,q,x}(h)}^2}\leq 2^{-10(n+q)} + C_x 2^{-9(n+m)} + 2^{-(n+m)}
\end{align*}
for $\mu$-a.e. $x\in X$ and every $q\geq L_x$, $n\in\N$, and $m\in(m_{x,k})_k$. Taking $m\to\infty$ along $m\in(m_{x,k})_k$, we get
\begin{align*}
    \sup_{N\in\N}\limsup_{\substack{m\to\infty,\\ m\in (m_{x,k})_k}}{\E_{h\in[\pm N]^\ell}\abs{\CD \f_n(T^h x)-\psi'_{n,m,q,x}(h)}^2}\leq 2^{-10(n+q)}
\end{align*}
for $\mu$-a.e. $x\in X$ and every $q\geq L_x$, $n\in\N$.
 By Lemma~\ref{L: sequential compactness}, for $\mu$-a.e. $x\in X$ and every $q\geq L_x$, $n\in\N$, we can find a  {1-bounded} $D_{n,q}$-structured sequence $\psi_{n,q,x}$ such that
\begin{align}\label{E: all scale approximation}
    \sup_{N\in\N}{\E_{h\in[\pm N]^\ell}\abs{\CD \f_n(T^h x)-\psi_{n,q,x}(h)}^2}\leq 2^{-10(n+q)}.
\end{align}
Thus, we showed that $\CD \f_n(T^h x)$ can be approximated by structured functions $\psi_{n,q,x}$ at all scales with arbitrary accuracy.

\smallskip
\textbf{Step 4: Constructing the structured extension.}
\smallskip
{\em From now on, we fix a generic point $x_0\in X$ that satisfies all of the above (such points form a full measure set). }

We will construct the system $Y$ as follows. First, for $M\in\N$, define
\begin{align*}
    Y_{M,j} := \begin{cases}
    (\D^{\Z^{\ell-1}})^M,\; &j\in[\ell]\\
    ((\S^1)^{\Z^\ell})^M,\; &j=\ell+1,\\
    \end{cases}
\end{align*}
together with the $\Z^\ell$-actions $\sigma_{M,j}$ on $Y_{M,j}$ given by
\begin{align}
    \label{E: sigma}
    \sigma_{M,j}^h(y_{i,k})_{i\in[M],\; k\in\Z^{\ell-1}} &:= (y_{i,k+\hat{h}_j})_{i\in[M],\; k\in\Z^{\ell-1}}\quad \textrm{for}\quad j\in[\ell]\\
    \nonumber \sigma_{M,\ell+1}^h(y_{i,k})_{i\in[M],\; k\in\Z^\ell} &:= (y_{i,k+h})_{i\in[M],\; k\in\Z^\ell}
\end{align}
for $h\in\Z^\ell$.
We then let
\begin{align}\label{E: Y_M}
    Y_M := Y_{M,1}\times \cdots \times Y_{M,\ell+1}
\end{align}
and define a $\Z^\ell$-action $R_M$ on $Y_M$ via
\begin{align*}
    R_M^h(y_1, \ldots, y_{\ell+1}) := (\sigma_{M,1}^h y_1, \ldots, \sigma_{M,\ell+1}^h y_{\ell+1}).
\end{align*}
We can then lift an $M$-structured function
\begin{align*}
        \psi(h) := \sum_{i=1}^{M}c_{i}\,  e(\phi_{i}(\hat h_\ell)h_\ell)\, b_{i,1}(\hat{h}_1)\cdots b_{i,\ell}(\hat{h}_\ell)
\end{align*}
(for $|c_i|\leq M$, $\phi_i\colon \Z^\ell\to\D$, $b_{i,j}\colon \Z^{\ell-1}\to\D$) to an element
$$
y_\psi:=(y_1,\ldots, y_{\ell+1})
$$
given by
\begin{gather*}
    y_j=(y_{j,i,\hat{h}_j})_{i\in[M],\; \hat{h}_j\in\Z^{\ell-1}} \quad \textrm{with}\quad y_{j,i,\hat{h}_j} :=b_{i,j}(\hat{h}_j)\quad \textrm{for}\quad j\in[\ell],\\
    y_{\ell+1} = (y_{\ell+1, i, h})_{i\in[M],\; h\in\Z^\ell}\quad \textrm{with}\quad y_{\ell+1,i,h} :=
    e(\phi_i(\hat{h}_\ell) h_\ell).
\end{gather*}
Then
\begin{align*}
    \psi(h) =
    \sum_{i\in[M]}c_i \, y_{1,i,\hat{h}_1}\cdots y_{\ell,i,\hat{h}_\ell}\, y_{\ell+1,i,h} = F_{\psi} (R_M^h y_\psi),
\end{align*}
where
\begin{align*}
    F_{\psi} (y) := \sum_{i\in[M]}c_i \, y_{1,i,0}\cdots y_{\ell+1,i,0}.
\end{align*}

If we only dealt with one $M$-structured function, we could define the structured extension directly on $Y_M$. However, we need to deal with the functions $(\psi_{n,q,x_0})_{n,q}$, each of which is $D_{n,q}$-structured. Therefore, we set
\begin{align*}
    Y := \prod_{n,q\in\N}Y_{n,q},
\end{align*}
where each $Y_{n,q}$ is an isomorphic copy of $Y_{D_{n,q}}$, endow $Y$ with the Borel $\sigma$-algebra $\CY$, and define a $\Z^\ell$-action $S$ on $Y$ by $S|_{Y_{n,q}} = R_{D_{n,q}}$. We then consider the point $y_0 := (y_{\psi_{n,q,x_0}})_{n,q}.$ By letting $F_{n,q}:Y\to\C$ equal $F_{\psi_{n,q,x_0}}$ on $Y_{n,q}$ and 0 elsewhere, we get that $F_{n,q}(S^h y_0) = \psi_{n,q,x_0}(h)$. Lastly, we take the measure $\nu$ on $Y$ to be a weak-* limit of the sequence of measures
\begin{align*}
    \nu_N:=\E_{h\in[\pm N]^\ell}\delta_{S^h y_0}
\end{align*}
(such a measure exists by the Banach-Alaoglu theorem).
{Note that since  $\psi_{n,q,x_0}$ is 1-bounded we get that $F_{n,q}$ is 1-bounded $\nu$-a.e.}
If $(N_l)_l$ is an increasing sequence of integers along which $\nu$ is realized, then we get that for any $r\in\N$, $n_1, q_1, \ldots, n_r, q_r\in\N$,
and $k_1, \ldots, k_r\in\Z^\ell$, we have
\begin{multline*}
    \lim_{l\to\infty}\E_{h\in[\pm N_l]^\ell} \psi_{n_1,q_1,x_0}(h+k_1) \cdots \psi_{n_r,q_r,x_0}(h+k_r)\\
    =\lim_{l\to\infty}\E_{h\in[\pm N_l]^\ell}F_{n_1,q_1}(S^{h+k_1}y_0) \cdots F_{n_r,q_r}(S^{h+k_r}y_0)\\
    =\int S^{k_1}F_{n_1,q_1} \cdots S^{k_r}F_{n_r,q_r}\; d\nu.
\end{multline*}

\smallskip
\textbf{Step 5: Properties of the structured extension.}
Recall \eqref{E: all scale approximation}. It follows that for every $q,n\in\N$ and $q'\geq q$, we have
\begin{align*}
    \sup_{N\in\N}\norm{\psi_{n,q',x_0}-\psi_{n,q,x_0}}_{L^2([\pm N]^\ell)}^2\ll 2^{-10(n+q)}.
\end{align*}
Taking the limit along $(N_l)_l$ and using the identification from the previous step, we conclude that
\begin{align*}
    \norm{F_{n,q}-F_{n,q'}}_{L^2(\nu)}^2\ll 2^{-10(n+q)}.
\end{align*}
Hence, the sequence $(F_{n,q})_q$ is Cauchy in $L^2(\nu)$. Let $\widetilde f_n$ be its limit.
{Note first that since  $F_{n,q}$ is 1-bounded  $\nu$-a.e. we get that $\widetilde f_n$ is 1-bounded $\nu$-a.e.} Moreover,
\begin{align*}
    \abs{\int_X \CD \f_n\; d\mu -\int_Y \widetilde f_n\; d\nu}&=\lim_{q\to\infty}\abs{\int_X \CD \f_n\; d\mu -\int_Y F_{n,q}\; d\nu}\\
    &= \lim_{q\to\infty}\lim_{l\to\infty}\abs{\E_{h\in[\pm N_l]^\ell}(\CD \f_n(T^h x_0) - F_{n,q}(S^hy_0))}\\
    &\leq \lim_{q\to\infty}\limsup_{l\to\infty}\norm{\CD \f_n(T^h x_0) - \psi_{n,q,x_0}(h)}_{L^2([\pm N_l]^\ell)}=0,
\end{align*}
where the ultimate vanishing of the expression follows from \eqref{E: all scale approximation}.
Repeating everything with $T^{h_1}\CD \f_{n_1}\cdots T^{h_r}\CD \f_{n_r}$ in place of $\CD \f_n$, we deduce \eqref{E: equal integrals}.

It remains to prove the structural property \eqref{E: structural property 2}.
For each $n,q\in\N$, we can write $Y_{n,q} = Y_{n,q,1}\times\cdots\times Y_{n,q,\ell+1}$ just like in \eqref{E: Y_M}.  We note that because of the way the action $S$ and the measure $\nu$ have been defined, all functions $F: Y\to\C$ that are independent of the coordinates indexed by $(n,q,\ell+1)_{n,q}$ are $\CI(S_1)\vee \cdots \vee \CI(S_{\ell})$-measurable (where we set $S_j:=S^{e_j}$ for all $j\in[\ell]$). {Indeed, each such function can be approximated by a linear combination of products of functions, each of which depends only on the coordinates indexed by $(n,q,j)_{n,q}$ for some fixed $j\in[\ell]$. Such a function then is $S_j$-invariant because for every $n,q\in\N$, the transformation $S_j$ reduces on $Y_{n,q,j}$ to $\sigma_{D_{n,q}, j}^{e_j}$ (the $j$-th coordinate of the $\Z^\ell$ action $\sigma_{D_{n,q},j}$ defined in \eqref{E: sigma}), which is the identity by definition.}

What remains to be shown is that the functions independent of the coordinates indexed by $(n,q,j)_{n,q\in\N,\; j\in[\ell]}$ are all $\CZ_1(S_\ell)$-measurable. Given $y\in Y$, let $$\tau:Y\to Y|_{\ell+1}:=\prod_{n,q\in\N}Y_{n,q,\ell+1}$$ be the natural projection. 
If $\tilde S:=\tau\circ S$ is the induced $\Z^\ell$-action on $Y|_{\ell+1}$, then for all $h\in \Z^{\ell+1}$ we have
\begin{align}\label{E: tau}
    \tilde S^{h} y_0 =
    \big( e(\phi_{i,n,q}(\hat{k}_\ell+\hat{h}_\ell) \, (k_\ell+h_\ell))\big)_{i\in[D_{n,q}],\; n,q\in\N,\; k\in\Z^\ell}.
\end{align}

For each fixed $n',q'\in\N$ and $i'\in D_{n',q'}$, consider the functions $G_{i',n',q'}: Y|_{\ell+1}\to\S^1$ given by
$$
G_{i',n',q'}((y_{i, n,q,k})_{i\in[D_{n,q}],\; n,q\in\N,\; k\in\Z^\ell}) := y_{i', n',q',0};
$$
these are variants of the usual ``projection onto the 0-th coordinate map''. Note that the algebra generated by the family $(\tilde S^h G_{i',n',q'})_{i', n',q',h}$ is dense in  $C(Y|_{\ell+1})$.

We will shortly show that the functions $\tilde S_\ell G_{i',n',q'}\cdot \overline G_{i',n',q'}$ are $\tilde S_\ell$-invariant.
 Then Lemma~\ref{L:NonErgodicEigen} and $|G_{i',n',q'}|=1$ imply that each $G_{i',n',q'}$ is $\CZ_1(\tilde S_\ell)$-measurable, hence, $Y|_{\ell+1}$ is a $\CZ_1(\tilde S_\ell)$-system, and  as a consequence  all functions on $Y$ independent of the coordinates indexed by $(n,q,j)_{n,q\in\N,\; j\in[\ell]}$ are $\CZ_1(S_\ell)$-measurable.

To prove the above claim, it suffices to show that
\begin{align*}
    \int_{Y|_{\ell+1}}\abs{\tilde S_\ell(\tilde S_\ell G_{i',n',q'}\cdot \overline G_{i',n',q'}) - \tilde S_\ell G_{i',n',q'}\cdot \overline G_{i',n',q'}}^2\; d\tilde\nu = 0,
\end{align*}
where $\tilde\nu := \tau_*\nu = \lim\limits_{l\to\infty}\E_{h\in[\pm N_l]^\ell}\delta_{\tilde S^h y_0}$.
Expanding the left-hand side and using that $|G_{i',n',q'}|~=~1$, we get that it is equal to
$$
     \int_{Y|_{\ell+1}} 2 - 2\Re\brac{\tilde S_\ell^2 G_{i',n',q'}\cdot \tilde S_\ell \overline G_{i',n',q'}\cdot \tilde S_\ell \overline G_{i',n',q'}\cdot G_{i',n',q'}}
    \; d\tilde\nu.
$$
Since the measure $\tilde\nu$ is a weak-* limit, the last integral is equal to
$$
  2-2  \lim\limits_{l\to\infty}\E_{h\in[\pm N_l]^\ell}
    \Re\brac{G_{i',n',q'}(\tilde S_\ell^2 S^h y_0)\cdot \overline G_{i',n',q'}(\tilde S_\ell S^h y_0)\cdot  \overline G_{i',n',q'}(\tilde S_\ell S^h y_0)\cdot G_{i',n',q'}(S^h y_0)}.
$$
By the definition of the function $G_{i',n',q'}$ and \eqref{E: tau}, this equals
\begin{align*}
  2-2\lim\limits_{l\to\infty}\E_{h_\ell\in[\pm N_l]}    \brac{  \Re\brac{e(\phi_{i',n',q'}(\hat{h}_\ell) ((h_\ell+2) - (h_\ell+1) - (h_\ell + 1) + h_\ell))}} = 0,
\end{align*}
and we are done.

\smallskip
\textbf{Step 6: Ergodicity.}
\smallskip
Lastly, we establish that the extension can be chosen to be ergodic.
Let $\pi\colon Y\mapsto X$ be the factor map and $\nu=\int \nu_y \, d\nu(y)$ be the ergodic decomposition of $\nu$ with respect to  the joint action $S_1,\ldots, S_\ell$. Then
$\mu=\int \pi(\nu_y)\, d\nu(y)$ and for $\nu$-a.e. $y\in Y$ the measures $\pi(\nu_y)$ are invariant under $T_1,\ldots, T_\ell$. Since the system $(X, \CZ, \mu,T_1,\ldots, T_\ell)$ is ergodic, we have $\mu=\pi(\nu_y)$ for $\nu$-a.e. $y\in Y$. Hence, $\pi\colon (Y,\CY, \nu_y,S_1,\ldots, S_\ell)\mapsto (X,\CZ, \mu,T_1,\ldots, T_\ell)$ is also a factor map   for $\nu$-a.e. $y\in Y$.
Moreover, by Lemma~\ref{L:magicergodeco},  for $\nu$-a.e. $y\in Y$ the extension
$(Y,\CY, \nu_y,S_1,\ldots, S_\ell)$ also has the asserted structural properties. This completes the proof.
\end{proof}
\subsection{Structured extension of the original system}\label{SS: extension}
In the remaining section, we upgrade Proposition \ref{P: structured extension} to \cref{T: structured extension}, i.e., we construct a structured extension of the original system $(X, \CZ, \mu, T_1, \ldots, T_\ell)$ rather than the factor
\begin{align*}
    \CZ := \CZ(T_1, \ldots, T_\ell, T_\ell).
\end{align*}
The general idea underlying this construction comes from Leng, who implemented it in a slightly different setting in \cite[Appendix C]{Leng25}.

 Before we give the full proof of \cref{T: structured extension}, we present its short outline that aims to shed light on the main issues arising in the construction. In Proposition \ref{P: structured extension}, we showed that the factor $(X, \CZ, \mu, T_1, \ldots, T_\ell)$ can be extended to $(\widetilde X, \widetilde\CZ, \widetilde\mu, \widetilde T_1, \ldots, \widetilde T_\ell)$ with 
\begin{align*}
    \widetilde\CZ := \CI(\widetilde T_1)\vee \cdots \vee \CI(\widetilde T_{\ell-1})\vee\CZ_1(\widetilde T_{\ell}).
\end{align*}
Unfortunately, this new system does not extend the original system $(X, \CX, \mu, T_1, \ldots, T_\ell)$, meaning that a non-$\CZ$-measurable function on $X$ cannot in general be lifted to a function on $\widetilde X$. To get the system that both extends the original system and on which elements of $L^\infty(X, \CZ, \mu)$ lift to ``structured'' functions, we pass to a relatively independent joining of $(X, \CX, \mu, T_1, \ldots, T_\ell)$ and $(\widetilde X, \widetilde\CZ, \widetilde\mu, \widetilde T_1, \ldots, \widetilde T_\ell)$ over $\CZ$, which we denote by $(X_1, \CX_1, \mu_1, T_{1,1}, \ldots, T_{\ell,1})$. Letting $\pi\colon  (\widetilde X, \widetilde\CZ, \widetilde\mu)\to (X, \CZ, \mu)$ be the factor map, any $f\in L^\infty(X, \CZ,\mu)$ lifts to an element  $f \otimes 1 = 1 \otimes (f\circ \pi)$ in $L^\infty(X_1, \CX_1, \mu_1)$, which is then measurable with respect to 
\begin{align}\label{E: 1,1}
    \CI(T_{1,1})\vee \cdots \vee \CI(T_{\ell-1, 1})\vee \CZ_1(T_{\ell,1}).
\end{align}
The problem, however, is that $\CZ_1:=\CZ(T_{1,1}, \ldots, T_{\ell, 1}, T_{\ell, 1})$ may contain more elements than just the lifts of elements of $\CZ$ and $\widetilde\CZ$ (the latter are already ``structured'', i.e., measurable with respect to \eqref{E: 1,1}). To amend this, we use \cref{P: structured extension} again to get a structured extension $(\widetilde X_1, \widetilde\CZ_1, \widetilde\mu_1, \widetilde T_{1,1}, \ldots, T_{\ell,1})$ of $(X_1, \CZ_1, \mu_1, T_{1,1}, \ldots, T_{\ell,1})$, and we pass to the bigger system $(X_2, \CX_2, \mu_2, T_{1,2}, \ldots, T_{\ell,2})$, a relatively independent joining of  $(X_1, \CX_1, \mu_1, T_{1,1}, \ldots, T_{\ell,1})$ and $(\widetilde X_1, \widetilde\CZ_1, \widetilde\mu_1, \widetilde T_{1,1}, \ldots, T_{\ell,1})$ over $\CZ_1$. Iterating this construction, we get an increasing sequence of extensions $(X_m, \CX_m, \mu_m, T_{1,m}, \ldots, T_{\ell,m})$ of $(X, \CZ, \mu, T_1, \ldots, T_\ell)$, and for each $m$, the $\CZ_m$-measurable functions become ``structured'' at the level $m+1$ of the construction. The system $(Y, \CY, \nu, S_1, \ldots, S_\ell)$ is then the inverse limit of this tower of extensions. 

\begin{proof}[Proof of \cref{T: structured extension}]
    We will take $(Y, \CY, \nu, S_1, \ldots, S_\ell)$ to be the inverse limit of ergodic systems $(X_m, \CX_m, \mu_m, T_{1,m}, \ldots, T_{\ell,m})$ constructed inductively as follows. For $m=0$, set 
    \begin{align*}
        (X_m, \CX_m, \mu_m, T_{1,m}, \ldots, T_{\ell,m}) := (X, \CX, \mu, T_1, \ldots, T_\ell).
    \end{align*}
    Given $m\in\N_0$, let
    \begin{align*}
        \CZ_m:=\CZ(T_{1,m}, \ldots, T_{\ell, m}, T_{\ell, m}),
    \end{align*}
    and let $(\widetilde X_m, \widetilde\CZ_m, \widetilde\mu_m, \widetilde T_{1,m}, \ldots, T_{\ell,m})$ be the ergodic structured extension of $\CZ_m$ constructed in \cref{P: structured extension}. Then define $(X_{m+1}, \CX_{m+1}, \mu^\dagger_{m+1}, T_{1,m+1}, \ldots, T_{\ell,m+1})$ to be the relatively independent joining of
    \begin{align*}
     (X_m, \CX_m, \mu_m, T_{1,m}, \ldots, T_{\ell,m})\quad \textrm{and}\quad    (\widetilde X_m, \widetilde\CZ_m, \widetilde\mu_m, \widetilde T_{1,m}, \ldots, T_{\ell,m})
    \end{align*}
    over $\CZ_m$,
    i.e.,  $X_{m+1}:=X_m\times \widetilde X_m$, $\CX_{m+1} = \CX_m\otimes\widetilde\CZ_m$, $T_{j,m+1}:=T_{j,m}\times \widetilde{T}_{j,m}$ for $j\in[\ell]$, and the measure $\mu_{m+1}^\dagger$ satisfies
    	\begin{equation}\label{E:rip}
    	\int f\otimes \widetilde{f}\, d\mu^\dagger_{m+1}:=	\int \E(f|\CZ_m)\cdot \E(\widetilde{f}|\CZ_m)\, d\mu_m
    	\end{equation}
        for all  $f\in L^\infty(X_m, \CX_m, \mu_m)$ and $\widetilde{f}\in L^\infty(\widetilde X_m, \widetilde\CX_m, \widetilde{\mu}_m)$, where by abuse of notation we let $\E(\widetilde{f}|\CZ_m)$ be the orthogonal projection on $L^2(X_m, \CZ_m, \mu_m)$ under the factor map  $\pi_m\colon (\widetilde X_m, \widetilde \CZ_m, \widetilde \mu_m)\to (X_m, \CZ_m, \mu_m)$.

    This is almost what we want except that to iterate the construction using \cref{P: structured extension}, we need the measure on $X_{m+1}$ to be ergodic. To this end, we consider the ergodic decomposition $\mu^\dagger_{m+1} = \int \mu^\dagger_{m+1, y}\; d\mu^\dagger_{m+1}(y)$. Arguing as in the last step of the proof of \cref{P: structured extension}, we deduce that for $\mu^\dagger_{m+1}$-a.e. $y\in X_{m+1}$, the measure $\mu^\dagger_{m+1, y}$ is ergodic, and its marginals on $X_m$ and $\widetilde X_m$ are $\mu_m, \widetilde \mu_m$ respectively. The measure $\mu_{m+1}$ will then equal $\mu^\dagger_{m+1, y}$ for some $y\in X_{m+1}$ satisfying additional properties specified at a later stage.

    For now, suppose that all the constructed systems $(X_m, \CX_m, \mu_m, T_{1,m}, \ldots, T_{\ell,m})$ are well-defined. Our goal then is to show that 
    \begin{align*}
        \CZ:=\CZ(S_1, \ldots, S_\ell, S_\ell)
    \end{align*}
    satisfies \eqref{E: structural property}. 

Let $\tau_m\colon (Y, \CY, \nu)\to (X_m, \CX_m, \mu_m)$ be the factor map; we first claim that 
    \begin{align}\label{E: inverse limit}
        \CZ = \bigvee_{m\in\N_0}\tau_m\inv(\CZ_m).
    \end{align}
    By the martingale convergence theorem, for every $f\in L^\infty(Y, \CZ, \nu)$ and $\veps>0$ we can find $m\in\N$ and $g\in L^\infty(X_m, \CX_m, \mu_m)$ such that $\norm{f-g\circ \tau_m}_{L^2(\nu)}\leq \veps$. If $\E(g|\CZ_m) = 0$. Then we have $\langle f, g\circ\tau_m\rangle = 0$ by approximating $f$ with linear combinations of dual functions $\CD_{S_1, \ldots, S_\ell, S_\ell}((f_\ueps)_\ueps)$ for some 1-bounded $(f_\ueps)_\ueps\subseteq L^\infty(Y, \CY,\nu)$ {(we can do this by the defining property of $\CZ$)}  and noting that
    \begin{align*}
        \abs{\langle \CD_{S_1, \ldots, S_\ell, S_\ell}((f_\ueps)_\ueps), g\circ \tau_m\rangle}\leq \nnorm{g\circ \tau_m}_{S_1, \ldots, S_\ell, S_\ell} = \nnorm{g}_{T_{1,m}, \ldots, T_{\ell,m}, T_{\ell,m}} = 0
    \end{align*}
    by applying the Gowers-Cauchy-Schwarz inequality and the factor property.
    Hence,  {by the Pythagorean theorem}, we can assume that the approximant $g$ is in fact $\CZ_m$-measurable, and the identity \eqref{E: inverse limit} follows on taking $\veps\to 0$.

    Because of \eqref{E: inverse limit},  property \eqref{E: structural property} will follow if we can show that 
    \begin{align}\label{E: dealing with level m}
        \tau_m\inv(\CZ_m) \subseteq \CI(S_1)\vee \cdots \vee \CI(S_{\ell-1})\vee\CZ_1(S_{\ell})
    \end{align}
    holds for every $m\in\N_0$. In fact, we will show something stronger: the $\CZ_m$-measurable functions attain the desired form at level $m+1$, in the sense that 
    \begin{align}\label{E: passing to level m+1}
        \tau_m\inv(\CZ_m) &\subseteq \tau_{m+1}\inv\brac{\CI(T_{1,m+1})\vee \cdots \vee \CI(T_{\ell-1, m+1})\vee \CZ_1(T_{\ell,m+1})}.
    \end{align}
{Since $\tau_{m+1}\colon (Y,\CY,\nu)\to (X_{m+1},\CX_{m+1},\mu_{m+1})$ is a factor map, the right-hand side of \eqref{E: passing to level m+1} is a subset of the right-hand side of \eqref{E: dealing with level m}. It follows that to prove \eqref{E: dealing with level m}, it suffices to establish \eqref{E: passing to level m+1}.}

    Fix $m\in\N_0$. To establish \eqref{E: passing to level m+1}, we need to describe further properties of the point $y\in X_{m+1}$ that defines $\mu_{m+1} := \mu^\dagger_{m+1, y}$.
    Consider a countable collection of functions $\CF\subseteq L^\infty(X_m,\CZ_m,\mu_m)$ that is $L^2(\mu_m)$-dense in $L^2(X_m, \CZ_m, \mu_m)$ {(such a set exists since we deal with standard  probability spaces).} 
    Then every $f\in \CF$ can be lifted to $f \otimes 1\in L^\infty(X_{m+1}, \CZ_{m+1}, \mu^\dagger_{m+1})$ in the sense that $f\otimes 1$ projects down to $f$ under the factor map. The defining property \eqref{E:rip} of $\mu^\dagger_{m+1}$ as a relatively independent joining over $\CZ_m$ gives\footnote{{Indeed,   one can immediately verify using \eqref{E:rip} that 
        	for all $g\in L^\infty(\mu_m)$, $\widetilde{g}\in L^\infty(\tilde{\mu}_m)$, both integrals 
          $\int (f\otimes 1) \cdot (g\otimes \widetilde{g})\, d\mu^\dagger_m$ and 
                 	$\int (1\otimes (f\circ\pi_m)) \cdot (g\otimes \widetilde{g})\, d\mu^\dagger_m$ equal  
                 	$\int f\cdot \E(g|\CZ_m)\cdot \E(\tilde{g}_m|\CZ_m)\, d\mu_m$, hence
                 	$\int (f\otimes 1 - 1\otimes (f\circ\pi_m)) \cdot g\otimes \widetilde{g}\, d\mu^\dagger_m=0$.}}
    \begin{equation}\label{E: lift 0}
        f \otimes 1 = 1 \otimes (f\circ \pi_m) \quad\textrm{in}\quad L^2(\mu^\dagger_{m+1}).
    \end{equation}
    Since $\CF$ is countable, we can find a  $\mu^\dagger_{m+1}$-full measure set $E\in \CX_{m+1}$ such that for any $f\in\CF$ and $y\in E$, we have
    \begin{align}\label{E: lift}
        f \otimes 1 = 1 \otimes (f\circ \pi_m) \quad\textrm{in}\quad L^2(\mu^\dagger_{m+1, y}),
    \end{align}
    the measure $\mu^\dagger_{m+1, y}$ is ergodic, and it projects down to $\mu_m$ on $X_m$ and $ \widetilde \mu_m$ on $\widetilde X_m$. Since for any $f\in L^\infty(X_m,\CZ_m,\mu_m)$ and $y\in E$, we have
    \begin{align*}
        \inf_{g\in\CF}\norm{f\otimes 1 - g\otimes 1}_{L^2\brac{\mu^\dagger_{m+1,y}}} &= \inf_{g\in\CF}\norm{1 \otimes (f\circ \pi_m) - 1 \otimes (g\circ \pi_m)}_{L^2\brac{\mu^\dagger_{m+1,y}}}\\
        &= \inf_{g\in\CF}\norm{f-g}_{L^2(\mu_m)} = 0,
    \end{align*}
where we use the facts that $\CF$ is dense in $L^2(X_m, \CZ_m, \mu_m)$ and 
     the measure $\mu^\dagger_{m+1, y}$  projects down to $\mu_m$ on $X_m$.
    We deduce that \eqref{E: lift} holds for all $f\in L^\infty(X_m,\CZ_m,\mu_m)$ and $y\in E$. We then set $\mu_{m+1}:=\mu^\dagger_{m+1, y}$ for any choice of $y\in E$ (different choices of $y$ may give a different measure $\mu_{m+1}$).

    The identity \eqref{E: lift} with $\mu_{m+1}=\mu^\dagger_{m+1, y}$ plays a crucial part in establishing \eqref{E: passing to level m+1}, as we are about to show. Recall first that the factor $\widetilde\CZ_m$ has the form
    \begin{align}\label{E: structure factor}
        \widetilde \CZ_m = \CI(\widetilde T_{1,m})\vee \cdots \vee \CI(\widetilde T_{\ell-1, m})\vee \CZ_1(\widetilde T_{\ell,m})
    \end{align}
    by \cref{P: structured extension}.
    Since $\pi_m\inv(\CZ_m)\subseteq \widetilde\CZ_m$, for any $f\in L^\infty(X_m,\CZ_m,\mu_m)$, the function $f\circ\pi_m$ is measurable with respect to \eqref{E: structure factor}, and so $1\otimes (f\circ\pi_m)$ is measurable with respect to the right-hand side of \eqref{E: passing to level m+1}. By \eqref{E: lift}, the same holds for $f\otimes 1$. This implies that \eqref{E: passing to level m+1} holds and  completes the proof. 
    \end{proof}

\section{Limiting formulas}\label{S: limiting formulas}
In this section, we establish Theorems \ref{T: limiting formula} and \ref{T: limiting formula 2} via a variant of the by now standard degree lowering argument. That is, assuming that the averages in \eqref{E:vanish} and \eqref{E:vanish'} are controlled by the seminorm $\nnorm{f_2}_{(T_2T_1\inv)^{\times s_1}, T_2^{\times s_2}}$ for some $s_1, s_2\in\N$ (which follows from our assumptions), we will iteratively lower $s_1, s_2$ one by one until we reach $s_1 = s_2 = 1$. We will then pass to an appropriate magic extension of the system and show that the claimed limiting formulas hold in the magic extension.

The proofs of Theorems \ref{T: limiting formula} and \ref{T: limiting formula 2} do not differ much; therefore we only give full details of the proof of the former and provide brief explanations of modifications needed to obtain the latter. The only meaningful difference in the proof of \cref{T: limiting formula 2} is the use of \cref{P: linear 2} below. Given this, the reader should have no problem in adapting our argument to recover the full proof of Theorem \ref{T: limiting formula 2}.

Moreover, we work with Ces\`aro averages only,
the arguments   adapt straightforwardly  to handle averages
   along  arbitrary F\o lner  sequences in $\N$.

This section is organized as follows. First, we prove Theorems \ref{T: limiting formula} and \ref{T: limiting formula 2} under the additional assumption that $f_2$ is a nonergodic eigenfunction of $T_2$ or $T_2T_1\inv$, in which case the proofs become significantly simpler. These results, labeled as Propositions \ref{P: base case} and \ref{P: base case 2}, will serve as the base cases in the proofs of Theorems \ref{T: limiting formula} and \ref{T: limiting formula 2} respectively. Next, we show that various linear averages along commuting transformations admit optimal box seminorm control. Since the inductive step of the degree lowering proceeds by comparing the average along $a(n)$ with an average along $n$, being able to control the linear average by a box seminorm of optimal degree plays a key part in reducing the degree of the box seminorm controlling the average along $a(n)$. Subsequently, we state versions of two standard technical lemmas, Proposition \ref{P: stashing} (stashing) and Proposition \ref{P: dual-difference} (dual-difference interchange), required for the forthcoming degree lowering argument. Section \ref{SS: degree lowering} is then dedicated to the inductive step of degree lowering, by far the main and most technically challenging argument in Section \ref{S: limiting formulas}. We conclude with deriving Theorems \ref{T: limiting formula} and \ref{T: limiting formula 2}.

\subsection{Base case of degree lowering: $f_2$ is a nonergodic eigenfunction}\label{SS:base}
We start with establishing the base cases of Theorems \ref{T: limiting formula} and \ref{T: limiting formula 2}, i.e., the variants of both results in which $f_2$ is a nonergodic eigenfunction of $T_2$ or $T_2T_1\inv$.

\begin{proposition}\label{P: base case}
	Let $a\colon \N\to\Z$ be a sequence that admits box seminorm control and equidistributes on nilsystems (see \cref{D: good properties}). Then the identity \eqref{E: equality of limits} holds for every system $(X, \CX, \mu, T_1, T_2)$ and all functions $f_0,f_1, f_2\in L^\infty(\mu)$ with $f_2\in\CE(T_2)\cup\CE(T_2T_1\inv)$.
\end{proposition}
\begin{proof}
	Let $(X, \CX, \mu, T_1, T_2)$ be a system and $f_0, f_1, f_2\in L^\infty(\mu)$. Suppose first that $f_2\in\CE(T_2)$.
	Since $a$ admits box seminorm control, there exists $s\in\N$ such that
	\begin{align*}
		\lim_{N\to\infty}\E_{n\in[N]}\int f_0 \cdot T_1^{a(n)}f_1\cdot T_2^{a(n)}f_2\; d\mu = 0
	\end{align*}
	whenever $\nnorm{f_1}_{s, T_1} = 0$. By \cite[Proposition 3.1]{CFH11} (a nonergodic variant of the Host-Kra decomposition~\cite{HK05a}), for every $\veps>0$ we can decompose $f_1 = g_1 + g_2 + g_3$ into functions $g_1, g_2, g_3\in L^\infty(\mu)$ for which:
	\begin{enumerate}
		\item $g_1(T_1^n x) = \psi_x(n)$ is an $(s-1)$-step nilsequence for $\mu$-a.e. $x\in X$;
		\item $\norm{g_2}_{L^2(\mu)}\leq \veps$;
		\item $\nnorm{g_3}_{s, T_1} = 0$.
	\end{enumerate}
	Hence, it suffices to prove the claim with $g_1$ in place of $f_1$.
	
	Since  $f_2$ is a nonergodic eigenfunction of $T_2$ we get by \eqref{E:nonergodiceigen}  that  there exists
	a  $T$-invariant set $E\in \CX$ and a measurable $T$-invariant function $\phi \colon X\to \T$ such that
	for every $n\in\Z$ and $\mu$-a.e. $x\in X$, we have
	 $$
	 f_2(T_2^n x) = \mathbf{1}_{E}(x)\, e(\phi(x)n)\, f_2(x).
	 $$
Hence, for every $N\in\N$ and $\mu$-a.e. $x\in E$, we have
	\begin{align*}
		\E_{n\in[N]}g_1(T_1^{a(n)}x)\cdot f_2(T_2^{a(n)}x) = (f_2\cdot \mathbf{1}_E)(x)\cdot \E_{n\in[N]}\psi_x({a(n)}) \cdot e(\phi(x) a(n)).
	\end{align*}
	On taking $N\to\infty$, the assumption that $a$ equidistributes on nilsystems gives
	\begin{align*}
		\lim_{N\to\infty}\E_{n\in[N]}g_1(T_1^{a(n)}x)\cdot f_2(T_2^{a(n)}x)
            &=(f_2\cdot \mathbf{1}_E)(x)\cdot\lim_{N\to\infty}\E_{n\in[N]}\psi_x(a(n))\cdot e(\phi(x) a(n))\\
		&=(f_2\cdot \mathbf{1}_E)(x)\cdot\lim_{N\to\infty}\E_{n\in[N]}\psi_x(n)\cdot e(\phi(x) n)\\
		&= \lim_{N\to\infty}\E_{n\in[N]}g_1(T_1^n x)\cdot f_2(T_2^n x).
	\end{align*}
	The claimed identity \eqref{E: equality of limits} in the case $f_2\in\CE(T_2)$ follows, with $g_1$ in place of $f_1$,  on taking the inner product of the average above with $f_0$.

    If $f_2\in\CE(T_2T_1\inv)$ instead, then we compose the integral with $T_1^{-a(n)}$ to get
    \begin{align}\label{E: change of variables trick}
       \int f_0 \cdot T_1^{a(n)}f_1\cdot T_2^{a(n)}f_2\; d\mu = \int f_1 \cdot T_1^{-a(n)}f_0\cdot (T_2T_1\inv)^{a(n)}f_2\; d\mu
    \end{align}
for every $n\in\N$.
    The conclusion follows by applying the previous case to the system $(X, \CX, \mu, T_1\inv, T_2T_1\inv)$.
\end{proof}

 The trick \eqref{E: change of variables trick} via which we prove the second part of Proposition \ref{P: base case} gives one reason why we want our assumptions on box seminorm control and equidistribution on nilsystems to hold for all systems, not just for the original system and its nilfactors. Later on, we will also pass to structured extensions and ergodic components of the system under consideration, making even more substantial use of the global nature of our assumptions.

Just like Proposition \ref{P: base case} serves as the ``base case'' in the proof of Theorem \ref{T: limiting formula}, the following result plays a similar role in deriving Theorem \ref{T: limiting formula 2}.
\begin{proposition}\label{P: base case 2}
    Let $a\colon \N\to \Z$ be a sequence and $n_k\in \N_0$ for $k\in \N$ so that the sequence  $(a(k!n+n_k))_{n,k}$ admits box seminorm control and equidistributes on nilsystems. Then the identity \eqref{E: equality of limits'} holds for every system $(X, \CX, \mu, T_1, T_2)$ and all functions $f_0,f_1, f_2\in L^\infty(\mu)$ with $f_2\in\CE(T_2)\cup\CE(T_2T_1\inv)$.
\end{proposition}
Proposition \ref{P: base case 2} follows from a straightforward adaptation of the proof of Proposition \ref{P: base case}. After invoking the seminorm control assumption and using the decomposition result from \cite{CFH11}, we obtain the claimed identity by directly applying the equidistribution on nilsystems assumption. We spare the reader the repetitive details of the argument.

\subsection{Optimal seminorm control for linear averages}\label{SS:lineark!}
Much like Proposition \ref{P: base case}, Theorem \ref{T: limiting formula 2} is proved by comparing averages along $a(n)$ with averages along $n$. A key fact about the former, proved by Host~\cite{H09}, is that they admit optimal box seminorm control. The existence of the limit has been proved by Tao~\cite{Ta08} beforehand, and it is not really needed for the inequality below if one is willing to replace limit by limsup.
\begin{proposition}[Optimal seminorm control for linear averages {\cite[Proposition 1]{H09}}]\label{P: linear}
	Let $(X, \CX, \mu, T_1, \ldots, T_\ell)$ be a system. For all 1-bounded $f_1, \ldots, f_\ell\in L^\infty(\mu)$, we have
	\begin{align*}
		\lim_{N\to\infty}\norm{\E_{n\in[N]}T_1^n f_1\cdots T_\ell^n f_\ell}_{L^2(\mu)} \leq\nnorm{f_\ell}_{T_\ell T_1\inv, \ldots, T_\ell T_{\ell-1}\inv,T_\ell}.
	\end{align*}
\end{proposition}

For the proof of Theorem \ref{T: limiting formula 2}, we will also need a variant of the previous result that covers the expressions
	$$
	\lim_{k\to \infty}\lim_{N\to\infty}\abs{\E_{n\in[N]}\int f_0\cdot T_1^{k!n} f_1\cdot T_2^{k!n} f_2\; d\mu}.
	$$
	Unfortunately, the seminorm bounds we get for the inside average by combining Proposition \ref{P: linear} with the scaling property of box seminorms introduce a multiplicative factor on the right-hand side that blows up when we take $k\to \infty$. However, we can get the following soft quantitative result that suffices for our purposes (we will only use the case $\ell=2$).

	\begin{proposition}[Optimal seminorm control for linear averages, II]\label{P: linear 2}
		Let $\ell\geq 2$ be an integer and $(X, \CX, \mu, T_1, \ldots, T_\ell)$ be a system.  Then for every $\varepsilon>0$ there
		exists $\delta>0$ (depending on $\veps$ and the system) such that for all 1-bounded $f_1, \ldots, f_\ell\in L^\infty(\mu)$, we have
		\begin{align*}
			\nnorm{f_\ell}_{T_\ell T_1\inv, \ldots, T_\ell T_{\ell-1}\inv, T_\ell}\leq \delta\quad \implies 		\quad \lim_{k\to \infty}\lim_{N\to\infty}\norm{\E_{n\in[N]}T_1^{k!n} f_1\cdots T_\ell^{k!n} f_\ell}_{L^2(\mu)} \leq \varepsilon.
		\end{align*}
	\end{proposition}
	\begin{remark}
		The important thing for us is that $\delta$ does not depend on the functions $f_1,\ldots, f_\ell$ as long as they are $1$-bounded.
	\end{remark}
	We will prove Proposition \ref{P: linear 2} by first establishing a qualitative version and then using the usual technique  of transferring qualitative results to soft quantitative ones. To perform the second of these steps, we first need to know the norm convergence of the relevant averages. This is provided by the result below.
	\begin{lemma}\label{L:convergence}
		Let $(X, \CX, \mu, T_1, \ldots, T_\ell)$ be a system and $f_1, \ldots, f_\ell\in L^\infty(\mu)$. Then the following iterated  limit exists in $L^2(\mu)$:
		$$
		\lim_{k\to \infty}\lim_{N\to\infty}\E_{n\in[N]}T_1^{k!n} f_1\cdots T_\ell^{k!n} f_\ell.
		$$
	\end{lemma}
	\begin{proof} Suppose first that $\ell\geq 2$; we will reduce by induction to the case $\ell=1$, and then resolve the base case via the spectral theorem.
		
        By Proposition~\ref{P: linear} and the scaling property of box seminorms (property~\eqref{I:scaling} in \cref{SS:boxseminorms}), we have
        \begin{align*}
            \lim_{k\to \infty}\lim_{N\to\infty}\norm{\E_{n\in[N]}T_1^{k!n} f_1\cdots T_\ell^{k!n} f_\ell}_{L^2(\mu)} = 0
        \end{align*}
        whenever $\nnorm{f_\ell}_{T_\ell T_1\inv, \ldots, T_\ell T_{\ell-1}\inv,T_\ell} = 0$ (the scaling step fails for $\ell=1$). Passing to an appropriate magic extension (which exists by Proposition \ref{P: magic extensions exist}), we can assume that our system is magic with respect to $T_\ell T_1\inv, \ldots, T_\ell T_{\ell-1}\inv,T_\ell$. Hence, it suffices to assume that $f_\ell$ is measurable with respect to
        \begin{align*}
             \CZ(T_\ell T_1\inv, \ldots, T_\ell T_{\ell-1}\inv,T_\ell) = \CI(T_\ell T_1\inv)\vee\cdots\vee \CI(T_\ell T_{\ell-1}\inv)\vee\CI(T_\ell).
        \end{align*}
        By an $L^2(\mu)$ approximation argument, we can take $f_\ell = g_1\cdots g_\ell$, where $g_j\in I(T_\ell T_j\inv)$ for $j\in[\ell-1]$ and $g_\ell \in I(T_\ell)$. Rearranging terms, it thus suffices to show the (double) convergence of
        \begin{align*}
            \E_{n\in[N]}T_1^{k!n} (f_1 g_1)\cdots T_{\ell-1}^{k!n} (f_{\ell-1}g_{\ell-1}).
        \end{align*}
        We have thus reduced the length of the average from $\ell$ to $\ell-1$; an iteration of this procedure brings us to $\ell=1$. In this case, $L^2(\mu)$ convergence is easy to check using the spectral theorem. Indeed, it suffices to show that for every $\alpha\in[0,1)$, the limit
		\begin{equation}\label{E:k!alpha}
			\lim_{k\to\infty} \lim_{N\to\infty} \E_{n\in[N]} e(k! n\alpha)=
			\lim_{k\to\infty} {\bf 1}_{\Z}(k!\alpha)
		\end{equation}
        exists. For $\alpha\notin\Q$, the right-hand side is $0$ while for $\alpha\in\Q$ it is $1$.\footnote{The place of $k!$ can be taken any sequence $c_k$ such that
			$\lim\limits_{k\to\infty} {\bf 1}_{r\Z}(c_k)$ exists for every $r\in \N$, i.e., for every $r\in \N$ either $r\mid c_k$ for all large enough $k$ or $r\nmid c_k$ for all large enough $k$.}
	\end{proof}

We now combine all the ingredients to prove Proposition \ref{P: linear 2}.
\begin{proof}[Proof of Proposition \ref{P: linear 2}]
    Just like in the proof of Lemma \ref{L:convergence}, we infer from Proposition~\ref{P: linear} and the scaling property of box seminorms (property~\eqref{I:scaling} in  \cref{SS:boxseminorms}) that
        \begin{align*}
            \lim_{k\to \infty}\lim_{N\to\infty}\norm{\E_{n\in[N]}T_1^{k!n} f_1\cdots T_\ell^{k!n} f_\ell}_{L^2(\mu)} = 0
        \end{align*}
        whenever $\nnorm{f_\ell}_{T_\ell T_1\inv, \ldots, T_\ell T_{\ell-1}\inv,T_\ell} = 0$. By Lemma \ref{L:convergence}, the multilinear  functional
        \begin{align*}
            A(f_1, \ldots, f_\ell) := \lim_{k\to \infty}\lim_{N\to\infty}\E_{n\in[N]}T_1^{k!n} f_1\cdots T_\ell^{k!n} f_\ell
        \end{align*}
        is well-defined. The result then follows from \cite[Proposition A.2]{FrKu22a} applied with $X_j:=L^{\ell+1}(\mu)$, $X_j':=L^\infty(\mu)$, $\norm{\cdot}_{X_j'}:=\norm{\cdot}_{L^\infty(\mu)}$, $\norm{\cdot}_{X_j}:=\norm{\cdot}_{L^{\ell+1}(\mu)},$ for $j\in[\ell]$, and the seminorm $\nnorm{\cdot}:= \nnorm{\cdot}_{T_\ell T_1\inv, \ldots, T_\ell T_{\ell-1}\inv,T_\ell}$.
\end{proof}

\subsection{Preliminary lemmas}
Before performing the inductive step of degree lowering, we present two tricks that we shall use abundantly in the proof of Proposition \ref{P: degree lowering} below. While both of them are by now a standard part of the degree lowering toolbox, they have not appeared in the ergodic literature in this particular form before, therefore we opted to state the precise statements used in this work. Their proofs are completely straightforward modifications of existing arguments.

\begin{proposition}[Stashing]\label{P: stashing}
    Let $(X, \CX, \mu, T_1, \ldots, T_\ell)$ be a system, $a_k\colon \N\to\Z$ be a sequence for every $k\in\N$, and $f_0, \ldots, f_\ell\in L^\infty(\mu)$ be 1-bounded. If
    \begin{align*}
        \limsup_{k\to\infty}\limsup_{N\to\infty}\abs{\E_{n\in[N]}\int f_0\cdot T_1^{a_k(n)}f_1\cdots T_\ell^{a_k(n)}f_\ell\; d\mu}\geq\delta
    \end{align*}
    for some $\delta>0$, then there exist sequences of integers $(N_{k,l})_{k,l}$ (increasing as $l\to\infty$ for every $k\in\N$) and $(k_i)_i$ (increasing as $i\to\infty$) as well as a function
    \begin{align*}
        F_\ell := \lim_{i\to\infty}\lim_{l\to\infty}\E_{n\in[N_{k_i,l}]}T_\ell^{-a_{k_i}(n)}\overline{f}_0\cdot (T_1T_\ell\inv)^{a_{k_i}(n)}\overline{f}_1\cdots (T_{\ell-1}T_\ell\inv)^{a_{k_i}(n)}\overline{f}_{\ell-1}
    \end{align*}
    (where the limit is a weak limit) such that
    \begin{align*}
        \limsup_{k\to\infty}\limsup_{N\to\infty}\abs{\E_{n\in[N]}\int f_0\cdot T_1^{a_k(n)}f_1\cdots T_{\ell-1}^{a_k(n)}f_{\ell-1}\cdot T_\ell^{a_k(n)}F_\ell\; d\mu}\geq\delta^2.
    \end{align*}
    Moreover, if $a_k = a$ for every $k\in\N$, then $(N_{k,l})$ is chosen independently of $k$.
\end{proposition}
The proof of Proposition \ref{P: stashing} is completely analogous to the one of \cite[Proposition 4.2]{Fr21} except we need to use weak compactness twice, first to obtain the sequence $(N_{k,l})_{k,l}$ for which the weak limit
\begin{align*}
    A_k:=\lim_{l\to\infty}\E_{n\in[N_{k,l}]}T_\ell^{-a_{k}(n)}\overline{f}_0\cdot (T_1T_\ell\inv)^{a_{k}(n)}\overline{f}_1\cdots (T_{\ell-1}T_\ell\inv)^{a_{k}(n)}\overline{f}_{\ell-1}
\end{align*}
exists, and then to find the sequence $(k_i)_i$ such that $A_{k_i}\to F_\ell$ weakly as $i\to\infty$.

\begin{proposition}[Dual-difference interchange]\label{P: dual-difference}
    Let $(X, \CX, \mu, T_1, \ldots, T_\ell)$ be a system and $\CG\subseteq L^\infty(\mu)$ be a $T_1, \ldots, T_\ell$-invariant collection of 1-bounded functions closed under multiplication and complex conjugation. For a collection of 1-bounded functions $(f_{n,k})_{n,k}\subseteq L^\infty(\mu)$ and a sequence of integers $(N_{k,l})_{k,l}$ increasing as $l\to\infty$ for every $k\in\N$, let
    \begin{align*}
        f := \lim_{k\to\infty}\lim_{l\to\infty}\E_{n\in[N_{k,l}]}f_{n,k},
    \end{align*}
    where the limit is a weak limit. Suppose that the integrals below are real and
    \begin{align*}
        \liminf_{H\to\infty}\E_{h\in [H]^\ell}\int \Delta_{T_1, \ldots, T_\ell; h} f\cdot g_h\; d\mu \geq \delta
    \end{align*}
    for some $g_h\in\CG$ and $\delta>0$. Then the integrals below are real and
    \begin{align*}
        \liminf_{H\to\infty}\E_{h,h'\in [H]^\ell}\limsup_{k\to\infty}\limsup_{l\to\infty}\E_{n\in[N_{k,l}]}\int \Delta_{T_1, \ldots, T_\ell; h-h'} f_{n,k}\cdot g_{h,h'}\; d\mu \geq \delta^{2^\ell}
    \end{align*}
    for some $g_{h,h'}\in \CG$.
\end{proposition}
The proof of Proposition \ref{P: dual-difference} follows very closely the proof of \cite[Proposition 4.3]{Fr21}. If needed, one can obtain a precise description of the term $g_{h,h'}$ (see e.g. \cite[Proposition 4.3]{Fr21} and \cite[Proposition 5.6]{FrKu22a} for the formula); they were required for past applications of dual-difference interchange in \cite{Fr21,  FrKu22c, FrKu22a} but are not needed in this paper.

\subsection{Inductive step of degree lowering}\label{SS: degree lowering}

With all the preliminaries out of the way, we are ready to carry out the inductive step in the proof of Theorems \ref{T: limiting formula}, captured in the result below.

\begin{proposition}[Degree lowering]\label{P: degree lowering}
    Let $a\colon \N\to\Z$ be a sequence that admits box seminorm control and equidistributes on nilsystems. Let $s_1, s_2\in\N$.
    Suppose that for every system $(X, \CX, \mu, T_1, T_2)$ and all functions $f_0, f_1, f_2\in L^\infty(\mu)$,   we have \eqref{E:vanish} whenever $\nnorm{f_2}_{(T_2T_1\inv)^{\times s_1}, T_2^{\times s_2}} = 0$.
        Then for every system $(X, \CX, \mu, T_1, T_2)$ and all functions $f_0, f_1, f_2\in L^\infty(\mu)$, we have \eqref{E:vanish} whenever $\nnorm{f_2}_{T_2T_1\inv, T_2} = 0$.
\end{proposition}

\begin{proof}
If $s_1 = s_2 = 1$, then there is nothing to show, so we assume without loss of generality that $s_1+s_2 \geq 3$. The claimed seminorm control then follows on iterating the following claim as long as $s_1 + s_2\geq 3$.

\begin{claim}\label{C: degree lowering}
Suppose that for every system $(X, \CX, \mu, T_1, T_2)$ and functions $f_0, f_1, f_2\in L^\infty(\mu)$,
\begin{align}\label{E: positivity}
        \limsup_{N\to\infty}\abs{\E_{n\in[N]}\int f_0\cdot T_1^{a(n)}f_1\cdot T_2^{a(n)}f_2\; d\mu} > 0
\end{align}
implies
$$
\nnorm{f_2}_{(T_2T_1\inv)^{\times s_1}, T_2^{\times s_2}}>0.
$$
Then also \eqref{E: positivity} implies
\begin{align*}
    &\nnorm{f_2}_{(T_2T_1\inv)^{\times s_1}, T_2^{\times (s_2-1)}}>0\quad \textrm{whenever}\quad s_2\geq 2\\
    \textrm{and}\quad &\nnorm{f_2}_{(T_2T_1\inv)^{\times (s_1-1)}, T_2^{\times s_2}}>0\quad \textrm{whenever}\quad s_1\geq 2.
\end{align*}
\end{claim}

It suffices to prove the second part of the claim. Indeed, on composing the integral in \eqref{E: positivity} with $T_1^{-a(n)}$, we have
\begin{align}\label{E: reparametrization}
            \limsup_{N\to\infty}\abs{\E_{n\in[N]}\int f_1\cdot T_1^{-a(n)}f_0\cdot (T_2T_1\inv)^{a(n)}f_2\; d\mu} > 0.
\end{align}
Hence, the second part of the claim follows from the first part on replacing the system $(X, \CX, \mu, T_1, T_2)$ with $(X, \CX, \mu, T_1\inv, T_2T_1\inv)$ and noticing that in the reparametrization \eqref{E: reparametrization}, $f_2$ is a nonergodic eigenfunction of the transformation $T_2T_1\inv$ acting on it.

We thus fix a system $(X, \CX, \mu, T_1, T_2)$ and 1-bounded functions $f_0, f_1, f_2\in L^\infty(\mu)$ satisfying \eqref{E: positivity}. We also assume Claim \ref{C: degree lowering} holds for some fixed $s_1, s_2\in\N$ with $s_2\geq 2$.

In several places, we will first run the argument in the simple case $s_1 = 1$, $s_2 = 2$ before delving into the general case. This will allow us to convey the main ideas of the proof without getting bogged down in distracting technicalities.

\smallskip
\textbf{Step 1: Reducing to the ergodic case.}
At the next stage of our argument, we will pass from the original system to its structured extension constructed in Theorem \ref{T: structured extension}. Since the extension is constructed for ergodic systems, it is desirable to work inside ergodic systems.
Let  $\mu = \int \mu_x\; d\mu(x)$ be the ergodic decomposition of $\mu$ with respect to the joint action of $T_1, T_2$.
Then our assumption and Fatou's lemma give
\begin{align*}
    \limsup_{N\to\infty}\abs{\E_{n\in[N]}\int f_0\cdot T_1^{a(n)}f_1\cdot T_2^{a(n)}f_2\; d\mu_x} > 0
\end{align*}
for all $x$ in a positive-measure set $E\in \CX$. Thus,  if we establish the desired conclusion $\nnorm{f_2}_{(T_2T_1\inv)^{\times s_1}, T_2^{\times (s_2-1)}; \mu_x} > 0$ for all $x\in E$, then it follows from \eqref{E: ergodic decomposition} that $$\nnorm{f_2}_{(T_2T_1\inv)^{\times s_1}, T_2^{\times (s_2-1)}; \mu}^{2^{s_1+s_2-1}} = \int\nnorm{f_2}_{(T_2T_1\inv)^{\times s_1}, T_2^{\times (s_2-1)}; \mu_x}^{2^{s_1+s_2-1}}\; d\mu(x) > 0.$$ Hence, we can assume without loss of generality that $\mu$ is ergodic. 

\smallskip
\textbf{Step 2: Passing to a structured extension.}
At this point, we pass to a structured extension on which a certain seminorm admits a nice inverse theorem. For the remainder of the proof, set $$\CZ:=\CZ(T_2 T_1\inv, T_2, T_2).$$
By Theorem \ref{T: structured extension}, the system $(X, \CX, \mu, T_1, T_2)$ admits an extension $(Y, \CY, \nu, S_1, S_2)$ for which
    \begin{align}\label{E: factor formula baby}
        \hat\CZ:=\CZ(S_2 S_1\inv, S_2, S_2)= \CI(S_2 S_1\inv)\vee \CZ_1(S_2).
    \end{align}
    For $j=0,1,2$, we lift $f_j$ to $\hat{f}_j\in L^\infty(\nu)$. Then
\begin{align*}
    \limsup_{N\to\infty}\abs{\E_{n\in[N]}\int \hat{f}_0\cdot S_1^{a(n)}\hat{f}_1\cdot S_2^{a(n)}\hat{f}_2\; d\nu} > 0.
\end{align*}

\smallskip
\textbf{Step 3: Replacing $\hat{f}_2$ by a structured function.}
Our next step is to replace $\hat{f}_2$ by a certain structured function that keeps track of the shape of the average under consideration. By Proposition \ref{P: stashing}, we can find an increasing integer sequence $(N_l)_l$ such that on setting
\begin{align}\label{E: F_2 baby}
    \hat{F}_2 := \lim_{l\to\infty}\E_{n\in[N_l]}S_2^{-a(n)}\overline{\hat f_0}\cdot (S_1S_2\inv)^{a(n)}\overline{\hat f_1}
\end{align}
(where the limit is a weak limit), we obtain the lower bound
\begin{align}\label{E: positivity on ergodic component baby}
    \limsup_{N\to\infty}\abs{\E_{n\in[N]} \int \hat f_0\cdot S_1^{a(n)}\hat{f}_1 \cdot S_2^{a(n)}\hat F_2\; d\nu} > 0.
\end{align}
Our assumption in \cref{C: degree lowering} then gives
\begin{align}\label{E: auxiliary with F_2}
    \nnorm{\hat{F}_2}_{(S_2 S_1\inv)^{\times s_1}, S_2^{\times s_2}}>0.
\end{align}
Hence, we have transferred the original control from an arbitrary function $f_2$ on $X$ to the more structured function $\hat{F}_2$ on $Y$.

\smallskip
\textbf{Interlude: Auxiliary seminorm control in the baby case $s_1 = 1$, $s_2 = 2$.}
Our next objective is to use the estimate \eqref{E: auxiliary with F_2} together with the inductive formula for box seminorms to reduce to the case where we can apply Proposition \ref{P: inverse theorem} for the seminorm $\nnorm{\cdot}_{S_1S_2\inv, S_2, S_2}$. After invoking the inverse theorem, we aim to reduce our average to one that can be handled using Proposition \ref{P: base case}. An application of Propositions \ref{P: base case} and \ref{P: linear} will then give us an auxiliary control over our original average by $\nnorm{f_1}_{T_1, T_1T_2\inv}$.

The maneuvers outlined above can be greatly simplified in the baby case  $s_1 = 1$, $s_2 = 2,$ as contrasted with the case $s_1 + s_2 \geq 4$. We therefore take a pause to explain how the aforementioned steps can be carried out in this baby case. If $s_1 = 1$, $s_2 = 2$, then \eqref{E: auxiliary with F_2} and Proposition \ref{P: inverse theorem} give $\hat g\in I(S_2 S_1\inv)$ and $\hat \chi\in\CE(S_2)$ such that
\begin{align*}
    \int \hat F_2 \cdot \overline{\hat g}\cdot\overline{\hat{\chi}}\; d\nu>0.
\end{align*}
Expanding the definition of $\hat F_2$, we get
\begin{align*}
	\limsup_{N\to\infty}\abs{\E_{n\in[N]} \int {\hat{f}_0}\cdot S_1^{a(n)}{\hat{f}_1}\cdot S_2^{a(n)}\bigbrac{\hat{g}\cdot\hat{\chi}}\; d\nu}>0.
\end{align*}
The $S_2 S_1\inv$-invariance of $\hat{g}$ allows us to rearrange terms so that
\begin{align*}
	\limsup_{N\to\infty}\abs{\E_{n\in[N]} \int {\hat{f}_0}\cdot S_1^{a(n)}\bigbrac{{\hat{f}_1}\cdot \hat{g}}\cdot S_2^{a(n)}\hat{\chi}\; d\nu}>0.
\end{align*}
The key observation is that since $\hat{\chi}$ is a nonergodic eigenfunction of $S_2$, this average is amenable to analysis via Proposition \ref{P: base case}. Indeed, this result allows us to swap $a(n)$ for $n$, yielding (the limit below exists by \cite{CL84})
\begin{align*}
    \lim_{N\to\infty}\abs{\E_{n\in[N]} \int {\hat{f}_0}\cdot S_1^{n}\bigbrac{{\hat{f}_1}\cdot \hat{g}}\cdot S_2^{n}\hat{\chi}\; d\nu}>0.
\end{align*}
Proposition \ref{P: linear} then implies that $\nnorm{\hat{f}_1\cdot \hat{g}}_{S_1, S_1S_2\inv}>0$. Using the $S_1S_2\inv$-invariance of $\hat{g}$, one  can upgrade this estimate to $\nnorm{\hat{f}_1}_{S_1, S_1S_2\inv}>0$ via a trick elucidated in Step 7 below (the point is that applying a multiplicative derivative along $S_1S_2\inv$ in the definition of the seminorm kills $\hat{g}$). On projecting down to $X$, we get an auxiliary seminorm control $\nnorm{{f}_1}_{T_1, T_1T_2\inv}>0$.

\smallskip
\textbf{Step 4: Applying the inverse theorem.}
These maneuvers become considerably more involved when $s_1 +s_2 \geq 4$, in which case we do not have any immediate inverse theorem for the seminorm \eqref{E: auxiliary with F_2}. To circumvent this difficulty, we use the inductive formula for box seminorms in order to recast \eqref{E: auxiliary with F_2} as an average of degree-3 box seminorms of multiplicative derivatives of $\hat{F}_2$. In preparation for that, set
\begin{align*}
    \hat{F}_{2,h} &:=\Delta_{(S_2 S_1\inv)^{\times (s_1-1)}, S_2^{\times (s_2-2)};h}\hat{F}_2.
\end{align*}
By the inductive formula for box seminorms, we get
\begin{align*}
    \nnorm{\hat{F}_2}_{(S_2 S_1\inv)^{\times s_1}, S_2^{\times s_2}}^{2^{s_1+s_2}} &=\lim_{H\to\infty}\E_{h\in[H]^{s_1+s_2-3}}\nnorm{\hat{F}_{2,h}}_{S_2 S_1\inv, S_2, S_2}^8,
\end{align*}
which is positive by \eqref{E: auxiliary with F_2}.
By Proposition \ref{P: inverse theorem} (we note that a purely qualitative inverse theorem does not suffice here), we can then find 1-bounded functions $\hat g_{h}\in \CI(S_2S_1\inv)$ and $\hat\chi_h\in\CE(S_2)$, such that
\begin{align*}
	 \liminf_{H\to\infty}\E_{h\in[H]^{s_1+s_2-3}}{\int \hat{F}_{2,h}\cdot \overline{\hat g_h}\cdot \overline{\hat{\chi}_h}\; d\nu}>0.
\end{align*}

\smallskip
\textbf{Step 5: Dual-difference interchange.}
By Proposition \ref{P: dual-difference} (dual-difference interchange) applied to the function $f:=\hat F_2$ and the collection
$$\CG := \{\hat{g}\cdot \hat\chi\colon \hat{g}\in I(S_2S_1\inv),\; \hat{\chi}\in\CE(S_2)\; \textrm{all 1-bounded}\}$$
(which is clearly $S_1, S_2$-invariant as well as closed under multiplication and complex conjugation),
we have
\begin{align}\label{E: after dual-difference}
    \liminf_{H\to\infty}\E_{h, h'\in[H]^{s_1+s_2-3}}\limsup_{N\to\infty}\abs{\E_{n\in[N]}\int \hat{f}_{0,h,h'}\cdot S_1^{a(n)}\hat{f}_{1,h,h'}\cdot S_2^{a(n)}\bigbrac{\hat g_{h,h'}\cdot \hat\chi_{h,h'}}\; d\mu}>0,
\end{align}
where for all $j=0,1$ and $h,h'$, we set
\begin{align*}
\hat{f}_{j,h,h'} &:=\Delta_{(S_2S_1\inv)^{\times (s_1-1)}, S_2^{\times (s_2-2)};h-h'}\hat{f}_j,
\end{align*}
while $\hat\chi_{h,h'}\in \CE(S_2)$ and $\hat g_{h,h'}\in I(S_2 S_1\inv)$ are some 1-bounded functions.

\smallskip
\textbf{Step 6: Invoking Proposition \ref{P: base case}.}
While not immediate, it is possible to use  Proposition \ref{P: base case} to   \eqref{E: after dual-difference}. Indeed, after regrouping the terms using the $S_2S_1\inv$-invariance of $\hat{g}_{h,h'}$, we get
\begin{align*}
        \liminf_{H\to\infty}\E_{h, h'\in[H]^{s_1+s_2-3}}\limsup_{N\to\infty}\abs{\E_{n\in[N]}\int \hat{f}_{0,h,h'}\cdot S_1^{a(n)}\bigbrac{\hat{f}_{1,h,h'}\cdot \hat{g}_{h,h'}}\cdot S_2^{a(n)}\hat{\chi}_{h,h'}\; d\mu}>0.
\end{align*}
Just like in the baby case above, the fact that $\hat\chi_{h,h'}$ is a nonergodic eigenfunction of $S_2$ allows us to apply Proposition \ref{P: base case} to the family of averages over $n$, yielding
\begin{align*}
        \liminf_{H\to\infty}\E_{h, h'\in[H]^{s_1+s_2-3}}\lim_{N\to\infty}\abs{\E_{n\in[N]}\int \hat{f}_{0,h,h'}\cdot S_1^n\bigbrac{\hat{f}_{1,h,h'}\cdot \hat{g}_{h,h'}}\cdot S_2^n\hat{\chi}_{h,h'}\; d\mu}>0,
\end{align*}
By  the $\ell=2$ case of Proposition \ref{P: linear}, we then have
\begin{align}\label{E: average of box norms}
    \liminf_{H\to\infty}\E_{h,h'\in[H]^{s_1+s_2-3}}\nnorm{\hat{f}_{1,h,h'}\cdot \hat{g}_{h,h'}}_{S_1, S_1S_2\inv} > 0.
\end{align}

\smallskip
\textbf{Step 7: Auxiliary seminorm control.}
Since $\hat{f}_{1,h,h'}$ is a multiplicative derivative of $\hat{f}_1$, we would like to apply the inductive formula for box seminorms to control \eqref{E: average of box norms} by a seminorm of $\hat{f}_1$. To make this work, we need to remove the superfluous term $\hat g_{h,h'}$'s first, which can be achieved with the help of its invariance property. Indeed, for every $h,h'$, we can write
\begin{align*}
    \nnorm{\hat{f}_{1,h,h'}\cdot \hat g_{h,h'}}_{S_1,S_1S_2\inv}^4 = \lim_{H\to\infty}\E_{h''\in[H]}\norm{\E(\Delta_{S_1; h''}\hat{f}_{1,h,h'} \cdot \Delta_{S_1; h''}\hat g_{h,h'}|\CI(S_1S_2\inv))}_{L^2(\nu)}^2.
\end{align*}
 Since $\hat g_{h,h'}$ are 1-bounded and $S_1S_2\inv$-invariant, we have
\begin{align*}
    \norm{\E(\Delta_{S_1; h''}\hat{f}_{1,h,h'} \cdot \Delta_{S_1; h''}\hat g_{h,h'}|\CI(S_1S_2\inv))}_{L^2(\nu)}\leq \norm{\E(\Delta_{S_1; h''}\hat{f}_{1,h,h'}|\CI(S_1S_2\inv))}_{L^2(\nu)},
\end{align*}
implying that
\begin{align*}
    \nnorm{\hat{f}_{1,h,h'}\cdot \hat g_{h,h'}}_{S_1,S_1S_2\inv}\leq \nnorm{\hat{f}_{1,h,h'}}_{S_1,S_1S_2\inv}.
\end{align*}
In conjunction with \eqref{E: average of box norms}, this bound gives
\begin{align*}
    \liminf_{H\to\infty}\E_{h,h'\in[H]^{s_1+s_2-3}}\nnorm{\hat{f}_{1,h,h'}}_{S_1, S_1S_2\inv} > 0.
\end{align*}
By \cite[Lemma 5.2]{FrKu22a}, the H\"older inequality, and the inductive formula for box seminorms, we get
\begin{align*}
    \nnorm{\hat{f}_1}_{S_1, (S_1S_2\inv)^{\times s_1}, S_2^{\times (s_2-2)}} > 0.
\end{align*}
Projecting down to $X$ gives
\begin{align}\label{E: auxiliary estimate}
    \nnorm{{f}_1}_{T_1, (T_1T_2\inv)^{\times s_1}, T_2^{\times (s_2-2)}} > 0.
\end{align}

\smallskip
\textbf{Step 8: Seminorm smoothing--passing to the magic extension.}
The outcome of our maneuvers so far can be summarized as follows: from controlling our average by the seminorm
$$\nnorm{f_2}_{(T_2T_1\inv)^{\times s_1}, T_2^{\times s_2}}$$ of degree $s_1+s_2$, we have inferred control by the seminorm $$\nnorm{f_1}_{T_1, (T_1T_2\inv)^{\times s_1}, T_2^{\times (s_2-2)}}$$ of degree $s_1+s_2-1$. The caveat is that the transformations appearing in this new seminorm are not of the desired form: specifically, the presence of $s_2-2$ copies of $T_2$ poses an issue, as no reasonable reparametrization of the original average makes $T_2$ act on $f_1$. Hence, we cannot simply invoke symmetry. That being said, we can use this auxiliary control to obtain control of our average by $$\nnorm{f_2}_{(T_2T_1\inv)^{\times s_1}, T_2^{\times (s_2-1)}}.$$ A lot of the maneuvers involved in this process will be similar to those already performed, therefore we will abbreviate some of the discussion.

    Let $(X^*, \CX^*, \mu^*, T_1^*, T_2^*)$ be a magic extension of the system $(X, \CX, \mu, T_1, T_2)$ with respect to the transformations $T_1, T_1T_2\inv$, meaning that
\begin{align*}
    \CZ^*:=\CZ(T_1^*, T_1^*{T_2^*}\inv) = \CI(T_1^*)\vee \CI(T_1^* {T_2^*}\inv)
\end{align*}
(it exists by Proposition \ref{P: magic extensions exist}).
We lift $f_0, f_1, f_2$ to $f^*_0, f^*_1, f^*_2\in L^\infty(\mu^*)$. Then
\begin{align}\label{E: positivity in magic}
        \limsup_{N\to\infty}\abs{\E_{n\in[N]} \int  f^*_0\cdot (T_1^*)^{a(n)}f^*_1\cdot (T_2^*)^{a(n)}f^*_2\; d\mu^*} > 0.
\end{align}
Arguing as in Step 2, we invoke Proposition \ref{P: stashing} to find an increasing integer sequence $(N_l)_l$ such that on defining
\begin{align*}
    F_1^* := \lim_{l\to\infty}\E_{n\in[N_l]}(T_1^*)^{-a(n)}\overline{f_0^*}\cdot (T_2^*{T_1^*}\inv)^{a(n)}\overline{f^*_2}
\end{align*}
as a weak limit much like $F_2$ before, we deduce from the auxiliary estimate \eqref{E: auxiliary estimate} that
\begin{align}\label{E: auxiliary for F_1}
    \nnorm{F^*_1}_{T^*_1, (T^*_1{T^*_2}\inv)^{\times s_1}, (T^*_2)^{\times (s_2-2)}} > 0.
\end{align}

\smallskip
\textbf{Interlude: Transferring control to $f_2$ in the baby case $s_1 = 1$, $s_2 = 2$.}
As we approach the end of the proof, the forthcoming maneuvers become once again significantly simpler whenever $s_1 = 1$, $s_2 = 2$. We therefore present the conclusion of the argument in this baby case first before delving into the general case. For $s_1 = 1$, $s_2 = 2$, the auxiliary seminorm control \eqref{E: auxiliary for F_1} reduces to
\begin{align*}
    \nnorm{F^*_1}_{T^*_1, T^*_1{T^*_2}\inv} > 0.
\end{align*}
Proposition \ref{P: inverse theorem magic} then gives us $g_1^*\in I(T_1^*)$ and $g_2^*\in I(T_1^*{T_2^*}\inv)$ for which
\begin{align*}
    \int F^*_1 \cdot \overline{g_1^*}\cdot\overline{g_2^*}\; d\mu^*>0.
\end{align*}
Expanding the definition of $F^*_1$, we get
\begin{align*}
    \limsup_{N\to\infty}\abs{\E_{n\in[N]} \int f^*_{0}\cdot (T_1^*)^{a(n)}\bigbrac{g_{1}^*\cdot g_{2}^*} \cdot (T^*_2)^{a(n)}f^*_{2}\; d\mu^*}>0.
\end{align*}
The invariance properties of $g^*_1, g^*_2$ allow us to rearrange terms so that
\begin{align*}
    \limsup_{N\to\infty}\abs{\E_{n\in[N]} \int (f^*_{0}\cdot g_1^*)\cdot (T^*_2)^{a(n)}(f^*_{2}\cdot g_2^*)\; d\mu^*}>0.
\end{align*}
By applying Proposition \ref{P: base case} in the degenerate case when the eigenfunction of $T^*_1$ is 1, we get that
\begin{align*}
    \limsup_{N\to\infty}\abs{\E_{n\in[N]} \int (f^*_{0}\cdot g_1^*)\cdot (T^*_2)^n(f^*_{2}\cdot g_2^*)\; d\mu^*}>0,
\end{align*}
and hence $\nnorm{f^*_{2}\cdot g_2^*}_{T_2^*}>0$ by the mean ergodic theorem. Once again, a simple argument relying on the $T_2^*{T_1^*}\inv$-invariance of $g_2^*$ gives
\begin{align*}
    \nnorm{f^*_{2}}_{T_2^*{T_1^*}\inv, T_2^*}>0.
\end{align*}
Projecting down to $X$ establishes Claim \ref{C: degree lowering} for $s_1 = 1$, $s_2 = 2$.

\smallskip
\textbf{Step 9: Seminorm smoothing--transferring control to $f_2$.}
When $s_1 + s_2\geq 4$, the maneuvers above become more complicated. Once again, we need to use the inductive formula for box seminorms and perform the dual-difference interchange in order to reduce to a degree-2 seminorm amenable to an application of \cref{P: inverse theorem magic}. Mimicking Step 4, we use the inductive formula for box seminorms to express
\begin{align*}
    \nnorm{F^*_1}_{T^*_1, (T^*_1{T^*_2}\inv)^{\times s_1}, (T^*_2)^{\times (s_2-2)}}^{2^{s_1+s_2-1}} = \lim_{H\to\infty}\E_{h\in[H]^{s_1+s_2-3}}\nnorm{F^*_{1,h}}_{T_1^*,{T_1^*T_2^*}\inv}^{2^{s_1+s_2-3}}
\end{align*}
for
\begin{align*}
        F^*_{1,h} &:= \Delta_{(T_2 T_1\inv)^{\times (s_1-1)}, T_2^{\times (s_2-2)};h}F^*_1.
\end{align*}
Then we use Proposition \ref{P: inverse theorem magic} to conclude that
\begin{align*}
    \liminf_{H\to\infty}\E_{h\in[H]^{s_1+s_2-3}}\int F^*_{1,h}\cdot \overline {g_{1,h}^*}\cdot \overline{g_{2,h}^*}\; d\mu^* > 0
\end{align*}
for some 1-bounded functions $g_{1,h}^*\in I(T_1^*)$ and $g_{2,h}^*\in I(T_1^*{T_2^*}\inv)$.

By Proposition \ref{P: dual-difference} (the dual-difference interchange) applied with $f:=F^*_1$ and
\begin{align*}
    \CG := \{g_1^*\cdot g_2^*\colon g_1^*\in I(T_1^*),\; g_2^*\in I(T_1^*{T_2^*}\inv)\; \textrm{all 1-bounded}\},
\end{align*}
we have
\begin{multline*}
    \liminf_{H\to\infty}\E_{h,h'\in[H]^{s_1+s_2-3}} \\ \limsup_{N\to\infty}\abs{\E_{n\in[N]} \int f^*_{0,h,h'}\cdot (T_1^*)^{a(n)}\bigbrac{g_{1,h,h'}^*\cdot g_{2,h,h'}^*} \cdot (T^*_2)^{a(n)}f^*_{2,h,h'}\; d\mu^*}>0,
\end{multline*}
where for all $h,h'\in \N^{s_1+s_2-3}$ and $j=0,2$, we set
\begin{align*}
f^*_{j,h,h'} &:=\Delta_{(T^*_1{T^*_2}\inv)^{\times (s_1-1)}, (T^*_2)^{\times (s_2-2)} ; h-h'}f^*_j,
\end{align*}
while $g_{1,h,h'}, g_{2,h,h'}$ are some 1-bounded elements of $I(T_1^*)$, $I(T_1^*{T_2^*}\inv)$ respectively. Using the invariance properties of $g^*_{j,h,h'}$'s and rearranging, we get
\begin{align*}
    \liminf_{H\to\infty}\E_{h,h'\in[H]^{s_1+s_2-3}}\limsup_{N\to\infty}\abs{\E_{n\in[N]} \int \bigbrac{f^*_{0,h,h'}\cdot g_{1,h,h'}^*}\cdot (T^*_2)^{a(n)}\bigbrac{f^*_{2,h,h'}\cdot g_{2,h,h'}^*}\; d\mu^*}>0.
\end{align*}
From Proposition \ref{P: base case} we infer that
\begin{align*}
    \liminf_{H\to\infty}\E_{h,h'\in[H]^{s_1+s_2-3}}\limsup_{N\to\infty}\abs{\E_{n\in[N]} \int \bigbrac{f^*_{0,h,h'}\cdot g_{1,h,h'}^*}\cdot (T^*_2)^n\bigbrac{f^*_{2,h,h'}\cdot g_{2,h,h'}^*}\; d\mu^*}>0.
\end{align*}
The mean ergodic theorem then yields
\begin{align*}
    \liminf_{H\to\infty}\E_{h,h'\in[H]^{s_1+s_2-3}}\nnorm{f^*_{2,h,h'}\cdot g_{2,h,h'}^*}_{T_2^*}>0.
\end{align*}
The monotonicity of box seminorms and the $T_1^*{T_2^*}\inv$-invariance of $g_{2,h,h'}$ allows us to bound
\begin{align*}
    \nnorm{f^*_{2,h,h'}\cdot g_{2,h,h'}^*}_{T_2^*}\leq \nnorm{f^*_{2,h,h'}\cdot g_{2,h,h'}^*}_{T_2^*{T_1^*}\inv, T_2^*}\leq \nnorm{f^*_{2,h,h'}}_{T_2^*{T_1^*}\inv, T_2^*}
\end{align*}
analogously as in Step 7. Hence,
\begin{align*}
    \liminf_{H\to\infty}\E_{h,h'\in[H]^{s_1+s_2-3}}\nnorm{f^*_{2,h,h'}}_{T_2^*{T_1^*}\inv, T_2^*}>0.
\end{align*}
Once again, \cite[Lemma 5.2]{FrKu22a}, the H\"older inequality, and the inductive formula for box norms, imply that
\begin{align*}
    \nnorm{f_2^*}_{(T^*_1{T^*_2}\inv)^{\times s_1}, (T^*_2)^{\times (s_2-1)}} >0.
\end{align*}
Claim \ref{C: degree lowering} follows on projecting this down to $X$.
\end{proof}

By adapting the proof of Proposition \ref{P: degree lowering} in obvious ways, i.e., carrying over the double averaging scheme and replacing the use of Proposition \ref{P: linear} with Proposition \ref{P: linear 2}, we establish the following inductive step in the proof of Theorem \ref{T: limiting formula 2}.
\begin{proposition}[Degree lowering, II]
    Let $a\colon \N\to \Z$ be a sequence and $n_k\in \N_0$ for $k\in \N$ so that $(a(k!n+n_k))_{n,k}$ admits box seminorm control and equidistributes on nilsystems. Let $s_1, s_2\in\N$.
    Suppose that for every system $(X, \CX, \mu, T_1, T_2)$ and all functions $f_0, f_1, f_2\in L^\infty(\mu)$,   we have \eqref{E:vanish'}
        whenever $\nnorm{f_2}_{(T_2T_1\inv)^{\times s_1}, T_2^{\times s_2}} = 0$.
  Then for every system $(X, \CX, \mu, T_1, T_2)$ and all functions $f_0, f_1, f_2\in L^\infty(\mu)$, we have \eqref{E:vanish'} whenever $\nnorm{f_2}_{T_2T_1\inv, T_2} = 0$.
\end{proposition}
\subsection{Concluding the proofs}\label{SS:end}
We start with proving the following qualitative version of Theorem \ref{T: optimal control}.
\begin{proposition}[Qualitative optimal seminorm control]\label{P: qualitative optimal control}
	Let $a\colon \N\to \Z$ be a sequence that admits box seminorm control and equidistributes on nilsystems. Then for every system $(X, \CX, \mu, T_1, T_2)$ and all functions $f_0,f_1, f_2\in L^\infty(\mu)$, we have \eqref{E:vanish}
    whenever $\nnorm{f_2}_{T_2 T_1\inv, T_2} = 0$.
\end{proposition}

\begin{proof}
    Since $a$ admits box seminorm control, we can find $s\in\N$ such that for every system $(X, \CX, \mu, T_1, T_2)$ and all $f_0, f_1, f_2\in L^\infty(\mu)$, we have \eqref{E:vanish} whenever $\nnorm{f_2}_{(T_2T_1\inv)^{\times s}, T_2^{\times s}}=0$. The claim then follows from Proposition \ref{P: degree lowering}.
\end{proof}
Next, we deduce Theorem \ref{T: limiting formula}.
\begin{proof}[Proof of Theorem \ref{T: limiting formula}]
    Let $(X^*, \CX^*, \mu^*, T_1^*, T_2^*)$ be a magic extension of $(X, \CX, \mu, T_1, T_2)$ (provided by  Proposition \ref{P: magic extensions exist} for $\ell=2$) with respect to the transformations $T_2T_1\inv, T_2$, meaning that
\begin{align*}
    \CZ(T_2^*{T_1^*}\inv, T_2^*) = \CI(T_2^* {T_1^*}\inv)\vee \CI(T_2^*).
\end{align*}

We lift $f_0, f_1, f_2$ to $f^*_0, f^*_1, f^*_2\in L^\infty(\mu^*)$. It suffices to prove that
\begin{align}\label{E: claimed identity final}
    \lim_{N\to\infty}\E_{n\in[N]} \int f_0^*\cdot (T_1^*)^{a(n)}f_1^*\cdot (T_2^*)^{a(n)}f_2^*\; d\mu^* = \lim_{N\to\infty}\E_{n\in[N]} \int f_0^*\cdot (T_1^*)^n f_1^*\cdot (T_2^*)^n f_2^*\; d\mu^*.
\end{align}
By Proposition \ref{P: qualitative optimal control}, the structure of $\CZ(T_2^*{T_1^*}\inv, T_2^*)$, and an $L^2(\mu^*)$ approximation argument, we can assume that  $f_2^* = g_1^*g_2^*$ for $g_1^*\in I(T_2^*{T_1^*}\inv)$ and $g_2^*\in I(T_2^*)$.
Then \eqref{E: claimed identity final} reduces to
\begin{align*}
        \lim_{N\to\infty}\E_{n\in[N]} \int (f_0^*\cdot g_2^*)\cdot (T_1^*)^{a(n)}(f_1^*\cdot g_1^*)\; d\mu^* = \lim_{N\to\infty}\E_{n\in[N]} \int (f_0^*\cdot g_2^*)\cdot (T_1^*)^n(f_1^*\cdot g_1^*)\; d\mu^*.
\end{align*}
This in turn follows from Proposition \ref{P: base case}. {Alternatively, this can be inferred directly from the spectral theorem and the equidistribution of $(a(n))_n$ for rotations on $\mathbb{T}$.}
\end{proof}
Theorem \ref{T: limiting formula} and Proposition \ref{P: linear} immediately give Theorem \ref{T: optimal control}.
\begin{proof}[Proof of Theorem \ref{T: optimal control}]
    By Theorem \ref{T: limiting formula} and Proposition \ref{P: linear}, we have
    \begin{multline*}
        \lim_{N\to\infty}\abs{\E_{n\in[N]} \int f_0\cdot T_1^{a(n)}f_1\cdot T_2^{a(n)}f_2\; d\mu}
        = \lim_{N\to\infty}\abs{\E_{n\in[N]} \int f_0\cdot T_1^n f_1\cdot T_2^n f_2\; d\mu}\\
        \leq \nnorm{f_2}_{T_2 T_1\inv, T_2}.
    \end{multline*}
    The bound by $\nnorm{f_1}_{T_1, T_2 T_1\inv}$ follows by symmetry while the bound by $\nnorm{f_0}_{T_1, T_2}$ can be deduced analogously after composing the original integral with $T_2^{-a(n)}$.
\end{proof}

Analogously to Proposition \ref{P: qualitative optimal control}, we can prove the following.
\begin{proposition}[Qualitative optimal seminorm control, II]\label{P: qualitative optimal control 2}
	Let $a\colon \N\to \Z$ be a sequence and $n_k\in \N_0$ for $k\in \N$ so that $(a(k!n+n_k))_{n,k}$ admits box seminorm control and equidistributes on nilsystems. Then for every system $(X, \CX, \mu, T_1, T_2)$ and all functions $f_0,f_1, f_2\in L^\infty(\mu)$, we have \eqref{E:vanish'}
    whenever $\nnorm{f_2}_{T_2 T_1\inv, T_2} = 0$.
\end{proposition}
Proposition \ref{P: qualitative optimal control 2} then implies Theorem \ref{T: limiting formula 2} just like Proposition \ref{P: qualitative optimal control} implies Theorem \ref{T: limiting formula}.


\begin{thebibliography}{9999}

 
	\bibitem{AS25}
D.~Altman, M.~Sawhney.
On polynomial progressions via transference.
Preprint (2025), \texttt{arXiv:2506.13010}.



\bibitem{Au09} T.~Austin. On the norm convergence of nonconventional ergodic averages.
	 	\emph{Ergodic Theory Dynam. Systems} {\bf 30} (2010), 321--338.

\bibitem{Au15a} T.~Austin.  Pleasant extensions retaining algebraic structure, I. \emph{J. Anal. Math.} \textbf{125} (2015), 1--36.

\bibitem{Au15b} T.~Austin.  Pleasant extensions retaining algebraic structure, II. \emph{J. Anal. Math.} \textbf{126} (2015), 1--111.

\bibitem{Bere}
D.~Berend, Y.~Bilu. Polynomials with roots modulo every integer.
{\em Proc. Amer. Math. Soc.}, \textbf{124}, no 6, (1996),
1663--1671.

	 \bibitem{BHK05} V.~Bergelson, B.~Host, B.~Kra, with an appendix by I. Ruzsa.
	 	Multiple recurrence and nilsequences.
	 	{\em Invent. Math.} {\bf 160} (2005), no. 2, 261--303.


	


    	\bibitem{BL96}
	V.~Bergelson, A.~Leibman.
	Polynomial extensions of van der Waerden's and Szemer\'edi's theorems.
	\emph{J. Amer. Math. Soc.} \textbf{9} (1996), 725--753.


	\bibitem{BL15}
V.~Bergelson, A.~Leibman.
Cubic averages and large intersections.
In \emph{Contemporary Mathematics} \textbf{631} (2015), 5--19.
	
	\bibitem{BLL08}
	V.~Bergelson, A.~Leibman, E.~Lesigne.
	Intersective polynomials and polynomial Szemer\'edi theorem.
	\emph{Adv. Math.} \textbf{219} (2008), no.~1, 369--388.


	
	
	\bibitem{BeMc96} V.~Bergelson, R.~McCutcheon.
	Uniformity in polynomial Szemer\'edi theorem,  Ergodic Theory of $\Z^d$-actions (edited by M. Pollicott and K. Schmidt). {\em London Math. Soc. Lecture Note Series}
	\textbf{228} (1996),  273--296.

      \bibitem{BMR20} V.~Bergelson, J.~Moreira,
  F.~Richter. Single and multiple recurrence along nonpolynomial sequences. {\em Adv. Math.} {\bf 368} (2020), 107--146.

        \bibitem{Ber21} A.~Berger. Popular differences for corners in Abelian groups. \emph{Math. Proc. Cambridge Philos. Soc.} \textbf{171} (2021), no. 1, 207--225.
    
    \bibitem{Bos94} M.~Boshernitzan. Uniform distribution and Hardy
    fields.  {\em J. Analyse Math.} \textbf{62} (1994), 225--240.
    
	\bibitem{Chu11}
	Q.~Chu.
	Multiple recurrence for two commuting transformations.
	\emph{Ergodic Theory Dynam. Systems} \textbf{31} (2011), 771--792.
	
	\bibitem{CFH11}
	Q.~Chu, N.~Frantzikinakis, B.~Host.
	Ergodic averages of commuting transformations with distinct degree polynomial iterates.
	\emph{Proc. Lond. Math. Soc.} \textbf{102} (2011), 801--842.
	
	
	\bibitem{CL84} J-P.~Conze, E.~Lesigne.
	Th\'eor\`emes ergodiques pour des mesures diagonales.
	{\em Bull.
		Soc. Math. France} {\bf 112}  (1984), no. 2,  143--17
	
		\bibitem{Da25} L.~Daskalakis.
	Ergodic theorems for bilinear averages, Roth's theorem and corners along fractional powers.
	Preprint (2025), \texttt{arXiv:2504.18307}.

		\bibitem{DFKS24} S.~Donoso, A.~Ferr\'e Moragues, A.~Koutsogiannis, W.~Sun.
	Decomposition of multicorrelation sequences and joint ergodicity.
	{\em Ergodic
		Theory Dynam. Systems}  \textbf{44} (2024), 432--480.
	
	\bibitem{DKKST24}
	S.~Donoso, A.~Koutsogiannis, B.~Kuca, W.~Sun, K.~Tsinas.
	Seminorm estimates and joint ergodicity for pairwise independent Hardy sequences.
	Preprint (2024), \texttt{arXiv:2410.15130}.
	
	\bibitem{DKS23}
	S.~Donoso, A.~Koutsogiannis, W.~Sun.
	Joint ergodicity for functions of polynomial growth.
	 \emph{Israel J. Math.}
	{\bf 268} (2025), 315--363.

        \bibitem{DS18}
        S.~Donoso, W.~Sun. Quantitative multiple recurrence for two and three transformations. \emph{Israel J. Math.} \textbf{226} (2018), 71--85.

	
	\bibitem{EFHN15}
	T.~Eisner, B.~Farkas, M.~Haase, R.~Nagel.
	\emph{Operator theoretic aspects of ergodic theory}, Vol.~272.
	Springer, 2015.
	
		 \bibitem{F21} A. Ferr\'e Moragues. Properties of multicorrelation sequences and large returns under some ergodicity
	assumptions. {\em Discrete Contin. Dyn. Syst.} {\bf 41}  (2021), 2809--2828.
	

        \bibitem{FSSSZ20} J.~Fox, A.~Sah, M.~Sawhney, D.~Stoner,  Y.~Zhao. Triforce and corners. {\em Math. Proc. Cambridge. Philos. Soc.} \textbf{169} (2020), no. 1, 209--223.


	
\bibitem{Fr08} N.~Frantzikinakis.  Multiple ergodic averages for
three polynomials and applications.   \emph{ Trans. Amer.
	Math. Soc.}  \textbf{360} (2008), no. 10, 5435--5475.
	
	\bibitem{Fr09} N.~Frantzikinakis. Equidistribution of sparse sequences on nilmanifolds.
	\emph{J. Analyse Math.} {\bf 109} (2009), 353--395.

    
\bibitem{Fr10}
N.~Frantzikinakis.
Multiple recurrence and convergence for Hardy sequences of polynomial growth.
\emph{J. Anal. Math.} \textbf{112} (2010), no.~1, 79--135.


	\bibitem{Fr15}
N.~Frantzikinakis.
Multiple correlation sequences and nilsequences. {\em Invent. Math.} {\bf 202} (2015), no. 2, 875--892.





	\bibitem{Fr16} N.~Frantzikinakis. Some open problems on multiple ergodic averages. {\em Bull. Hellenic Math. Soc.} \textbf{60}
	(2016),  41--90.


    	\bibitem{Fr21}
N.~Frantzikinakis.
Joint ergodicity of sequences.
\emph{Adv. Math.} \textbf{417} (2023), 108918.
	

	
	\bibitem{FH18}
	N.~Frantzikinakis, B.~Host.
	Weighted multiple ergodic averages and correlation sequences.
	\emph{Ergodic Theory Dynam. Systems} \textbf{38} (2018), no.~1, 81--142.
	
	\bibitem{FrHK11} N.~Frantzikinakis, B.~Host, B.~Kra. The polynomial multidimensional Szemer\'edi Theorem along shifted primes.
	{\em Isr. J. Math.} {\bf 194} (2013), 331--348.
	
	
		\bibitem{FrKu22b}
	N.~Frantzikinakis, B.~Kuca.
	Seminorm control for ergodic averages with commuting transformations
	along pairwise dependent polynomials.
	\emph{Ergodic Theory Dynam. Systems} \textbf{43} (2023), no.~12, 4074--4137.
	
	\bibitem{FrKu22c}
	N.~Frantzikinakis, B.~Kuca.
	Degree lowering for ergodic averages along arithmetic progressions.
	\emph{J. Anal. Math.}  \textbf{154}  (2024), 199--253.
	
	\bibitem{FrKu22a}
	N.~Frantzikinakis, B.~Kuca.
	Joint ergodicity for commuting transformations and applications to polynomial sequences.
	\emph{Invent. Math.} \textbf{239} (2025), no.~2, 621--706.




\bibitem{Fu77} H.~Furstenberg.
Ergodic behavior of diagonal measures and a theorem of Szemer\'edi
on arithmetic progressions. {\em J. Anal. Math.}  \textbf{71}
(1977), 204--256.


\bibitem{Fu81} H.~Furstenberg.
Recurrence in ergodic theory and combinatorial number theory. {\em
	Princeton University Press}, Princeton, (1981).

        \bibitem{FuK79} H.~Furstenberg, Y.~Katznelson.
        An ergodic Szemer\'edi theorem for commuting transformations.
        {\em J. Analyse Math.} \textbf{34} (1979), 275--291.

	\bibitem{FW96}  H.~Furstenberg, B.~Weiss.
	A mean ergodic theorem for $(1/N)\sum\sp N\sb {n=1}$
	$f(T\sp nx)$ $\ g(T\sp {n\sp 2}x)$. Convergence in ergodic theory
	and probability (Columbus, OH, 1993),
	Ohio State Univ. Math. Res. Inst. Publ., {\bf 5}, de Gruyter, Berlin,
	(1996), 193--227.
	
	
	\bibitem{GT08} B.~Green, T.~Tao.
	The primes contain arbitrarily long arithmetic progressions.  \emph{Annals Math.}
	\textbf{167} (2008), 481--547.

\bibitem{GT09b} B.~Green, T.~Tao. Linear equations in primes.
{\em  Annals  Math.} \textbf{171} (2010), 1753--1850.

	\bibitem{GT10a}
	B.~Green, T.~Tao.
	An arithmetic regularity lemma, an associated counting lemma, and applications.
	\emph{Bolyai Soc. Math. Stud.} \textbf{21} (2010), 261--334.


    \bibitem{Hardy12}
    G.~Hardy. Properties of logarithmico-exponential functions. \emph{Proc. Lond. Math. Soc.} (1912), 54--90.
	
	\bibitem{H09}
	B.~Host.
	Ergodic seminorms for commuting transformations and applications.
	\emph{Studia Math.} \textbf{195} (2009), 31--49.


  	\bibitem{HK05a}
B.~Host, B.~Kra.
Nonconventional ergodic averages and nilmanifolds.
\emph{Ann. of Math.} \textbf{161} (2005), no.~1, 397--488.
	
	\bibitem{HK18}
B.~Host, B.~Kra.
{\em Nilpotent Structures in Ergodic Theory.}
Mathematical Surveys and Monographs, vol. 236. American Mathematical Society, Providence, RI, 2018.


  
	
	
	\bibitem{Kou18}	 A.~Koutsogiannis. Integer part polynomial correlation sequences. 
	{\em Ergodic Theory Dynam. Systems} {\bf 38}  (2018),
	1525--1542.
	
	
	\bibitem{KLMR21}	 A.~Koutsogiannis, A.~Le, J.~Moreira, F.~Richter. Structure of multicorrelation sequences with integer
	part polynomial iterates along the primes. {\em Proc. Amer. Math. Soc.} {\bf 149} (2021), 209--216.
	
		\bibitem{KT25} A.~Koutsogiannis, K.~Tsinas.	
	Ergodic averages for sparse sequences along primes. To appear in
	{\em J. Mod. Dyn.}, \texttt{arXiv:2309.04939}.


	\bibitem{KKL24a}
N.~Kravitz, B.~Kuca, J.~Leng.
Quantitative concatenation for polynomial box norms.
Preprint (2024), \texttt{arXiv:2407.08636}.


	\bibitem{KKL24b}
	N.~Kravitz, B.~Kuca, J.~Leng.
	Corners with polynomial side length.
	Preprint (2024), \texttt{arXiv:2407.08637}.
	

	
	\bibitem{Le18} A.~Le.  Nilsequences and multiple correlations along subsequences. 
	{\em  Ergodic Theory Dynam. Systems} {\bf  40}
	(2018), 1634--1654.
	
\bibitem{LME21}	A.~Le, J.~Moreira,  F.~Richter. A decomposition of multicorrelation sequences for commuting transformations along primes. {\em Discrete Analysis} (2021), 1--27.
	
	
	\bibitem{Lei05a} A.~Leibman.
	Pointwise convergence of ergodic averages for polynomial
	sequences of rotations of a nilmanifold. {\em Ergodic Theory Dynam.
		Systems} {\bf 25}  (2005), no. 1,   201--213.

	 \bibitem{Lei15} A.~Leibman. Nilsequences, null-sequences, and multiple correlation sequences. {\em Ergod. Theory  Dynam. Systems}, 
{\bf 35} (2015), 176--191.

	
	\bibitem{Leng25}
	J.~Leng.
	Structured extensions and multi-correlation sequences.
	Preprint (2025), \texttt{arXiv:2504.07038}.
	
	

        \bibitem{Ma21} M.~Mandache. A variant of the corners theorem. {\em Math. Proc. Cambridge. Philos. Soc.} \textbf{171} 2021, no. 3, 607--621.

\bibitem{MSTT22} K.~Matom\"aki, X.~Shao, T.~Tao,  J.~Ter\"av\"ainen. Higher uniformity of arithmetic functions in
short intervals I - All intervals.
{\em  Forum Math., Pi} {\bf 11} (2023), e29,  97pp.

 


     
        \bibitem{P19a} S.~Peluse. On the polynomial Szemer\'edi theorem in finite fields. {\em Duke
		Math. J.} \textbf{168} (2019), 749--774.


\bibitem{P19b} S.~Peluse. Bounds for sets with no polynomial progressions.  {\em  Forum Math., Pi} {\bf 8} (2020), e16, 55pp.
		
	\bibitem{PP19} S.~Peluse, S.~Prendiville. Quantitative bounds in the nonlinear Roth theorem.   	{\em Invent.  Math.}  \textbf{238} (2024), no. 3, 865--903.

        
        	\bibitem{PSS23}
        S.~Peluse, A.~Sah, M.~Sawhney.
        Effective bounds for Roth's theorem with shifted square common difference.
        Preprint (2023), \texttt{arXiv:2309.08359}.

   \bibitem{SSZ21} A.~Sah, M.~Sawhney, Y.~Zhao. Patterns without a popular difference. {\em Discrete Anal.}  {\bf 8} (2021), 30 pp.
	
	\bibitem{Sz75} E.~Szemer\'edi.
	On sets of integers containing no $k$ elements in arithmetic
	progression. {\em Acta Arith.}  {\bf 27}  (1975), 299--345.
	


    	 \bibitem{Ta08} T.~Tao.
	Norm convergence of multiple ergodic averages for commuting transformations.
	{\em Ergodic Theory Dynam. Systems} \textbf{28} (2008), 
	
	
		\bibitem{Tao15}
	T.~Tao.
	Deducing a weak ergodic inverse theorem from a combinatorial inverse theorem.
	Blog post (2015), \url{https://terrytao.wordpress.com/2015/07/23/deducing-a-weak-ergodic-inverse-theorem-from-a-combinatorial-inverse-theorem/}.no. 2,  657--688.
	
	
	\bibitem{TT19} T.~Tao, J.~Ter\"av\"ainen. The structure of logarithmically averaged correlations of multiplicative functions, with applications to the Chowla and Elliott conjectures. {\em  Duke Math. J.}  {\bf 168} (2019), 1977--2027.
	
	\bibitem{TZ16}
	T.~Tao, T.~Ziegler.
	Concatenation theorems for anti-Gowers uniform functions
	and Host-Kra characteristic factors.
	\emph{Discrete Anal.} \textbf{13} (2016), 61pp.
	
	
	
	
		\bibitem{Wal12}
	M.~Walsh. Norm convergence of nilpotent ergodic averages. {\em Ann. of Math.} {\bf 175} (2012), no. 3, 1667--1688.
	
 \bibitem{Wa82} P.~Walters. An introduction to ergodic theory. {\em Graduate
 	Texts in Mathematics}, \textbf{79}, Springer-Verlag, New
 York-Berlin, (1982).
	
	\bibitem{Zi07} T.~Ziegler.
	Universal characteristic factors and Furstenberg averages. {\em J.
		Amer. Math. Soc.} \textbf{20} (2007), 53--97.
		
	 \bibitem{Zo15a}
	 P.~Zorin-Kranich. Norm convergence of multiple ergodic averages on amenable groups.
	 {\em J. Analyse Math.} {\bf  130} (2016), 219--241.
	 
	
	 
	 
	
\end{thebibliography}
\end{document}